\documentclass[reqno,oneside,11pt]{amsart}
\usepackage[T1]{fontenc}
\usepackage[utf8]{inputenc}
\usepackage{lmodern}
\usepackage[french,english]{babel}
\usepackage{geometry} 
\usepackage{csquotes}

\usepackage{amsmath,amssymb,amsfonts,amsthm,mathtools}
\usepackage{mathrsfs}
\usepackage{graphicx}
\usepackage{caption}
\usepackage{subcaption} 
\usepackage{booktabs}
\usepackage{array}
\usepackage{paralist}
\usepackage{fancyhdr}
\usepackage{emptypage}
\fancypagestyle{plain}{
	\fancyhf{}
	\fancyfoot[RO,RE]{\thepage}
	
	}

\usepackage{comment} 
\usepackage{diagbox}
\usepackage{xcolor}
\usepackage{epigraph}

\usepackage[noadjust]{cite}
\usepackage[
pagebackref=true,
colorlinks=true,
urlcolor=purple,
linkcolor=purple!87!black,
citecolor=green!60!black,
pdfborder={0 0 0}
]{hyperref}
\renewcommand*{\backref}[1]{}
\renewcommand*{\backrefalt}[4]{[{\tiny%
		\ifcase #1 Not cited.%
		\or Cited on page~#2.%
		\else Cited on pages #2.%
		\fi%
	}]}
\usepackage{cleveref}

\usepackage{bookmark}
\usepackage{dsfont}

\newcommand{\Z}{\mathbb{Z}}
\newcommand{\N}{\mathbb{N}}
\newcommand{\R}{\mathbb{R}}
\newcommand{\Q}{\mathbb{Q}}
\renewcommand{\P}{\mathbb{P}}
\newcommand{\E}{\mathbb{E}}

\newcommand{\supp}[1]{\mathrm{supp}(#1)}

\usepackage{tikz}
\usetikzlibrary{patterns,positioning,arrows,decorations.markings,calc,decorations.pathmorphing,decorations.pathreplacing}
\usetikzlibrary{fadings}
\usepackage{forest}
\usepackage{tikz-qtree}
\RequirePackage{pgffor} 
\makeatletter
\newcommand{\neutralize}[1]{\expandafter\let\csname c@#1\endcsname\count@}
\makeatother

\newcommand{\thistheoremname}{}

\newtheorem*{genericthm*}{\thistheoremname}
\newenvironment{namedthm*}[1]
{\renewcommand{\thistheoremname}{#1}%
	\begin{genericthm*}}
	{\end{genericthm*}}
\theoremstyle{plain}
\newtheorem{thm}{Theorem}[section] 

\newtheorem{prop}[thm]{Proposition}
\newtheorem{lem}[thm]{Lemma}

\newtheorem{claim}[thm]{Claim}
\newtheorem{question}[thm]{Question}

\newtheorem{innercustomthm}{Theorem}
\newenvironment{customthm}[1]
  {\renewcommand\theinnercustomthm{#1}\innercustomthm}
  {\endinnercustomthm}

\theoremstyle{definition}
\newtheorem{rem}[thm]{Remark}

\newtheorem{defn}[thm]{Definition} 

\usepackage[notocbasic,refpage]{nomencl}
\makenomenclature

\usepackage{etoolbox}
\newcommand{\getgroupname}[1]{%
  \ifstrequal{#1}{C}{Cellular $\ell^2$-chains}{%
  \ifstrequal{#1}{G}{Group presentations}{%
  \ifstrequal{#1}{O}{Notation used in Section \ref{section:convergence of the spectral measure}}{%
  \ifstrequal{#1}{P}{Partitions}{%
  Other}}}}%
}

\newcounter{nomcount}
\newcommand{\padzeroes}{%
  \ifnum\value{nomcount}<10 000\else
  \ifnum\value{nomcount}<100 00\else
  \ifnum\value{nomcount}<1000 0\fi\fi\fi
  \arabic{nomcount}%
}

\newcommand{\mynomenclature}[3][]{%
  \stepcounter{nomcount}%
  \nomenclature[#1\padzeroes]{#2}{#3}%
}

\usepackage{bm,thmtools}
\usepackage{amscd}

\usetikzlibrary{arrows.meta,decorations.markings}

\newcommand{\freesum}[1]{\mathbf{s}(#1)}
\newcommand{\typemap}{\mathbf{j}}

\newcommand{\rg}{\mathbf{rg}}

\newcommand{\cost}{\mathrm{cost}}
\DeclareMathOperator{\lab}{lab}
\newcommand{\Cay}{\mathrm{Cay}}
\newcommand{\Sch}{\mathrm{Sch}}

\allowdisplaybreaks[1]
\author{Miguel Donoso-Echenique}
\email[Miguel Donoso-Echenique]{mdonosoe@uni-muenster.de, migueldonosoe@gmail.com}
\address{University of Münster, Einsteinstrasse 62, Münster 48149, Germany}
\urladdr{\url{https://sites.google.com/view/miguel-donoso-echenique}}

\author{Eduardo Silva} 
\email[Eduardo Silva]{eduardo.silva@uni-muenster.de, edosilvamuller@gmail.com}
\address{University of Münster, Einsteinstrasse 62, Münster 48149, Germany}

\urladdr{\url{https://edoasd.github.io/eduardo_silva_math/}}

\title[]{The first $\ell^2$-Betti number, cost and rank gradient of few-relator random groups}
\allowdisplaybreaks
\usepackage{url}
\usepackage{enumitem}
\setlist[enumerate, 1]{label = (\roman*), ref = \roman*}
\setlist[enumerate, 2]{label = \theenumi.\alph*}
\newlist{stepsAone}{enumerate}{1}
\setlist[stepsAone,1]{label = \textbf{\underline{Step A.1.\arabic*}}:}

\newlist{stepsAtwo}{enumerate}{1}
\setlist[stepsAtwo,1]{label = \textbf{\underline{Step A.2.\arabic*}}:}

\newlist{stepsB}{enumerate}{1}
\setlist[stepsB,1]{label = \textbf{\underline{Step B.\arabic*}}:}
\begin{document}	
	\begin{abstract}
Let $G$ be a random group in the few-relator model with $n\ge 2$ generators and $m\ge 1$ relators. We prove that asymptotically almost surely the first $\ell^2$-Betti number, the cost minus one, and the absolute rank gradient of $G$ are all equal to $\max\{n-m-1,0\}$. In particular, the equality $\smash{\beta_1^{(2)}(G)=0}$ when $m\ge n-1$ answers positively a conjecture of Kielak, Kropholler and Wilkes \cite{KielakKrophollerWilkes2022}. 
	\end{abstract}
	
	\maketitle
	
\section{Introduction}\label{section: introduction}
The \emph{cost} of a probability-measure-preserving (p.m.p.)\ action of a countable group is an invariant of its orbit equivalence relation. It was introduced by Levitt \cite{levitt1995cost} and subsequently developed by Gaboriau \cite{gaboriau2000cout}. The cost of a countable group $G$, denoted by $\cost(G)$, is the infimum of the costs of its free p.m.p.\ actions. It also has the following probabilistic characterization: $\cost(G)$ is half of the infimum of the expected degree of the identity over all $G$-invariant random connected spanning graphs on $G$; see \cite[Proposition~29.5]{KechrisMiller2004}. The cost of $G$ is closely related to the \emph{first $\ell^2$-Betti number} $\beta_1^{(2)}(G)$. Gaboriau proved that every infinite countable group $G$ satisfies $\beta_1^{(2)}(G)+1\leq \cost(G)$ \cite[Corollaire 3.23]{gaboriau2002}.

Suppose now that $G$ is finitely generated and residually finite, and let $(G_i)_{i\geq 0}$ be a decreasing chain of finite-index normal subgroups with $G_0=G$. The \emph{rank gradient} of the chain measures the growth of the minimal number of generators along the chain relative to the corresponding indices. It was introduced by Lackenby in \cite{Lackenby2005} and is defined by
\[
\rg\left(G,(G_i)_{i\geq 0}\right) \coloneqq \lim_{i\to\infty}\frac{d(G_i)-1}{[G:G_i]},
	\]
where $d(G_i)$ denotes the minimal number of generators of the group $G_i$ for each $i\ge 0$.
	
The chain $(G_i)_{i\geq 0}$ naturally gives rise to a p.m.p.\ action of $G$ as follows. Equip each finite coset space $G/G_i$ with the uniform probability measure, and consider the inverse limit $X_{(G_i)_{i\ge 0}}\coloneqq\varprojlim_i G/G_i$ endowed with the associated Haar measure. Left multiplication induces a p.m.p.\ action $G\curvearrowright X_{(G_i)_{i\ge 0}}$, which is essentially free when $\bigcap_{i\geq 0}G_i=\{e_G\}$. With this latter assumption, Abért and Nikolov proved that the cost of the equivalence relation associated with the action $G\curvearrowright X_{(G_i)_{i\ge 0}}$ equals $\rg\left(G,(G_i)_{i\geq 0}\right)+1$ \cite[Theorem 1]{abert2012rank}.
	
Taking the inverse limit over all finite-index normal subgroups yields the profinite completion $\widehat{G}$ of $G$, equipped with its Haar probability measure, and the associated free p.m.p.\ action $G\curvearrowright\widehat{G}$. The corresponding rank gradient is called the \emph{absolute rank gradient} of $G$, which we denote by $\rg(G)$. Equivalently, we have
	\begin{equation*}\label{eq:absolute-rank-gradient}
		\rg(G)=\inf\left\{ \frac{d(H)-1}{[G:H]} : H\lhd G,\ [G:H]<\infty
		\right\}.
	\end{equation*}
	By combining Gaboriau's inequality between cost and the first $\ell^2$-Betti number with the definition of group cost as the infimum of the costs over all free p.m.p.\ actions, every infinite finitely generated residually finite group satisfies the inequalities
	\begin{equation}\label{eq:general-inequalities}
		\beta_1^{(2)}(G)
		\leq \cost(G)-1
		\leq \cost(G\curvearrowright\widehat{G})-1
		=\rg(G).
	\end{equation}
	There is no known example of an infinite countable residually finite group $G$ in which either inequality in \eqref{eq:general-inequalities} is strict. 
    
    The following questions have appeared several times in the literature and have motivated the study of cost and rank gradient of groups; see, e.g., Question 8.2 in Gaboriau's 2010 ICM survey \cite{gaboriauICM}, \cite[Problem 1.5]{Osin2011}, \cite[Question 1.10]{ErshovLuck2014} and \cite[Question 4.2]{Luck2016}.

\begin{question}\label{question: equality cost betti}
Is it true that $\beta_1^{(2)}(G)=\cost(G)-1$ for every infinite countable group $G$?
\end{question}
Gaboriau asked, more generally, whether $\beta_1^{(2)}(\mathcal{R})=\cost(\mathcal{R})-1$ for every aperiodic countable p.m.p.\ Borel equivalence relation $\mathcal{R}$ \cite[page 129]{gaboriau2002}. The claimed example of a group without fixed price by OpenAI \cite{OpenAI26} would answer this stronger question negatively. Indeed, the proposed group $G$ has cost $1$, so $\beta_1^{(2)}(G)=0$, while its Bernoulli action has cost strictly greater than one. The orbit relation $\mathcal{R}$ of this Bernoulli action satisfies $\beta_1^{(2)}(\mathcal{R})=\beta_1^{(2)}(G)=0$  \cite[Corollary 3.16]{gaboriau2002}. The group $G$ in OpenAI's example satisfies $\beta_1^{(2)}(G)=0=\mathrm{cost}(G)-1$, and therefore does not give a negative answer to Question \ref{question: equality cost betti}.

\begin{question}\label{question: equality cost rank gradient}
 Suppose that $G$ is infinite, finitely generated and residually finite. Does it hold that $\cost(G)-1=\rg(G)$?
\end{question}

A more general version of Question \ref{question: equality cost rank gradient} is the following: fix an arbitrary decreasing chain $(G_i)_{i\ge 0}$ of finite-index normal subgroups of $G$ with trivial intersection. Does it hold that $\cost(G)-1=\rg(G,(G_i)_{i\ge 0})$? An affirmative answer would imply a positive solution to Problem 11 in Abért's 2022 ICM survey \cite{AbertICM} and to the first sentence of Question 5.1 in Bader's 2026 ICM survey \cite{BaderpublishedICM}; see also Question 82 in the extended version of this survey by Bader and Sauer \cite{BaderICM}.

Lück observes in \cite[Remark 4.3]{Luck2016} that, for finitely presented groups, the equality $\beta_1^{(2)}(G)=\rg(G,(G_i)_{i\ge 0})$ along a residual normal chain implies that the minimum number of generators of $G_i$ and the free abelian rank of its abelianization agree asymptotically, after normalizing by the index $[G:G_i]$. He notes that this would be surprising, since these quantities need not agree for an individual finitely generated group. He writes:
\begin{displayquote}
\itshape
This raises the question whether this equality holds for a ``random group'' in the sense of Gromov.
\end{displayquote}
Our main result ensures that, asymptotically almost surely, the answers to both Question \ref{question: equality cost betti} and Question \ref{question: equality cost rank gradient} are affirmative for a random group in the equal-length version of the few-relator model; see also Subsection \ref{subsection: intro few relator} below for more details.
\begin{customthm}{A}\label{thm: main theorem}
		Let $G_\ell=\langle x_1,\ldots,x_n\mid r_1,\ldots,r_m\rangle$, $\ell\ge 1$, be a random group with $n\geq2$ generators and $m\geq1$ relators in the few-relator model. Here, $\ell$ denotes the length of the relators. Then, asymptotically almost surely as $\ell\to\infty$, we have
		\[
			\beta_1^{(2)}(G_\ell) = \mathrm{cost}(G_\ell)-1 = \rg(G_\ell) = \max\{n-m-1,0\}.
		\]
	\end{customthm}

Thus, all quantities in \eqref{eq:general-inequalities} agree for a finitely presented group whose relations are ``chosen at random''. 

Our methods do not extend to other random group models, such as Gromov's density model. They also leave open whether random few-relator groups have \emph{fixed price}, that is, whether every essentially free p.m.p.\ action of such a group has cost equal to $\max\{n-m,1\}$. We further discuss these limitations in Subsection \ref{subsection: scope}.

\subsection{The few-relator model and further context of Theorem \ref{thm: main theorem}}\label{subsection: intro few relator}
In the few-relator model of a random group, one fixes the number of generators $n\ge 2$ and the number of relators $m\ge 1$. For each $\ell \ge 1$ one considers the group $G_\ell$ with presentation obtained by choosing $m$ relators independently and uniformly among all cyclically reduced words of length $\ell$ in the generators. One says that a property holds ``asymptotically almost surely'' if the probability that $G_\ell$ satisfies the property converges to $1$ as $\ell\to \infty$. The formal definition of the probability space and model considered is given in Subsection \ref{section: probability small cancellation}.

The few-relator model was introduced by Arzhantseva and Ol'shanskii \cite{ArzhantsevaOlshanskii}, following earlier work of Guba \cite{Guba1986}. In their formulation, the number of generators and relators is fixed, and the relators are chosen among cyclically reduced words of length at most $\ell$. They proved that, generically, every subgroup generated by fewer than $n$ elements is free, thereby answering \cite[Problem 11.75]{KourovkaNotebook}. Their estimates also show that, for every fixed $\lambda>0$, the defining presentation satisfies the $C'(\lambda)$ small-cancellation condition asymptotically almost surely, and hence the group is hyperbolic \cite[Lemma 3]{ArzhantsevaOlshanskii}. These conclusions also hold when all relators have length exactly $\ell$, as explained in \cite[Proof of Theorem 2 in page 355]{ArzhantsevaOlshanskii}. We use this equal-length formulation in our paper. For further background on random groups, we refer to Ollivier's survey \cite{Ollivier2005}, and for a comparison of the Arzhantseva--Ol'shanskii model with Gromov's density model to \cite{KapovichSchupp2008}.

Kielak, Kropholler, and Wilkes previously studied $\ell^2$-homology,
coherence, and virtual fibering in the few-relator model \cite{KielakKrophollerWilkes2022}. Their
results are divided according to the sign of the Euler characteristic $\chi(G_\ell)$ of the random group, which equals $1-n+m$ asymptotically almost surely, since the presentation complex is asymptotically
almost surely aspherical.

When the presentation has negative Euler characteristic, i.e., when $1-n+m<0$, they proved that $\beta_1^{(2)}(G_\ell)=n-m-1$ asymptotically almost surely, that the remaining $\ell^2$-Betti numbers vanish, and that $G_\ell$ is coherent and embeds in a virtually free-by-cyclic group. They also obtained results that hold with positive probability (not one) in the zero Euler characteristic case and results for two-generator presentations of positive Euler characteristic; see \cite[Theorems A and B, Proposition 2.12]{KielakKrophollerWilkes2022}.
	
Kielak, Kropholler and Wilkes conjectured that $\beta_1^{(2)}(G_\ell)=\max\{n-m-1,0\}$ asymptotically almost surely for every $n\geq2$ and $m\geq1$ \cite[Conjecture 2.14]{KielakKrophollerWilkes2022}, where their main theorem covers the case $n>m+1$. The equality $\beta_1^{(2)}(G_{\ell})=\max\{n-m-1,0\}$ guaranteed by Theorem \ref{thm: main theorem} gives a complete affirmative answer to their conjecture.

The estimates of Arzhantseva-Ol'shanskii \cite[Lemma 3]{ArzhantsevaOlshanskii} imply that, for every fixed $\lambda>0$, the defining presentation of $G_\ell$ satisfies the $C'(\lambda)$ small-cancellation condition asymptotically almost surely. Taking $\lambda=1/6$, the group is hyperbolic, and \cite[Theorem 1.2]{Wise2004} together with \cite[Theorem 1.1]{Agol2013} imply that it is virtually compact special. Such a group has a finite-index subgroup that embeds in a right-angled Artin group \cite[Theorem 1.1 and Proposition 3.10]{HaglundWise2008}. Since right-angled Artin groups are residually finite, the group $G_\ell$ is residually finite asymptotically almost surely. This justifies considering its rank gradient in the statement of Theorem \ref{thm: main theorem}.

\subsection{Convergence of the tracial spectral measure of the Laplacian}\label{subsection: convergence of tracial spectral measure of the laplacian}
 We next describe the main operator-theoretic component of the proof of Theorem \ref{thm: main theorem}. In Subsections \ref{subsection: vn algebras}, \ref{subsection: preliminaries tracial spectral measure}, \ref{subsection: preliminaries homology} and \ref{subsection: laplacian} we recall in further detail the relevant definitions of $\ell^2$-cellular complexes, $\ell^2$-Betti numbers and von Neumann algebras used below.

Let $X_\ell$ be the presentation complex of $G_\ell$ and let $A_\ell=\partial_{2,\ell}^{(2)}:\ell^2(G_\ell)^m\longrightarrow\ell^2(G_\ell)^n$ be its degree-$2$ cellular $\ell^2$-boundary operator. Since $X_\ell$ is two-dimensional, its degree-$2$ Laplacian is $\Delta_{2,\ell}=A_\ell^*A_\ell$, whose kernel identifies with the reduced second $\ell^2$-homology; see \cite[Equation (1.17) and Lemma 1.18]{Luck2002}. On the asymptotically almost-sure event that $X_\ell$ is aspherical and $G_\ell$ is infinite, the $\ell^2$-Euler--Poincaré formula states that
\begin{equation}\label{eq: euler poincare}
\beta_1^{(2)}(G_\ell)=n-m-1+\mathrm{Tr}_m\!\left(\mathds{1}_{\{0\}}(A_\ell^*A_\ell)\right); \end{equation}
see \cite[Theorem 1.35 (2) and (8)]{Luck2002}.

We study the normalized Laplacian $\widehat{\Delta}_{2,\ell}=\ell^{-1}A_\ell^*A_\ell$, $\ell \ge 1$, through its tracial spectral measure
\[
\sigma_\ell(B)\coloneqq \mathrm{Tr}_m\left(\mathds{1}_B(\widehat\Delta_{2,\ell})\right), \text{ for each Borel subset } B\subseteq\R,
\]
where $\mathrm{Tr}_m$ is the unnormalized matrix trace on $M_m(L(G_\ell))$ induced by the canonical trace on the group von Neumann algebra $L(G_\ell)$. Since the normalization by $\ell$ does not change the kernel of the operator, the last term on the right-hand side of Equation \eqref{eq: euler poincare} is equal to $\sigma_\ell(\{0\})$.

The expected limiting behavior of the tracial spectral measure $\sigma_\ell$ can be motivated by an analogy with classical random matrix theory. Heuristically, since the relators of $G_\ell$ are chosen independently from each other, one may expect the operator $A_\ell/\sqrt{\ell}$ to exhibit spectral behavior similar to that of a large rectangular matrix with independent and identically distributed entries as $\ell \to \infty$. Just as the central limit theorem gives a universal Gaussian limit for suitably normalized sums of i.i.d.\ random variables with finite, non-zero variance, the
Marchenko--Pastur theorem gives a universal spectral limit for sample covariance matrices formed from i.i.d.\ entries \cite{MarchenkoPastur1967}. With the appropriate normalization, this limit is a \emph{free-Poisson law}; see \cite[Section 4.5.1]{MingoSpeicher2017} and Remark \ref{rem: comparison random matrix theorem} in Subsection \ref{subsection:free poisson preliminaries}. This theorem does not apply directly to $A_\ell$. Its entries are operators on $\ell^2(G_\ell)$ instead of complex numbers, and entries in the same column are not independent, since they all depend on a same given relator.

The free-Poisson distribution with mean $\lambda>0$ is the probability measure $\mu_{\lambda}$ on $\R$ defined by
\[
\mu_\lambda(B) =(1-\lambda)_+\delta_0(B)
+ \int_{B\cap(a_\lambda,b_\lambda)}
\frac{\sqrt{(b_\lambda-t)(t-a_\lambda)}}{2\pi t}\ dt, \text{ for each Borel subset }B\subseteq \R,
\]
where $a_\lambda=(1-\sqrt{\lambda})^2$, $b_\lambda=(1+\sqrt{\lambda})^2$  and $(1-\lambda)_+=\max\{1-\lambda,0\}$; see \cite[Proposition 12.11 and Definition 12.12]{NicaSpeicher2006} and Subsection 2.6 for further background.

The following theorem identifies the limiting tracial spectral measure in our setting.
	
	\begin{customthm}{B}\label{thm:main spectral law}
		Let
		$G_\ell=\langle x_1,\ldots,x_n\mid r_1,\ldots,r_m\rangle$, $\ell\ge 1$,
		be a random group with $n\geq2$ generators and $m\geq1$ relators in the
		few-relator model. Denote by $L(G_{\ell})$ the group von Neumann algebra of $G_{\ell}$, and let
		\[
		\widehat{\Delta}_{2,\ell}
		=
		\frac{1}{\ell}A_\ell^*A_\ell
		\in M_m\bigl(L(G_\ell)\bigr)
		\]
		be the normalized degree-$2$ cellular $\ell^2$-Laplacian on the universal cover of the presentation complex of $G_\ell$. Let $\sigma_\ell$ be the tracial spectral measure of $\widehat{\Delta}_{2,\ell}$ with respect to the matrix trace on $M_m(L(G_\ell))$. Let us denote by $\mu_\lambda$ the free-Poisson law of mean $\lambda\coloneqq m/(n-1)$. Then
		\[
		\sigma_\ell \xrightarrow[\ell\to\infty]{} \sigma_\infty \coloneqq (n-1)\mu_\lambda+(m-n+1)\delta_0
		\]
		weakly in probability\footnote{Note that the measure $\sigma_\infty$ is positive:
			when $m<n-1$, the negative correction at zero cancels exactly the atom
			of $(n-1)\mu_\lambda$ at zero.}. More precisely, for every bounded continuous function $f:\R\to\R$ and every $\varepsilon>0$, we have
		\begin{equation}\label{eq:weak spectral convergence}
			\P_\ell\left( \left| \int_{\R}f(t)\ d\sigma_\ell(t) - \int_{\R}f(t)\ d\sigma_\infty(t) \right|>\varepsilon \right) \xrightarrow[\ell\to\infty]{}0.
		\end{equation}
	\end{customthm}
To prove Theorem \ref{thm:main spectral law}, we establish convergence in probability of the moments of $\sigma_\ell$ to those of $\sigma_\infty$. The main difficulty is to control the dependencies arising from the defining relators. This requires combining known geometric estimates for intersections of lifted relator cells, which rely on small-cancellation methods, with probabilistic estimates for subwords of the random relators. We carry this out in Sections \ref{section: trace laplacian} and \ref{section:convergence of the spectral measure}, and deduce weak convergence from moment convergence in Subsection \ref{subsection:proof of spectral convergence}.
 
\subsection{Background, related literature and scope of our methods}\label{subsection: scope}
A group $G$ is said to have \emph{fixed price} if all its free p.m.p.\ actions have the same cost. For a residually finite group with fixed price, the second inequality in \eqref{eq:general-inequalities} is therefore an equality.

A well-known question of Gaboriau asks whether there are countable groups that do not have fixed price. A recent preprint of OpenAI claims that such groups do exist \cite{OpenAI26}. However, it is not clear from their arguments whether one can also find groups for which either of the inequalities in \eqref{eq:general-inequalities} is strict.

 Several classes of groups are known to have fixed price. Gaboriau proved that the free group $F_r$ of rank $r$ has fixed price $r$, and that every group containing an infinite amenable normal subgroup has fixed price $1$ \cite[Section VI.26]{gaboriau2000cout}. More recently, Fr\k{a}czyk, Mellick, and Wilkens proved fixed price one for lattices in higher-rank semisimple real Lie groups \cite{fraczykmellickwilkens}. A recent preprint by Slutsky gives a combinatorial criterion for fixed price $1$ that recovers several previously known examples and provides new ones \cite{Slutsky2026}.
	
In some situations, the cost of a group can be determined without answering the question on whether the group has fixed price. A notable example is a theorem of Hutchcroft and Pete, stating that every infinite group with Kazhdan's property (T) has cost $1$ \cite{HutchcroftPete2020}; whether all such groups have fixed price one remains an open problem. In recent work, the authors proved that free Burnside groups of sufficiently large odd exponent also have cost $1$ \cite{donososilva2026}, although fixed price is not known for these groups either. Note that free Burnside groups of sufficiently large non-prime exponent do not have property (T) \cite{Osajda2018}, so this result is not covered by Hutchcroft and Pete's theorem.

In a recent preprint, Poulin and Wróbel computed the cost of every infinite one-relator group \cite{poulinwrobel2026}. They show that if $G$ admits a presentation $G=\langle S\mid w^q\rangle$, where $w$ is not a proper power, then $\cost(G)=|S|-q^{-1}$. Combined with the corresponding computation of the first $\ell^2$-Betti number, their result shows that Gaboriau's inequality $\beta_1^{(2)}(G)+1\leq\cost(G)$ is an equality for every infinite one-relator group. In particular, for few-relator random groups with a single defining relator, the equality $\cost(G)=\beta_1^{(2)}(G)+1$ established in Theorem \ref{thm: main theorem} already follows from their result. Indeed, the defining relator is asymptotically almost surely not a proper power, so their formula gives $\cost(G)=|S|-1$. Their results also apply to certain presentations with several relators, including finite reducible presentations without proper powers \cite[Corollary 4.2]{poulinwrobel2026}. However, with at least two relators, the random presentations considered here are asymptotically almost surely not reducible: every relator contains every generator asymptotically almost surely, whereas reducibility requires a generator that occurs in only one relator.
	
Rank gradient was introduced by Lackenby in his study of expansion in finite quotients of finitely presented groups \cite{Lackenby2005}. His main result \cite[Theorem 1.1]{Lackenby2005} can be phrased as follows. Let $G$ be a finitely presented group and let $(G_i)_{i\geq 0}$ be a family of finite-index normal subgroups closed under taking intersections. Suppose that only finitely many members of this family admit a non-trivial splitting as an HNN extension or amalgamated free product. If $\rg(G,(G_i)_{i\ge 0})>0$ then the Cayley graphs of the quotients $G/G_i$, with respect to the images of a fixed finite generating set of $G$, form an expander family.

Subsequent computations of rank gradient include the vanishing of the rank gradient for any infinite finitely generated amenable groups along every residual normal chain \cite[Theorem 5]{AbertJaikinNikolov2011}, and more generally, for finitely generated residually finite groups containing an infinite amenable normal subgroup \cite[Theorem 3]{abert2012rank}. For infinite limit groups, Bridson and Kochloukova established the formula $\rg(G,(G_i)_{i\geq 0})=-\chi(G)$ along every residual normal chain \cite[Theorem A]{BridsonKochloukova2017}, where $\chi(G)$ is the Euler characteristic of $G$. A major recent advance is the work of Frączyk, Mellick, and Wilkens, who showed the vanishing rank gradient along every residual normal chain in a lattice in a higher-rank semisimple real Lie group \cite[Theorem D]{fraczykmellickwilkens}.

Nevertheless, even for residual normal chains, it remains unknown whether the value of the rank gradient can depend on the chosen chain. Abért asked whether every infinite finitely presented residually finite group with property (T) has vanishing rank gradient along every such chain \cite[Problem 11]{AbertICM}. A challenging unresolved case is that of lattices in rank-one simple Lie groups with property (T), such as $\mathrm{Sp}(n,1)$ for $n\geq 2$. Although these lattices have vanishing first $\ell^2$-Betti number and cost one, these computations do not settle their rank gradient. Neither the results of Hutchcroft and Pete nor those of Poulin and Wróbel provide a computation of rank gradient.

The arguments and methods in our proofs do not establish fixed price for random few-relator groups. The upper bounds for cost are obtained through
finite-index subgroups and graphings of the profinite action, whose cost is related to the absolute rank gradient by \cite[Corollary 11]{abert2012rank}. The graphings constructed using finite quotients do not furnish such graphings for an arbitrary action. By \cite[Corollary 2]{AbertWeiss}, Bernoulli actions attain the maximal cost among essentially free p.m.p.\ actions of a finitely generated group. Consequently, in view of our computation of $\cost(G_\ell)$, it would suffice to prove that a non-trivial \emph{Bernoulli action} of $G_\ell$ has cost at most $\max\{n-m,1\}$. Our methods do not yield this additional bound; see also \cite[Remark 3.6]{poulinwrobel2026} where the limitations of the edge-replacement method are discussed.

The proof of Theorem \ref{thm: main theorem} does not extend to Gromov's density model. Our spectral moment estimates use both the fixed number of relators and the logarithmic bound $O(\log\ell)$ on the maximal length of a piece. At a fixed density $d\in(0,1/2)$, the number of relators grows exponentially with $\ell$, and the maximal piece length is of order $\ell$; see \cite[Section 9.B]{GromovAsymptotic} and
\cite[Proposition 1.5]{MackayConformalDimension2012}.
Although random presentations still satisfy $C'(1/6)$
asymptotically almost surely in the density regime $d<1/12$, this condition allows overlaps whose lengths are a fixed positive proportion of $\ell$. The stronger logarithmic bound required by our moment estimates therefore fails even at arbitrarily small fixed positive density. The same obstruction applies to the low-density version of the
Arzhantseva--Ol'shanskii model defined in
\cite[Definition 2.3]{KapovichSchupp2008}, where
the positive density $d=d(n)$ may depend on the number of generators but remains fixed as $\ell\to\infty$.

Spectral measures have also been studied for random simplicial complexes in the literature. In the Linial-Meshulam model, one considers the complete $(d-1)$-skeleton on $N$ vertices and includes each $d$-simplex independently with probability $p$. Knowles and Rosenthal prove a semicircle law for rescaled adjacency eigenvalues \cite[Theorem 2.5]{KnowlesRosenthal2017}. See also, e.g., \cite[Theorem 2]{GundertWagner2016} and \cite[Theorem 2.4]{LeibzirerRosenthal2025} for related results on the distribution of eigenvalues. Closer to our homological applications, Linial and Peled use the mass at zero of a limiting Laplacian spectral measure to study Betti numbers \cite[Corollary 3.4 and Section 4]{LinialPeled2016}. These results concern finite complexes with growing numbers of vertices. Our presentation complexes have fixed numbers of cells, whereas their contours' lengths go to infinity.
	
The next three subsections outline the proofs of Theorems
\ref{thm: main theorem} and \ref{thm:main spectral law}. These sketches are intended to provide both an overview of the proofs as well as a roadmap when reading the complete proofs in detail. A glossary of the notation used throughout the paper can be found in page \pageref{sec:nomenclature}.

\subsection{Sketch of the proof of Theorem \ref{thm: main theorem} in the case of non-negative Euler characteristic}
 We first consider the case in which the presentation complex of $G_\ell$ has non-negative Euler characteristic, i.e., when $m \ge n-1$. In this case, the statement of Theorem \ref{thm: main theorem} is that
\[
\P_\ell\left(\beta_1^{(2)}(G_\ell) = \cost(G_\ell)-1 = \rg(G_\ell)= 0\right)\xrightarrow[\ell\to\infty]{} 1.
\]
This result is stated separately in Theorem \ref{thm:low_deficiency}.

\begin{stepsAone}
 \item Our starting point are several properties known to hold asymptotically for a random group in the few-relator model that we gather in Proposition \ref{prop: lifted face geometry} and which we briefly recall now. Asymptotically almost surely, the presentation of $G_{\ell}$ satisfies the $C^{\prime\prime}(1/6)$ small-cancellation condition, and $G_\ell$ is infinite, hyperbolic, and torsion-free, with an aspherical presentation complex. These properties follow from Arzhantseva--Ol'shanskii estimates in \cite[Lemmas 3 and 4]{ArzhantsevaOlshanskii} together with known consequences of the small cancellation condition; see \cite[Theorem 13.3]{OlshanskiiDefiningRelations} and \cite[Section 4.7]{GromovHyperbolic}. By Wise's cubulation theorem \cite[Theorem 1.2]{Wise2004}, $G_\ell$ acts properly and cocompactly on a $\mathrm{CAT}(0)$ cube complex. Agol's theorem then implies that $G_\ell$ is virtually special \cite[Theorem 1.1]{Agol2013}. Consequently, $G_\ell$ is residually finite \cite[Theorem 4.4]{HaglundWise2008}, virtually residually finite rationally solvable (RFRS) \cite[Corollary 2.3]{Agol2008}, and satisfies the strong Atiyah conjecture over $\mathbb{Q}$ \cite[Theorem 1.2]{Schreve2014}.
 \item It is thus enough to prove that
 \[
\P_\ell\left(\beta_1^{(2)}(G_\ell)<1\right) \xrightarrow[\ell\to\infty]{} 1.
\]
 Indeed, let $G_\ell$ be a group in the event of the previous step. Torsion freeness and the fact that $G_\ell$ satisfies the Atiyah conjecture over $\Q$ imply that $\beta_1^{(2)}(G_\ell)$ is an integer. Thus $\beta_1^{(2)}(G_\ell)<1$ forces $\beta_1^{(2)}(G_\ell)=0$. Since $G_\ell$ is virtually RFRS, it follows by Kielak's virtual fibration theorem \cite[Theorem 5.3]{Kielak2020} that $G_\ell$ has a finite-index subgroup that maps homomorphically onto $\Z$ with a finitely generated kernel. Any such group has a vanishing absolute rank gradient (see Proposition \ref{prop: vanishing_rank_gradient}) and therefore Inequalities \eqref{eq:general-inequalities} give $\cost(G_\ell)=1$.

 \item
 The proof that $\beta_1^{(2)}(G_\ell)<1$ asymptotically almost surely relies on Theorem \ref{thm:main spectral law}. Recall from Equation \eqref{eq: euler poincare} that, on the event that the presentation complex is aspherical and $G_\ell$ is infinite, we have
 \[
 \beta_1^{(2)}(G_\ell) = n-m-1+\mathrm{Tr}_m\left(\mathds{1}_{\{0\}}\left(\widehat{\Delta}_{2,\ell}\right)\right)=n-m-1+\sigma_\ell(\{0\}),
 \]
 where $\widehat{\Delta}_{2,\ell}$ is the normalized degree-$2$ Laplacian; see also Subsection \ref{subsection: laplacian}. Applying the weak convergence in probability established in Theorem \ref{thm:main spectral law}, we obtain in Theorem \ref{thm: vanishing first l2 betti number} that
 \[
 \P_{\ell}\left(\sigma_{\ell}(\{0\})-\sigma_\infty(\{0\})>\frac{1}{2}\right)\xrightarrow[\ell \to \infty]{}0.
 \]
 Since $\sigma_\infty(\{0\})=m-n+1$, we obtain that asymptotically almost surely 
 \[\beta_1^{(2)}(G_\ell)=n-m-1+\sigma_\ell(\{0\})\leq \frac{1}{2}<1.\]
 As explained in the preceding step, this finishes the proof.
\end{stepsAone}

\subsection{Sketch of the proof of Theorem \ref{thm:main spectral law}}
Let $X_\ell$ be the presentation complex of $G_\ell$, and let $\widetilde{X}_\ell$ be its universal cover. To prove Theorem \ref{thm:main spectral law}, we first prove convergence in probability of every moment of $\sigma_\ell$ to the corresponding moment of $\sigma_\infty$, and then deduce weak convergence. The  steps of the proof are the following.

\begin{stepsB}
\item For each fixed integer $k\geq1$, Lemma \ref{lem: trace expansion} computes the tracial moment $\mathrm{Tr}_m(\Delta_{2,\ell}^k)$ as a sum over tuples $(F_0,\ldots,F_{k-1})$ of $2$-cells in $\widetilde{X}_\ell$, where the initial face $F_0$ ranges only over a fundamental domain of the $m$ defining $2$-cells of the presentation complex. Each tuple determines a partition $\pi$ of $\{0,\ldots,k-1\}$ that encodes which faces coincide: two indices $i,j$ belong to the same class precisely when $F_i=F_j$. The map $\typemap:\pi\to\{1,\ldots,m\}$ assigns to each class the index of the defining relator corresponding to its face. Different classes may have the same relator type. Grouping the terms that appear in the expansion of the tracial moment according to these data, we obtain in Lemma \ref{lem: partition trace} the equality
\[
\int_{\R}t^k\ d\sigma_\ell(t) = \frac{1}{\ell^k} \sum_{\pi\in\mathcal{P}(k)} \sum_{\typemap:\pi\to\{1,\ldots,m\}} Z_\ell(\pi,\typemap).
\]
 Here, $\mathcal{P}(k)$ denotes the set of partitions of $\{0,\ldots,k-1\}$ and $Z_\ell(\pi,\typemap)$ is the sum of the trace terms with the prescribed equality pattern and relator types, as defined in Equation \eqref{eq: def of Z}.

\item We estimate these terms on the generic event $\mathcal{G}_\ell$ of Proposition \ref{prop: lifted face geometry}. Here, every lifted face has an embedded boundary of length $\ell$, and every non-empty intersection of two distinct lifted faces is a path of length at most $\rho_\ell=O(\log\ell)$.

Lemma \ref{lemma: partition off diagonal} rewrites
$Z_\ell(\pi,\typemap)$ as a signed sum over choices of common boundary edges between consecutive distinct faces.
The geometric bounds above allow us to count these choices. In particular, Lemmas \ref{lem:bound order l} and
\ref{lem: noncrossing_partitions} imply that, whenever $\pi$ is a crossing partition in the sense of
Definition \ref{defn: noncrossing partitions}, we have
\[
\left|\frac{Z_\ell(\pi,\typemap)}{\ell^k}\right|
\leq C\frac{(\log\ell)^{2k}}{\ell} \xrightarrow[\ell\to\infty]{}0,
\]
where $C$ is independent of $\ell$ and of the presentation
in $\mathcal{G}_\ell$. Since the number of partitions and
type maps is finite for fixed $k$, the contribution of crossing partitions is negligible to the asymptotic value of the moment.

\item
We next evaluate the terms indexed by non-crossing partitions, whose collection is denoted by $\mathrm{NC}(k)$. 

Fix $\pi\in\mathrm{NC}(k)$ and a type map $\typemap:\pi\to \{1,\ldots,m \}$. Write $|\pi|$ for the number of classes of $\pi$. The main technical result of the section is Proposition \ref{prop:main_technical}, which states that
\[
\frac{Z_\ell(\pi,\typemap)}{\ell^k} \xrightarrow[\ell\to\infty]{\P_\ell} \frac{1}{(n-1)^{|\pi|-1}}.
\]
When $|\pi|=1$, this follows directly from the embedded
boundary of the prescribed initial face, since we have
$Z_\ell(\pi,\typemap)=\ell^k$ on the event $\mathcal{G}_\ell$.

For $|\pi|\geq 2$, we use windows of radius $R_\ell=2k\rho_\ell=O(\log\ell)$ in the defining relators.
A window of radius $R$ consists of the $2R+1$ consecutive
letters centered at a chosen relator position, with positions read cyclically; see Definition \ref{defn: windows}. The function $\Phi_{\pi,R}$ of Definition \ref{defn: observables} counts, with orientation signs, the placements of these windows in the Cayley graph
of the free group $F_n$ that satisfy the edge identifications encoded by $\pi$.

Lemma \ref{lem: comparison between Z and average of observables} approximates $\ell^{-k}Z_\ell(\pi,\typemap)$ by the average of this function over tuples of positions in the prescribed relators. In this average, windows belonging to the same relator have distinct centers, as specified in Definitions \ref{defn: tau admissible roots} and \ref{defn: average over tau admissible roots}. The approximation error tends to zero on the event $\mathcal{G}_\ell$. Proposition \ref{prop: weak ergodic theorem} then replaces
this average, up to an error tending to zero in probability, by its expectation for independent windows, each uniformly distributed among non-backtracking words of length $2R_\ell+1$. Lemma \ref{lem: expectation of the observable} shows that these expectations converge to $(n-1)^{1-|\pi|}$. This proves the claimed limit.

\item For each partition $\pi$, there are $m^{|\pi|}$ possible type maps $\typemap$. Combining the vanishing of the crossing terms with the limits of the non-crossing terms gives, for every $k\geq1$, that
\[
\int_{\R}t^k\,d\sigma_\ell(t) \xrightarrow[\ell\to\infty]{\P_\ell} \sum_{\pi\in\mathrm{NC}(k)} \frac{m^{|\pi|}}{(n-1)^{|\pi|-1}} = (n-1)\sum_{\pi\in\mathrm{NC}(k)} \left(\frac{m}{n-1}\right)^{|\pi|}.
\]
By Proposition \ref{prop:free poisson moments}, the last
expression above is the $k$-th moment of $(n-1)\mu_\lambda$, where $\lambda=m/(n-1)$. It is also the $k$-th moment of the measure
\[
\sigma_\infty=(n-1)\mu_\lambda+(m-n+1)\delta_0,
\]
because $t^k$ vanishes at zero for $k\geq1$.
\item
Finally, Subsection \ref{subsection:proof of spectral convergence} constructs events $\mathcal{H}_\ell$ with
$\P_\ell(\mathcal{H}_\ell)\to1$ such that, along any choice of one presentation from each $\mathcal{H}_\ell$, all moments of the resulting spectral measures converge to those of $\sigma_\infty$. These measures are supported on $[0,\infty)$ and have mass $m$. Their bounded first moments show that the sequence is tight, while bounds on higher moments ensure that every subsequential weak limit has the same moments as $\sigma_\infty$. Since $\sigma_\infty$ has compact support, these moment identities force the weak limit to have the same compact support, and the Stone--Weierstrass theorem then identifies it with $\sigma_\infty$. This proves the weak convergence in probability stated in Equation \eqref{eq:weak spectral convergence}.
\end{stepsB}

\subsection{Sketch of the proof of Theorem \ref{thm: main theorem} in the negative Euler characteristic case}
 
The negative Euler characteristic case is stated in Theorem \ref{thm:high_deficiency} and proved in Section \ref{section:main theorem high deficiency}. Its proof is independent of Sections \ref{section: trace laplacian}, \ref{section:convergence of the spectral measure}, and \ref{section:main theorem low deficiency}. The argument uses groupoid cost and proceeds as follows.

\begin{stepsAtwo}
 \item
 By \cite[Theorem B]{KielakKrophollerWilkes2022}, we already know that $\beta_1^{(2)}(G_\ell)=n-m-1$ asymptotically almost surely. In view of
 Inequalities \eqref{eq:general-inequalities}, to prove Theorem \ref{thm:high_deficiency}, it suffices to prove that $\rg(G_\ell)\le n-m-1$ with probability tending to $1$ as $\ell \to \infty$.

 \item Kielak, Kropholler, and Wilkes also show that, when $m<n-1$, the random presentation asymptotically almost surely satisfies a minimum condition (Definition \ref{def: minimum condition}). The minimum condition allows certain edges in the Cayley graph to be replaced by paths with the same endpoints, as described in Lemma \ref{lem: path lemma}. We learned the latter technique from Poulin and Wr{\'o}bel's recent preprint \cite{poulinwrobel2026}.

 \item In Proposition \ref{prop: groupoid cost}, we use the technique described in the previous step to bound the groupoid cost of the action of $G_\ell$ on its profinite completion by $n-m$. Since $G_\ell$ is residually finite asymptotically almost surely, this action is free. Its groupoid cost therefore equals its cost, which is $\rg(G_\ell)+1$. We conclude that $\rg(G_\ell)\le n-m-1$, as required.
\end{stepsAtwo}

	\subsection{Organization of the article}
	In Section \ref{section: preliminaries von neumann algebras and free Poisson} we review basic definitions and results concerning the few-relator model for random groups and some consequences of small cancellation theory; group von Neumann algebras and traces; functional calculus and spectral measures; cellular $\ell^2$-chains and $\ell^2$-Betti numbers; the Hodge Laplacian and the Euler-Poincaré formula, which allows to relate its tracial spectral measure with the first $\ell^2$-Betti number; the free-Poisson distribution; non-crossing partitions. 

 The main body of the article is contained in Sections \ref{section: trace laplacian}, \ref{section:convergence of the spectral measure}, \ref{section:main theorem low deficiency} and \ref{section:main theorem high deficiency}. In Section \ref{section: trace laplacian} we rewrite the trace of the $k$-th power of the Laplacian in terms of partitions. Section \ref{section:convergence of the spectral measure} contains the convergence of the spectral measure of the normalized Laplacian to the measure $\sigma_{\infty}$, established in Theorem \ref{thm:main spectral law}. Section \ref{section:main theorem low deficiency} proves Theorem \ref{thm: main theorem} in the non-negative Euler characteristic setting (Theorem \ref{thm:low_deficiency}). Finally, Section \ref{section:main theorem high deficiency} is devoted to the proof of Theorem \ref{thm: main theorem} in the negative Euler characteristic setting (Theorem \ref{thm:high_deficiency}). 

 We include a list of notation in page \pageref{sec:nomenclature} to facilitate the reading of the article.

	\subsection*{Acknowledgments}
    We thank Antoine Poulin and Konrad Wróbel for stimulating conversations during the Oberwolfach workshop Groups and Dynamics: Topology, Measure, and Borel Structure in April 2026 which led to this project. We are grateful to Alexander Van Werde for reading a preliminary draft, for useful discussions, and for pointing us to literature on random matrices and spectral limit theorems. We also thank Konrad Wróbel for reading our manuscript, and for remarking to us that OpenAI's claimed example of a group without fixed price implies a negative solution to Gaboriau's question on the equality between the first $\ell^2$-Betti number plus one and the cost of an aperiodic countable Borel p.m.p.\ equivalence relation. We are very grateful to Goulnara Arzhantseva for reading an early draft of this paper and for her valuable comments, which helped us improve both the historical account and the exposition of the paper. We thank Konstantin Recke for reading a preliminary draft of this paper and giving us feedback.
    
    The authors are supported by the Deutsche Forschungsgemeinschaft (DFG, German Research Foundation) under Germany's Excellence Strategy EXC 2044/2 –390685587, Mathematics Münster: Dynamics–Geometry–Structure. 

 \subsection*{AI disclosure}
 The mathematical arguments and the writing process were carried out entirely by the authors. GPT 5.6 Sol and GPT 6 Astra were used to improve the exposition of the document, to proofread, and to explore the literature. The proofreading part allowed us to remove unnecessary steps from our application of Theorem \ref{thm:main spectral law} to Theorem \ref{thm: main theorem}, thus simplifying the argument.

	\section{Preliminaries}\label{section: preliminaries von neumann algebras and free Poisson}

 \subsection{The few-relator model and consequences of small cancellation}\label{section: probability small cancellation}
Let us fix integers $n\geq2 $ and $m\geq1 $, and denote
	\[
	\mathcal{A}=\{x_1,x_1^{-1},\ldots,x_n,x_n^{-1}\}.
	\]
 \mynomenclature[G]{$\mathcal{A}$}{The set $\{x_1,x_1^{-1},\dots,x_n,x_n^{-1}\}$ of generators of $G_\ell$ and inverses}
	Let $F_n$ be the free group with free basis $\{x_1,\ldots,x_n\}$. For each $\ell\geq 1$, let $\mathcal{C}_\ell$
 \mynomenclature[G]{$\mathcal{C}_\ell$}{The set of cyclically reduced words of length $\ell$ on the alphabet $\mathcal{A}$}
 be the set of cyclically reduced words of length $\ell$ on $\mathcal{A}$. We equip $\mathcal{C}_\ell^m$ with the uniform probability measure $\P_\ell$.
 \mynomenclature[G]{$\P_\ell$}{The uniform probability measure on the set $\mathcal{C}_\ell^m$ of $m$-tuples of cyclically reduced words of length $\ell$ on the alphabet $\mathcal{A}$, where $m\geq 1$ is fixed}
 Equivalently, choose words $r_{1,\ell},\ldots,r_{m,\ell}$ independently and uniformly distributed. The associated random presentation is
	\[
	G_\ell =\langle x_1,\ldots,x_n \mid r_{1,\ell},\ldots,r_{m,\ell}\rangle.
	\]
A sequence of events $(E_\ell)_{\ell\geq 1}$ occurs asymptotically \emph{almost surely}, or \emph{with high probability}, if $\lim_{\ell\to\infty}\P_\ell(E_\ell)=1$. Given a sequence of random variables $Y_\ell$, $\ell\ge 1$, depending on the presentation $G_{\ell}$, and a constant $y\in \R$, we will write $Y_{\ell}\xrightarrow[\ell\to \infty]{\P_{\ell}}y$
\mynomenclature[G]{$\xrightarrow[\ell\to \infty]{\P_{\ell}}$}{Convergence in probability along the presentations $G_\ell$}
to mean that for every $\varepsilon>0$, we have $\P_\ell\left(\left|Y_{\ell}-y\right| >\varepsilon\right)\xrightarrow[\ell\to \infty]{}0$.

 A sequence $z_0,z_1,\ldots,z_h$ of elements of $\mathcal{A}$ is said to be \emph{backtracking} if for some $i=0,\ldots, h-1$ there is $1\le j\le n$ such that $z_i=x_j$ and $z_{i+1}=x_j^{-1}$ or $z_i=x_j^{-1}$ and $z_{i+1}=x_j$. If no such indices exist, we say the sequence is \emph{non-backtracking}.

  We record in the next lemma elementary estimates that can be found in the proof of Lemma 3 in \cite{ArzhantsevaOlshanskii}.  
\begin{lem}\label{lem: number of cyclically reduced words}
	Consider the notation from the preceding paragraph. Let $a,b\in\mathcal{A}$ and let $h\geq0$. Denote by $N_h(a,b)$ the number of non-backtracking sequences $a=z_0,z_1,\ldots,z_h=b$ of length $h$. Then
		\begin{equation}\label{eq: path count}
			\left|	N_h(a,b)-\frac{(2n-1)^h}{2n}\right|	\leq1.
		\end{equation}
		If $\mathcal{C}_\ell$ denotes the set of cyclically reduced words of length $\ell$, then
		\begin{equation}\label{eq: number of cyclically reduced words}
			|\mathcal{C}_\ell| =(2n-1)^\ell+n+O(1).
		\end{equation}
	\end{lem}

	Write the random relators of $G_\ell$ as $r_{j,\ell}=s_{j,0}s_{j,1}\cdots s_{j,\ell-1}$
 \mynomenclature[G]{$r_{j,\ell}$}{The $j$-th relator from the presentation of $G_\ell$, where $1\leq j\leq m$. We write $r_{j,\ell}=s_{j,0}s_{j,1}\cdots s_{j,\ell-1}$, with $s_{j,p}\in \mathcal{A}$ for each $0\leq p<\ell$}
 for each $1\leq j\leq m$, where $s_{j,l}\in\mathcal{A}$ for every $0\le l <\ell$. The second index is meant to be read cyclically as an element of $\Z/\ell\Z$. A \emph{rooted oriented reading} is a triple
	\[
	\omega=(j,p,\varepsilon)\in\{1,\ldots,m\}\times \Z/\ell\Z\times \{+1,-1\}.
	\]
	The triple $\omega=(j,p,\varepsilon)$ specifies the relator $r_{j,\ell}$, a root at the position $p$ in this word, and an orientation in which to read the relator. For each $0\leq l<\ell$, we define
	\[
	\mathsf{R}_\omega(l)\coloneqq
	\begin{cases}
		s_{j,p+l}, & \text{if }\varepsilon=+1,\\
		s_{j,p-1-l}^{-1}, & \text{if }\varepsilon=-1.
	\end{cases}
	\]
	Thus, in the positive orientation we read forward starting with $s_{j,p}$, whereas in the negative orientation we read backward starting with $s_{j,p-1}^{-1}$.
	
A word $u$ with $1\leq |u|\leq\ell$ is called a \emph{piece} if it occurs as the initial segment of two distinct rooted oriented readings. Equivalently, there exist $\omega\neq\omega'$ such that
	\[
	u=\mathsf{R}_\omega(0)\cdots\mathsf{R}_\omega(|u|-1)=\mathsf{R}_{\omega'}(0)\cdots\mathsf{R}_{\omega'}(|u|-1).
	\]
	We denote by $P_\ell$ the maximum length of a piece, with the convention that $P_\ell=0$ if no piece exists.
	
	The following logarithmic bound on the length of pieces is standard in
	the few-relator model. The $C'(\lambda)$ conclusion below for each fixed $\lambda>0$ is proved in \cite[Lemma 3]{ArzhantsevaOlshanskii}.  We refer to \cite[Chapter V, Section 2]{LyndonSchupp1977} for the definitions of the classical small-cancellation conditions and to \cite[Definition 1.2]{Gruber2015} for their graphical counterparts. Graphical small-cancellation theory was introduced by Gromov in \cite[Section 2]{Gromov2003}.
	\begin{lem}\label{lem: logarithmic}
		Consider the few-relator model of a random group $G_{\ell}$ with $n$ generators and $m$ relators. Denote by $P_\ell$ the maximum length of a piece in the presentation. Then there exists a constant $C_0>0$, depending only on $n$ and $m$, such that
		\[
		\P_\ell\left(P_\ell\leq C_0\log(\ell)\right)\xrightarrow[\ell\to \infty]{} 1.
		\]
		As a consequence, for every fixed $\lambda>0$ with asymptotic probability $1$ the presentation satisfies
		$C^{\prime\prime}(\lambda)$ and the $C'(\lambda)$ condition.
	\end{lem}
	\begin{proof}
		The first statement is proved in \cite[Proposition 2.2]{MackayConformalDimension2012}. In that proposition it is shown that for some $C_0>0$, with probability tending to $1$, no two distinct rooted oriented readings have a common initial segment of length greater than $C_0\log(\ell)$. This is equivalent to saying that $P_\ell\leq C_0\log\ell$. Although the proposition is formulated using relators of length at most $\ell$, its proof gives estimates uniform over prescribed
		relator lengths in $[0.99\ell,\ell]$. In particular, it applies
		when every relator has length exactly $\ell$.

For every fixed $\lambda>0$ and for all sufficiently large $\ell$ we have
	\[
	C_0\log(\ell)<\lambda\ell.
	\]
On the event $P_\ell\leq C_0\log(\ell)$, every piece therefore has length less than $\lambda\ell$. Moreover, no defining relator is a proper power: otherwise, two distinct cyclic roots of that relator would determine rooted oriented readings agreeing for all $\ell$ letters. Thus, the presentation satisfies $C^{\prime\prime}(\lambda)$, and hence also the $C'(\lambda)$ condition.
\end{proof}
	
	Throughout the paper, let us fix a constant $C_0$ as in Lemma \ref{lem: logarithmic} and set
	\begin{equation}\label{eq: rho ell}
	\rho_\ell\coloneqq\max\{1,\lceil C_0\log(\ell)\rceil\}.
	\end{equation}
	Then, thanks to Lemma \ref{lem: logarithmic}, $\rho_\ell$ is a positive integer that satisfies	$\rho_\ell=O(\log\ell)$ and $\P_\ell(P_\ell\leq\rho_\ell) \xrightarrow[\ell\to \infty]{} 1$. We use $\rho_\ell$ as the upper bound for piece lengths in the presentation of $G_{\ell}$.

 We will also need the following basic fact.
	\begin{lem}\label{lem: generator present}
		Consider the few-relator model of a random group $G_{\ell}$ with $n$ generators and $m$ relators. Then with asymptotic probability $1$, every letter $a\in\mathcal{A}$
		occurs in every defining relator $r_1,\ldots,r_m$.
	\end{lem}
	\begin{proof}
		Fix $0<c<\left(\log(2n-1)\right)^{-1}$. By Proposition 2.6 and its proof in \cite{MackayConformalDimension2012}, with asymptotic probability $1$ every reduced word of length $\lceil c\log(\ell)\rceil$ occurs as a subword of every defining relator. For each $a\in\mathcal{A}$, choose such a reduced word beginning with $a$. It follows that every $a\in\mathcal{A}$ occurs in every defining relator.
	\end{proof}
	We recall terminology concerning simplicial complexes that will be used below in Proposition \ref{prop: lifted face geometry}. We refer to \cite[Definition 1.1 and Fact 1.2]{JanuszkiewiczSwiatkowski} for the next definitions.
	\begin{defn}
		\begin{itemize}
			\item A simplicial complex is \emph{flag} if every finite collection of
			pairwise adjacent vertices spans a simplex.
			
			\item A cycle in a simplicial complex is \emph{full} if the subcomplex
			spanned by its vertices is precisely the cycle.
			
			\item Let $p\geq4$. A simplicial complex is \emph{$p$-large} if it is
			flag and its $1$-skeleton contains no induced cycle of length $r$ with $4\leq r<p$.
			
			\item Let $K$ be a simplicial complex and let $\sigma\in K$ be a simplex. The \emph{link} of $\sigma$ in $K$ is the subcomplex
			\[
			\mathrm{Lk}_K(\sigma) \coloneqq \left\{\tau\in K:	\tau\cap\sigma=\varnothing \text{ and } \tau\cup\sigma\in K \right\}.
			\]
			
			\item A simplicial complex $K$ is \emph{$p$-systolic} if it is connected, simply connected, and for every non-empty simplex $\sigma\in K$, the
			link $\mathrm{Lk}_K(\sigma)$ contains no full cycle of length less
			than $p$.
		\end{itemize} 
	\end{defn}
	The last condition is the link formulation of local $p$-largeness given in
	\cite[Fact 1.2(0)]{JanuszkiewiczSwiatkowski}, and it is equivalent to the definition of a $p$-systolic complex in \cite[Definition 1.1]{JanuszkiewiczSwiatkowski}. If $p\geq 6$, every $p$-systolic complex is itself $p$-large by \cite[Proposition 1.4]{JanuszkiewiczSwiatkowski}.
	
	\begin{defn}
		The \emph{nerve} of a cover $(U_\alpha)_{\alpha\in I}$ is the simplicial
		complex with vertex set $I$ in which $\alpha_0,\ldots,\alpha_d$ span a simplex when $U_{\alpha_0}\cap\cdots\cap U_{\alpha_d}\neq \varnothing$.
	\end{defn}
	
	Let $X_\ell$ be the presentation complex of $G_{\ell}$: it has one vertex, one oriented $1$-cell for each generator, and one $2$-cell attached along each defining relator.
 \mynomenclature[G]{$X_\ell$}{The presentation $2$-complex of the group $G_\ell$}
 Write $\widetilde X_\ell$ for its universal cover. A lifted face is a lift of a defining $2$-cell, together with its boundary attachment.
 \mynomenclature[G]{$\widetilde{X}_\ell$}{The universal cover of the presentation $2$-complex $X_\ell$}
 When the lifted faces cover $\widetilde{X}_\ell$, we denote the
	nerve of this cover by $\mathcal{W}_\ell$ and call it the Wise
	complex; see \cite[Definition 7.7]{OsajdaPrytula}. The covering property holds in particular on the event considered in Proposition \ref{prop: lifted face geometry} below.
 
The following proposition collects the properties that we will use throughout the paper. These follow from known results in the literature and we claim no originality. We present them together to specify a single event of asymptotic probability $1$ on which all the required properties hold simultaneously. Note that, as remarked in the introduction of \cite{ArzhantsevaOlshanskii}, the intersection of finitely many properties that hold with asymptotic probability $1$ also holds with asymptotic probability $1$.

\begin{prop}\label{prop: lifted face geometry}
Fix $n\geq 2$ and $m\geq 1$, and consider the few-relator model with $n$ generators and $m$ relators. Then there exist events
$\mathcal{G}_\ell\subseteq\mathcal{C}_\ell^m$ with $\P_\ell(\mathcal{G}_\ell)\to 1$ as $\ell\to\infty$ such that, for every fixed integer $p\geq 7$ and all sufficiently large $\ell$, the following properties hold simultaneously for every presentation in $\mathcal{G}_\ell$. \mynomenclature[G]{$\mathcal{G}_\ell$}{The asymptotically almost sure event in $\mathcal{C}_\ell^m$ satisfying properties \eqref{item: prob 1}-\eqref{item: prob 8} in Proposition \ref{prop: lifted face geometry}}
	\begin{enumerate}
		\item\label{item: prob 1} Every piece has length at most $\rho_\ell$, every letter of
		$\mathcal{A}$ occurs in every defining relator, and the presentation
		satisfies $C''(1/6)$.
		
		\item \label{item: prob 2} The boundary contour of every lifted face is an embedded cycle of length $\ell$. The intersection of two distinct lifted faces is empty or a path, possibly reduced to a vertex, of length at most	$\rho_\ell$.
		
		\item\label{item: prob 3} Every finite family of pairwise-intersecting lifted faces has non-empty total intersection.
		
		\item\label{item: prob 4} The lifted faces cover the universal cover $\widetilde{X}_\ell$ and the associated Wise complex $\mathcal{W}_\ell$ is $p$-systolic.
		
		\item \label{item: prob 5} If $w\in F_n$ is a non-trivial reduced word of length at most $\ell/2$, then $q_\ell(w)\neq e_{G_\ell}$, where $q_\ell:F_n\to G_\ell$ is the quotient homomorphism.
		
		\item \label{item: prob 6} The presentation complex $X_\ell$ is aspherical, and $G_\ell$ is infinite, torsion-free, and Gromov-hyperbolic.
		
		\item \label{item: prob 7} The group $G_\ell$ is virtually compact special. In particular, it is residually finite, virtually RFRS, and satisfies the strong Atiyah conjecture over $\Q$.
		
		\item \label{item: prob 8} If $1\leq m\leq n-2$, the presentation satisfies a minimum condition in the sense of Definition \ref{def: minimum condition}, and $\beta_1^{(2)}(G_\ell)=n-m-1$.
	\end{enumerate}
\end{prop}

\begin{proof}
	Define $\mathcal{G}_\ell$ to be the event where the following conditions hold.	
	\begin{itemize}
		\item $P_\ell\leq\rho_\ell$.
		\item Every letter of $\mathcal{A}$ occurs in every relator.
		\item If $1\leq m\leq n-2$, also ask that the presentation satisfies a minimum condition and $\beta_1^{(2)}(G_\ell)=n-m-1$.
	\end{itemize} 
	
	Lemmas \ref{lem: logarithmic} and \ref{lem: generator present}, together with \cite[Theorem B and Proposition 5.3]{KielakKrophollerWilkes2022} and a union bound give $\P_\ell(\mathcal{G}_\ell)\xrightarrow[\ell\to \infty]{}1$.

Fix $p\geq7$. Since $\rho_\ell=o(\ell)$, for all sufficiently large $\ell$ we have $\rho_\ell<\min\{\ell/6,\ell/p\}$. This implies that the $C''(1/6)$ small-cancellation condition holds. This proves Item \eqref{item: prob 1}.

Next, we apply results of Osajda and Prytu\l{}a \cite{OsajdaPrytula}. To do so, we first subdivide the relator discs as follows. For each defining relator, take an abstract polygon whose oriented boundary is labelled by that relator. Add a vertex in its interior and join it to every boundary vertex, so that the polygon is subdivided into triangles. We then attach the boundary of the polygon to the presentation graph according to its labels, exactly as in the original presentation complex. When the relator graphs are cycles, this is the construction described in \cite[Definition 7.4]{OsajdaPrytula}. This subdivision lifts in an equivariant way to the relator discs in the universal cover $\widetilde{X}_\ell$ of the presentation complex of $G_\ell$.

To verify the hypothesis in \cite[Definitions 6.6 and 7.5]{OsajdaPrytula}, note that a piece in their sense gives two distinct rooted oriented readings. Indeed, the relator type, boundary position, and directed initial edge determine the lifted face uniquely. Its length is therefore at most $\rho_\ell$. Since a nontrivial reduced closed path in a relator circle has length at least $\ell$, a decomposition into fewer than $p$ pieces would imply $\ell\leq(p-1)\rho_\ell<\ell$. This is a contradiction.

Since $\widetilde{X}_\ell$ is simply connected, Items \eqref{item: prob 2} and \eqref{item: prob 3}  follow from \cite[Lemmas 6.10 (1), (2), 6.11 and Remark 7.6]{OsajdaPrytula}. Here, any nonempty intersection of two distinct lifted faces is a finite tree contained in an embedded boundary cycle, hence a path, possibly reduced to a vertex. If this path contains at least one edge, reading it in the two boundary contours gives two distinct rooted oriented readings, as above. Thus, its label is a piece, and its length is at most $\rho_\ell$.

The condition that each letter of $\mathcal{A}$ occurs in every relator ensures that every edge lies on a lifted face. Thus the covering hypothesis of \cite[Definition 7.7]{OsajdaPrytula} holds, and then \cite[Theorem 7.10]{OsajdaPrytula} proves Item \eqref{item: prob 4}.

For Item \eqref{item: prob 5}, suppose that a non-empty freely reduced word $w$ of length at most $\ell/2$ represents the identity in $G_{\ell}$. Its cyclic reduction is still non-empty and has length at most $\ell/2$. Greendlinger's lemma for $C'(1/6)$ presentations gives a subword of this cyclic reduction that is also a subword of a defining relator and has length strictly greater than $\ell/2$. This is a contradiction; see \cite[Chapter V, Theorem 4.4]{LyndonSchupp1977}.
	
Asphericity, torsion-freeness, and hyperbolicity follow from the classical small-cancellation results; see \cite[Section 4.7]{GromovHyperbolic} and \cite[Theorem 13.3]{OlshanskiiDefiningRelations}. By \eqref{item: prob 5}, a generator represents a non-trivial element. Since the group is torsion-free, it is infinite. This proves \eqref{item: prob 6}. Wise's cubulation theorem \cite{Wise2004}, followed by \cite[Theorem 1.1]{Agol2013}, gives virtual compact specialness. Residual finiteness follows from \cite[Theorem 4.4]{HaglundWise2008}. Virtual RFRS follows from the embedding of a special group in a right-angled Artin group and \cite[Corollary 2.3]{Agol2008}. Finally, \cite[Theorem 1.2]{Schreve2014} gives the strong Atiyah	conjecture. This proves Item \eqref{item: prob 7}. Finally, Item \eqref{item: prob 8} is satisfied by the definition of the event $\mathcal{G}_{\ell}$.
\end{proof}
	
	The systolicity of the Wise complex will be used in the next lemma. For each $u,v\in\Z/\ell\Z$, we write $d_{\ell}(u,v)$ for their cyclic distance on $\Z/\ell\Z$.
	
	\begin{lem}\label{lem: close overlap positions}
Fix integers $K\geq3$ and $p\ge\max\{7,K+1\}$. For all sufficiently large $\ell$, consider a presentation in the event $\mathcal{G}_\ell$ of Proposition \ref{prop: lifted face geometry}. Let $\mathcal{S}$ be a family of at most $K$ lifted faces of $\widetilde X_\ell$. Let $E,L,H\in\mathcal{S}$ be distinct faces such that $E\cap L$ and $E\cap H$ contain edges. Suppose that $L$ and $H$ can be joined by a chain of intersecting faces in $\mathcal{S}\backslash\{E\}$. Choose edges $e_L\subseteq E\cap L$ and $e_H\subseteq E\cap H$. Let $u,v\in\Z/\ell\Z$ be the positions of their initial vertices on the oriented contour of $E$. Then $d_\ell(u,v)\leq K \rho_\ell $.
	\end{lem}
	\begin{proof}
		We use Items \eqref{item: prob 2}, \eqref{item: prob 3}, and \eqref{item: prob 4} of Proposition \ref{prop: lifted face geometry}. In particular, we use that the Wise complex $\mathcal{W}_\ell$ is $p$-large.
		
Choose a chain $L=H_0,H_1,\ldots,H_s=H$ of minimal length among all chains in $\mathcal{S}\backslash\{E\}$ joining $L$ to $H$. Thus $H_i\cap H_{i+1}\neq\varnothing$ for every $0\leq i<s$. The minimality of this chain implies that $H_i\cap H_j=\varnothing$ whenever $|i-j|\geq2$, since any such non-empty intersection would give a shorter chain.
		
We claim that $E\cap H_i\neq\varnothing$ for every $0\leq i\leq s$. Indeed, suppose that this is not the case. Since $E$ meets both $H_0=L$ and $H_s=H$, there would exist indices $a<b$, with $b-a\geq2$, such that $E\cap H_a\neq \varnothing$ and $E\cap H_b\neq \varnothing$, but $E\cap H_j=\varnothing$ for every $a+1\le j<b$. The sequence $E,H_a,H_{a+1},\ldots,H_b$ would then form an induced cycle in the $1$-skeleton of the Wise complex $\mathcal{W}_\ell$, of length at least $4$ and at most $s+2\leq|\mathcal{S}|\leq K<p$. This contradicts the fact that \(\mathcal{W}_\ell\) is \(p\)-large. Hence, we conclude that $E\cap H_i\neq\varnothing$ for every $0\leq i\leq s$.
		
For each $0\leq i<s$, the three faces $E,H_i,H_{i+1}$ intersect pairwise. Since we are assuming that the conclusions of Proposition \ref{prop: lifted face geometry} hold, it must be that their total intersection is non-empty. Hence, we may choose $z_i\in E\cap H_i\cap H_{i+1}$.
		
Let $y_0$ and $y_s$ be the initial vertices of $e_L$ and $e_H$, respectively. Let us set $z_{-1}=y_0$ and $z_s=y_s$. For every $0\leq i\leq s$, we have $z_{i-1}, z_i\in E\cap H_i$. Since $E\cap H_i$ is a path of length at most $\rho_\ell$, these two points can be joined inside the contour $\partial E$ by a path of length at most $\rho_\ell$. Concatenating these $s+1$ paths gives a path in $\partial E$ from $y_0$ to $y_s$ of length at most $(s+1)\rho_\ell$. Finally, since $y_0$ and
$y_s$ have positions $u$ and $v$ on $\partial E$, it follows that
		\[
		d_\ell(u,v)\leq (s+1)\rho_\ell\leq (K-1)\rho_\ell \leq K\rho_\ell.
		\]
This finishes the proof.
	\end{proof}

	\subsection{The group von Neumann algebra and its trace}\label{subsection: vn algebras}
	
	Let $G$ be a countable discrete group with identity element $e_G$ and consider the associated Hilbert space $\ell^2(G)$. Denote by $\delta_g\in \ell^2(G)$ the point mass at $g\in G$. The left regular
	representation $\lambda:G\to\mathcal{U}(\ell^2(G))$ is defined by $\lambda(g)\delta_h=\delta_{gh}$ for each $g,h\in G$, where $\mathcal{U}(\ell^2(G))$ denotes the space of unitary operators on $\ell^2(G)$. The group von Neumann algebra of $G$ is given by
	\begin{equation*}
		L(G) \coloneqq \{T\in B(\ell^2(G)):T\lambda(g)=\lambda(g)T
		\text{ for every }g\in G\}.
	\end{equation*}
	Thus, $L(G)$ is the algebra of bounded $G$-equivariant operators on
	$\ell^2(G)$; see \cite[Definitions 1.1 and 1.2]{Luck2002}. The canonical trace on $L(G)$ is given by $\tau_G(T)\coloneqq \langle\delta_e,T\delta_e\rangle$, for each $T\in L(G)$. For every $s\geq1$ the (unnormalized) matrix trace on $M_s(L(G))$ is
	\[
	\mathrm{Tr}_s(T)\coloneqq \sum_{j=1}^s\tau_G(T_{jj})=\sum_{j=1}^s \left\langle \delta_e\otimes e_j,T(\delta_e\otimes e_j)\right\rangle,\]
	for each $T=(T_{ij})_{i,j=1}^s\in M_s(L(G))$. In particular,
	it holds that $\mathrm{Tr}_s(I_s)=s$. The trace is faithful, normal and positive. The definition for positive endomorphisms of Hilbert $L(G)$-modules and these basic properties are given in \cite[Definition 1.8 and Theorem 1.9]{Luck2002}.

We define the von Neumann dimension of a subspace of a Hilbert $L(G)$-module following \cite[Definitions 1.5 and 1.10]{Luck2002}.
\begin{defn}
 Let $V\subseteq\ell^2(G)^s$ be a closed	$G$-invariant subspace and let us denote by $p_V\in M_s(L(G))$ the orthogonal projection onto $V$. The \emph{von Neumann dimension} of $V$ is defined as
	\begin{equation*}\label{eq:vn dimension}
		\dim_{L(G)}\left(V\right)\coloneqq \mathrm{Tr}_s(p_V).
	\end{equation*}
\end{defn}

\subsection{Functional calculus and the tracial spectral measure} \label{subsection: preliminaries tracial spectral measure}
	
	Let $T\in M_s(L(G))$ be bounded and self-adjoint. The spectral theorem
	provides a projection-valued measure $E_T$ on the Borel subsets of
	$\R$ such that $f(T)=\int_{\R}f(t)\,dE_T(t)$ for every bounded Borel function $f$ on the spectrum of $T$. We refer to \cite[Section 1.4, Equations (1.64)-(1.67), and Definition 1.68]{Luck2002} for further details on Borel functional calculus, and also to \cite[Chapter VII, Sections 1-3]{ReedSimon1980}. 
 
 Applying the above to the positive operator $\widehat{\Delta}_{2,\ell}=\frac{1}{\ell} \Delta_{2,\ell}$ we obtain
 \mynomenclature[C]{$\widehat{\Delta}_{2,\ell}$}{The normalized Laplacian $\widehat{\Delta}_{2,\ell}=\frac{1}{\ell}\Delta_{2,\ell}$}
 its
	\emph{tracial spectral measure with respect to the unnormalized matrix
		trace} as the finite positive Borel measure $\sigma_{\ell}$ defined by
	\begin{equation*}
		\sigma_\ell(B)
		\coloneqq \mathrm{Tr}_m
		\bigl(E_{\widehat{\Delta}_{2,\ell}}(B)\bigr)
		=\mathrm{Tr}_m
		\left(\mathds{1}_B(\widehat\Delta_{2,\ell})\right),
		\text{ for every Borel subset } B\subseteq\R.
	\end{equation*}
 \mynomenclature[C]{$\sigma_\ell$}{The tracial spectral measure of $\widehat{\Delta}_{2,\ell}$}
	Normality of $\mathrm{Tr}_m$ implies that $\sigma_{\ell}$ is countably additive. Moreover, we have that $\mathrm{supp}(\sigma_\ell)\subseteq\left[0,\|\widehat{\Delta}_{2,\ell}\|\right]$ and $\sigma_\ell(\R)=\mathrm{Tr}_m(I_m)=m$.
 
	For every bounded Borel function $f$ defined on the spectrum of $\widehat{\Delta}_{2,\ell}$, we have
	\begin{equation*}
		\mathrm{Tr}_m\bigl(f(\widehat\Delta_{2,\ell})\bigr)
		=\int_{\R}f(t)\,d\sigma_\ell(t).
	\end{equation*}
	In particular, for every integer $k\geq0$, we have that
	\begin{equation*}\label{eq: spectral moments}
		\int_{\R}t^k\,d\sigma_\ell(t)
		=\mathrm{Tr}_m(\widehat\Delta_{2,\ell}^{k}).
	\end{equation*}

\subsection{Cellular \texorpdfstring{$\ell^2$}{l2}-chains and \texorpdfstring{$\ell^2$}{l2}-Betti numbers}\label{subsection: preliminaries homology}
	We refer to \cite[Definitions 1.15 and 1.16]{Luck2002} for the definitions of Hilbert chain complexes and reduced $\ell^2$-homology. Below, we recall the basic terminology and results needed for our paper.
	
The action of $G_{\ell}$ on $\widetilde{X}_{\ell}$ by deck transformations makes the cellular chain groups $C_p(\widetilde{X}_\ell;\Z)$ based free $\Z G_\ell$-modules. We use the completed cellular chain complex
	\begin{equation*}
		C_p^{(2)}(\widetilde X_\ell)
		\coloneqq \ell^2(G_\ell)\otimes_{\Z G_\ell}
		C_p(\widetilde X_\ell;\Z).
	\end{equation*}

After choosing one oriented lift of each cell of $X_\ell$, the elementary tensors $\delta_g\otimes e$ form an orthonormal cellular basis, and one completes in the resulting Hilbert norm. Thus we have the identifications
	\begin{equation*}
		C_2^{(2)}(\widetilde{X}_\ell)\cong\ell^2(G_\ell)^m, \ \ 
		C_1^{(2)}(\widetilde X_\ell)\cong\ell^2(G_\ell)^n, \ \ \text{ and }\ \ \ C_0^{(2)}(\widetilde X_\ell)\cong\ell^2(G_\ell);
	\end{equation*}
	see \cite[Section 1.22]{Luck2002} and \cite[Definition 1.29]{Luck2002}. The reduced cellular $\ell^2$-homology of $\widetilde{X}_\ell$ is defined by
	\[
	\overline{H}_p^{(2)}(\widetilde{X}_\ell)	\coloneqq	\ker\left(\partial_p^{(2)}\right)\Big/ \overline{\mathrm{im}\left(\partial_{p+1}^{(2)}\right)},
	\]
	where the closure is taken in $C_p^{(2)}(\widetilde{X}_\ell)$. Its induced $G_\ell$-action makes it a finitely generated Hilbert $L(G_\ell)$-module. The $p$-th $\ell^2$-Betti number is then defined as
	\[
	\beta_p^{(2)}(X_\ell)\coloneqq\dim_{L(G_\ell)}\left(\overline{H}_p^{(2)}(\widetilde{X}_\ell)\right);
	\]
	see \cite[Definitions 1.16 and 1.30]{Luck2002}.
	
	For a discrete group $G$, its $\ell^2$-Betti numbers are defined using a contractible free $G$-CW-complex $EG$:
	\[
	\beta_p^{(2)}(G) \coloneqq \beta_p^{(2)}(EG;L(G)).
	\]
	
	If the classifying space $BG=G\backslash EG$ has finitely many cells	in each degree, the $\ell^2$-Betti numbers are computed by the Hilbert chain construction above; see \cite[Lemma 6.53]{Luck2002}. In this case the definition also appears in \cite[Example 1.38]{Luck2002}, and its independence of the chosen model follows from \cite[Theorem 1.35(1)]{Luck2002}. In the generic event where $X_\ell$ is aspherical, its universal cover is contractible and is a model for $EG_\ell$.
	Hence $\beta_p^{(2)}(G_\ell)=\beta_p^{(2)}(X_\ell)$ for each $p\ge 0$.
 \subsection{The Laplacian and the Euler-Poincaré formula}\label{subsection: laplacian}
	We now recall how the degree-two $\ell^2$-boundary is computed. Denote by $F_n$ the free group with free basis $\{x_1,\ldots, x_n\}$. Let $\varepsilon:\Z F_n\to\Z$ be the augmentation map. The Fox derivative $\partial/\partial x_i:\Z F_n\to\Z F_n$ is the
	unique $\Z$-linear map satisfying
	\begin{align*}
		\frac{\partial x_j}{\partial x_i}&=\delta_{ij},\ \ \ \text{ and }\ \ \
		\frac{\partial(uv)}{\partial x_i}
		=\frac{\partial u}{\partial x_i}\varepsilon(v)
		+u\frac{\partial v}{\partial x_i}
	\end{align*}
	for each $i,j=1,\ldots,n$ and every $u,v\in \Z F_n$. For group elements $u,v\in F_n$, the second formula above becomes $\partial(uv)/\partial x_i=\partial u/\partial x_i
	+u\,\partial v/\partial x_i$, so that one has $\frac{\partial x_j^{-1}}{\partial x_i}=-\delta_{ij}x_j^{-1}$.
	The fundamental formula is
	\begin{equation*}
		w-1=\sum_{i=1}^n
		\frac{\partial w}{\partial x_i}(x_i-1),
	\end{equation*}
	for each $w\in F_n$. These are Fox's formulas \cite[Section 1 formulas (1.1)--(1.2), and Section 2 formula (2.3)]{Fox1953}; see also \cite[Chapter II, Section 2]{Brown1994}.
	
	Let $q_\ell:\Z F_n\to\Z G_\ell$
 \mynomenclature[C]{$q_\ell$}{The linear extension $q_\ell\colon \Z F_n\to \Z G_\ell$ of the canonical quotient map $F_n\to G_{\ell}$}
 be the linear extension of the canonical quotient map $F_n\to G_{\ell}$. In the
	chosen cellular bases, the completed degree-two boundary map is the bounded
	$G_\ell$-equivariant operator $A_\ell=\partial_2^{(2)}:\ell^2(G_\ell)^m\to\ell^2(G_\ell)^n$ given by
	\begin{equation*}A_\ell=\left(q_\ell\!\left(\frac{\partial r_{j}}{\partial x_i}\right)\right)_
		{1\leq i\leq n,\ 1\leq j\leq m}.
	\end{equation*}
 \mynomenclature[C]{$A_\ell$}{The Fox matrix associated with $G_\ell$}
	Each entry in the above matrix acts through the right $\Z G_\ell$-action on $\ell^2(G_{\ell})$, and hence as a bounded operator commuting with $\lambda(G_\ell)$. Consequently $A_\ell\in M_{n\times m}(L(G_\ell))$ and therefore $A_\ell^*A_\ell\in M_m(L(G_\ell))$. The cellular interpretation of the Fox matrix is explained in \cite[Chapter II, Section 2]{Brown1994}.
	
For the completed cellular chain complex, the degree-$p$ Laplacian is given by
	\[
	\Delta_{p,\ell} = \partial_{p+1}^{(2)}\left(\partial_{p+1}^{(2)}\right)^*	+	\left(\partial_p^{(2)}\right)^*\partial_p^{(2)};
	\]
	see \cite[Equation (1.17)]{Luck2002}. For every $v\in C_p^{(2)}(\widetilde{X}_\ell)$, we have $\langle\Delta_{p,\ell}v,v\rangle=\|(\partial_{p+1}^{(2)})^*v\|^2+\|\partial_p^{(2)}v\|^2$, so that
	\[
	\ker\left(\Delta_{p,\ell}\right)=	
\ker\left(\partial_p^{(2)}\right)\cap\left(\overline{\mathrm{im}\left(\partial_{p+1}^{(2)}\right)}\right)^\perp.	
 \]
Hence, we obtain the identification $\overline{H}_p^{(2)}(\widetilde{X}_\ell) \cong \ker\Delta_{p,\ell}$; see \cite[Lemma 1.18]{Luck2002}. Since $X_\ell$ is two-dimensional, we have that $C_3^{(2)}(\widetilde X_\ell)=0$. Therefore $\Delta_{2,\ell}=A_\ell^*A_\ell$ and
	\[
	\overline{H}_2^{(2)}(\widetilde{X}_\ell)=\ker\left(A_\ell\right)=\ker\left(A_\ell^*A_\ell\right).
	\]
 \mynomenclature[C]{$\Delta_{2,\ell}$}{The Hodge Laplacian $\Delta_{2,\ell}=A_\ell^*A_\ell$ associated to $G_\ell$}
	The spectral projection $\mathds{1}_{\{0\}}(A_\ell^*A_\ell)$ is the orthogonal projection onto this kernel. Consequently, we can write
	\[
	\beta_2^{(2)}(X_\ell)=\mathrm{Tr}_m\left(	\mathds{1}_{\{0\}}(A_\ell^*A_\ell)\right).\]
	
Next, we explain the Euler-Poincaré identity for this complex.
 
 Set $s_{p,\ell} \coloneqq\dim_{L(G_\ell)}\left(\overline{\operatorname{im}\left(\partial_p^{(2)}\right)}\right)$. Additivity of von Neumann dimension for weakly exact sequences \cite[Theorem 1.12(2)]{Luck2002} gives
\[
	\dim_{L(G_\ell)}\left(C_p^{(2)}(\widetilde{X}_\ell)\right)=\dim_{L(G_\ell)}\left(\ker\left(\partial_p^{(2)}\right)\right)+s_{p,\ell}=	\beta_p^{(2)}(X_\ell)+s_{p+1,\ell}+s_{p,\ell}.
\]
	Since $s_{0,\ell}=s_{3,\ell}=0$, we have $1=\beta_0^{(2)}(X_\ell)+s_{1,\ell}$, $n=\beta_1^{(2)}(X_\ell)+s_{1,\ell}+s_{2,\ell}$ and $m=\beta_2^{(2)}(X_\ell)+s_{2,\ell}$. Taking the alternating sum gives
	\[
	1-n+m=\beta_0^{(2)}(X_\ell) -\beta_1^{(2)}(X_\ell) +\beta_2^{(2)}(X_\ell);
	\]
	see also \cite[Theorem 1.35(2)]{Luck2002}. The left-hand side in the equation above is the Euler characteristic $\chi(X_\ell)$. 
 
On the event that $G_\ell$ is infinite we have $\beta_0^{(2)}(X_\ell)=0$ \cite[Theorem 1.35(8)]{Luck2002}, and if $X_\ell$ is aspherical we may replace the cellular $\ell^2$-Betti numbers by the $\ell^2$-Betti numbers of the group. Rearranging the preceding identity gives
	\[
	\beta_1^{(2)}(G_\ell)=n-m-1+\mathrm{Tr}_m\left(\mathds{1}_{\{0\}}(A_\ell^*A_\ell)\right),
	\]
 which is the Euler-Poincaré formula.

	\subsection{The free-Poisson distribution} \label{subsection:free poisson preliminaries}
	
	We recall the facts about the free-Poisson distribution that will be used below. Our convention is that the free-Poisson distribution with parameter $\lambda\geq 0$ has mean $\lambda$. For further background on the free-Poisson distribution and its connection with non-crossing partitions and random matrices, we refer to Lectures 9 and 12 in \cite{NicaSpeicher2006}, and particularly to \cite[Proposition 12.11 and Definition 12.12]{NicaSpeicher2006}. We also refer to the foundational monograph \cite{VoiculescuDykemaNica}, and to \cite[Chapters 1-3]{MingoSpeicher2017}.
	
	\begin{defn}
		The \emph{free-Poisson distribution} with parameter $\lambda\geq 0$ is the
		probability measure $\mu_\lambda$ on $[0,\infty)$ given by the density
		\begin{equation*}
			d\mu_\lambda(t) =
			(1-\lambda)_+\,d\delta_0(t)
			+
			\frac{\sqrt{(b_\lambda-t)(t-a_\lambda)}}{2\pi t}
			\mathds{1}_{[a_\lambda,b_\lambda]}(t)\ dt,
		\end{equation*}
		where $a_\lambda\coloneqq (1-\sqrt{\lambda})^2$, 	$b_\lambda\coloneqq(1+\sqrt{\lambda})^2$, and $(1-\lambda)_+\coloneqq\max\{1-\lambda,0\}$.
	\end{defn}
    
	See also \cite[Equation (12.15) and Definition 12.12]{NicaSpeicher2006}. \footnote{ Note that there is a typo in \cite[Equation (12.14)]{NicaSpeicher2006}. The factor $\lambda$ multiplying $\widetilde{\nu}$ in the case $0\leq\lambda\leq1$ should be omitted; see Speicher's errata for the book \url{https://rolandspeicher.com/literature/nica-speicher/}.}
	
	As we explain in the next remark, the free-Poisson distribution appears naturally when studying the limits of random matrices. 
	\begin{rem}\label{rem: comparison random matrix theorem}
Let $(x_{ij})_{i,j\geq1}$ be independent and identically
distributed real random variables with mean zero and variance one. Let $(M_N)_{N\geq1}$ be a sequence of positive integers such that $M_N/N\to\lambda>0$, and set
\[
X_N=(x_{ij})_{1\leq i\leq N,\ 1\leq j\leq M_N}.
\]
For a self-adjoint $r\times r$ matrix $B$, denote by $\theta_j(B)$, $j=1,\ldots,r$ its eigenvalues counted with multiplicity, and consider its empirical spectral measure
\[
\mu_B\coloneqq\frac{1}{r}\sum_{j=1}^r\delta_{\theta_j(B)}.
\]
The Marchenko--Pastur theorem states that almost surely we have $\mu_{N^{-1}X_NX_N^*}
\xrightarrow[N\to\infty]{\mathrm{weakly}}
\mu_\lambda$,
where $\mu_\lambda$ is the free-Poisson law of mean $\lambda$; see \cite{MarchenkoPastur1967} and \cite[Theorem 3.6]{BaiSilverstein}.

An alternative normalization uses $M_N^{-1}X_NX_N^*$, whose empirical spectral measure converges
almost surely to the mean-one Marchenko--Pastur law
$\rho_y$, with $y=1/\lambda$, given by
\[
d\rho_y(t) = \left(1-\frac{1}{y}\right)_+d\delta_0(t)
+ \frac{\sqrt{(b_y-t)(t-a_y)}}{2\pi yt}
\mathds{1}_{[a_y,b_y]}(t)\ dt,
\]
where $a_y=(1-\sqrt y)^2$ and $b_y=(1+\sqrt y)^2$.
Thus, if $Y$ has distribution $\rho_{1/\lambda}$, then
$\lambda Y$ has distribution $\mu_\lambda$. In particular,
\[
\int_{\R}t^k\ d\mu_\lambda(t)
=
\lambda^k\int_{\R}t^k\ d\rho_{1/\lambda}(t),
\]
for each $k\ge 0$; see \cite[Remark 2.12 and Exercise 2.12]{MingoSpeicher2017} for this comparison.

To relate this to Theorem \ref{thm:main spectral law},
note that both matrices $X_N^*X_N$ and $X_NX_N^*$ have the same non-zero eigenvalues with multiplicity. Consequently,
\[
\mu_{N^{-1}X_N^*X_N}
= \frac{N}{M_N}\mu_{N^{-1}X_NX_N^*} + \left(1-\frac{N}{M_N}\right)\delta_0.
\]
Taking $\lambda=m/(n-1)$ therefore gives
\[
m \mu_{N^{-1}X_N^*X_N}
\xrightarrow[N\to\infty]{\mathrm{weakly}}
(n-1)\mu_{m/(n-1)}+(m-n+1)\delta_0 \ \ \ \text{almost surely}.
\]

In our setting, the operator $A_\ell/\sqrt{\ell}$ maps
$\ell^2(G_\ell)^m$ into $\ker\left(\partial_{1,\ell}^{(2)}\right)$. When $G_\ell$ is infinite, these spaces have von Neumann dimensions $m$ and $n-1$, respectively. This motivates both the parameter $\lambda=m/(n-1)$ and the form of the limiting measure in Theorem \ref{thm:main spectral law}.
\end{rem}

	\subsection{Non-crossing partitions}
	Next, we recall non-crossing partitions. We use as main reference \cite[Definition 9.1]{NicaSpeicher2006}. 
 
 For a partition $\pi$ of a set $X$, we write $x\sim_\pi y$ if the elements $x,y\in X$ belong to the same atom of $\pi$. Write $x\not\sim_\pi y$ otherwise.
	
	\begin{defn}\label{defn: noncrossing partitions}
		Let $k\geq1$. A partition $\pi\in\mathcal{P}(k)$ of $\{0,\ldots,k-1\}$ is \emph{crossing} if there exist $0\leq i<j<r<s\leq k-1$ such that $i\sim_\pi r$, $j\sim_\pi s$, and $i\not\sim_\pi j$. The partition is \emph{non-crossing} if no such four indices exist. We denote the set of non-crossing partitions by $\mathrm{NC}(k)$.
 \mynomenclature[P]{$\mathrm{NC}(k)$}{The subset of all non-crossing partitions in $\mathcal{P}(k)$}
 If $\pi\in\mathrm{NC}(k)$, we write $|\pi|$ for the number of equivalence classes of the partition $\pi$.
	\end{defn}
	
	Non-crossing partitions appear naturally when computing the moments of the free-Poisson distribution. The following proposition is a standard result; see, e.g., Proposition 1.1 and the paragraph following its proof in \cite{OraveczPetz1997}, and Exercise 11 in Section 2 of \cite{MingoSpeicher2017}.
	\begin{prop}\label{prop:free poisson moments}
		Let $\lambda>0$ and consider the free-Poisson distribution $\mu_{\lambda}$ with parameter $\lambda$. Then, for every $k\geq1$ we have
		\[
		\int_{\R}t^k\,d\mu_\lambda(t) = \sum_{\pi\in\mathrm{NC}(k)} \lambda^{|\pi|}=\sum_{j=1}^k \frac{1}{k} \binom{k}{j}\binom{k}{j-1}\lambda^j.
		\]
	\end{prop}
	 
\section{The trace expansion of the Laplacian} \label{section: trace laplacian}
	
Recall from Subsection \ref{section: probability small cancellation} that we denote by $G_\ell$ the group with $n$ generators and $m$ relators of length $\ell$ chosen uniformly on the set of cyclically reduced words of length $\ell$, all independent from each other. We also denote its presentation complex by $X_\ell$ and the associated universal cover by $\widetilde{X}_\ell$.
	
	\subsection{Expansion of the moments of the Laplacian}\label{subsection: expansion moments laplacian}
	
	Let us denote by $\mathcal{F}_\ell$
 \mynomenclature[C]{$\mathcal{F}_\ell$}{The set of $2$-cells of the universal cover $\widetilde{X}_\ell$ of the presentation $2$-complex $X_\ell$} the set of $2$-cells of $\widetilde{X}_\ell$. Choose one lift $F^0_{j}$
 \mynomenclature[C]{$\mathcal{F}_\ell^0$}{Fundamental domain for the action $G_\ell\curvearrowright \mathcal{F}_\ell$, containing one lift of the $2$-cell attached to the relator $r_j$ in $X_\ell$ for each $1\leq j\leq m$} 
 of the $2$-cell attached by $r_{j}$ to $X_{\ell}$, for each $1\leq j\leq m$, and put $\mathcal{F}_\ell^0 \coloneqq \{F^0_{1},\ldots,F^0_{m}\}$. Every $F\in\mathcal{F}_\ell$ belongs to a unique $G_\ell$-orbit represented by one of $F_1^0,\ldots,F_m^0$. We will call this the \emph{type} of $F$ and denote it by $\mathrm{type}(F)\in\{1,\ldots,m\}$.
 \mynomenclature[C]{$\mathrm{type}(F)$}{The number $j\in\{1,\dots,m\}$ such that $F\in \widetilde{X}_\ell$ is the lift of the $2$-cell attached to the relator $r_j$ in $X_\ell$}
	Let $u_F\in C_2^{(2)}(\widetilde X_\ell)$ be the unit cellular basis vector
	corresponding to $F$,
 \mynomenclature[C]{$u_F$}{Unit cellular basis vector in $C_2^{(2)}(\widetilde X_\ell)$ corresponding to $F\in \mathcal{F}_\ell$}
 and define
	\begin{equation}\label{eq: bF}
	 	b_F\coloneqq A_\ell u_F\in C_1^{(2)}(\widetilde X_\ell),
	\end{equation} \mynomenclature[C]{$b_F$}{The signed cellular boundary of $F\in \mathcal{F}_\ell$, i.e., $b_F=A_\ell u_F$}
 and
 \begin{equation}\label{eq: h ell}
 h_\ell(F,H)\coloneqq \langle b_F,b_H\rangle \ \ \ \text{ for each } F,H\in\mathcal{F}_\ell.
 \end{equation}
 
 \mynomenclature[C]{$h_\ell(F,H)$}{The value $\langle b_F,b_H\rangle=\langle u_F,\Delta_{2,\ell} u_H\rangle$}
	Note that we have for all
	$F,H\in\mathcal{F}_\ell$ the equality $\langle u_F,\Delta_{2,\ell} u_H\rangle=h_\ell(F,H)$.

	\begin{lem}
		For every $F\in\mathcal{F}_\ell$, the vector $b_F$ defined in Equation \eqref{eq: bF} is the signed cellular boundary of $F$.
	\end{lem}
	\begin{proof}
		Let $F\in\mathcal{F}_\ell$ and denote $j=\mathrm{type}(F)$. By $G_\ell$-equivariance, it suffices to prove the statement for the representative $F^0_{j}$. 
		
		Let us write $r_{j}=s_0s_1\cdots s_{\ell-1}$ as a word of length $\ell$ in $\{x_1,\ldots, x_n\}^{\pm 1}$. Denote $w_0=1$ and $w_p=s_0s_1\cdots s_{p-1}$ for each $1\le p\le \ell$. Repeated use of the Fox product rule gives the equality
		\[
		\frac{\partial r_{j}}{\partial x_i} =\sum_{\{p:s_p=x_i\}}w_p -\sum_{\{p:s_p=x_i^{-1}\}}w_{p+1};
		\]
		see also \cite[Section 2, Equation (2.8)]{Fox1953}. The term $w_p$ is the
		initial vertex of the positively traversed $x_i$-edge at position $p$.
		If $s_p=x_i^{-1}$, the term $-w_{p+1}$ represents the same geometric edge
		with the opposite orientation. Hence the $j$-th Fox column is exactly the signed cellular boundary of $F^0_{j}$. Equivariance proves the same assertion for every lifted face $F$.
	\end{proof}
	
	\begin{lem} \label{lem: trace expansion}
		For each $\ell\ge 1$, let $\Delta_{2,\ell}$ be the degree-$2$ Laplacian of the cellular $\ell^2$-chain complex of the universal cover $\widetilde{X}_\ell$ of the presentation complex of $G_{\ell}$. Denote by $h_{\ell}:\mathcal{F}_{\ell}\times \mathcal{F}_{\ell}\to \R$ be the map defined in Equation \eqref{eq: h ell}. Then, for every $k\ge 1$, we have
		\[
		\mathrm{Tr}_m\left(\Delta_{2,\ell}^k\right)=\sum_{F_0\in\mathcal{F}_\ell^0}
		\ \sum_{F_1,\ldots,F_{k-1}\in\mathcal{F}_\ell}
		\prod_{i=0}^{k-1}h_\ell(F_i,F_{i+1}),
		\]
		with the convention $F_k\coloneqq F_0$. For each fixed $F_0$ in the first sum, only finitely many tuples in the second sum have a non-zero term.
	\end{lem}
	
	\begin{proof}
		Under the identification
		$C_2^{(2)}(\widetilde X_\ell)\cong\ell^2(G_\ell)^m$, the vectors
		$u_{F^0_{j}}$, $1\le j\le m$ are exactly the vectors $\delta_e\otimes e_j$ occurring in the definition of the matrix trace. Therefore, we have
		\[\mathrm{Tr}_m(\Delta_{2,\ell}^k)
		=\sum_{F_0\in\mathcal{F}_\ell^0} \langle u_{F_0},\Delta_{2,\ell}^ku_{F_0}\rangle. \]

		Next, for each fixed $F_0$ we can write
		\begin{align*}
			\langle u_{F_0},\Delta_{2,\ell}^ku_{F_0}\rangle
			&={}
			\sum_{F_1,\ldots,F_{k-1}\in\mathcal{F}_\ell}
			\langle u_{F_0},\Delta_{2,\ell} u_{F_1}\rangle
			\langle u_{F_1},\Delta_{2,\ell} u_{F_2}\rangle\cdots \\
			&\hspace{35mm}\cdots
			\langle u_{F_{k-2}},\Delta_{2,\ell} u_{F_{k-1}}\rangle
			\langle u_{F_{k-1}},\Delta_{2,\ell} u_{F_0}\rangle.
		\end{align*}
		The last two equations prove the first statement of the lemma. 
		
		Finally, let us explain the second sentence in the statement of the lemma. Every $b_F$ is supported on finitely many lifted $1$-cells. Moreover, each
		lifted $1$-cell occurs in the boundaries of only finitely many lifted $2$-cells, because $X_\ell$ has finitely many attaching maps and each has finite length. Thus, for each fixed $F$, only finitely many $2$-cells $H$ satisfy	$h_\ell(F,H)\neq0$. Starting at the fixed face $F_0$, there are therefore only finitely many non-zero choices at each of the $k-1$ steps.
	\end{proof}

	\subsection{Partitions and reindexation of the trace expansion of the Laplacian}\label{subsection: partitions and reindexation}
	
	Fix $k\geq1$, and let us regard the index set $\{0,1,\ldots,k-1\}$ cyclically, so that $k$ is identified with $0$. Let $\mathcal{P}(k)$ be the set of all set partitions of this index set.
 \mynomenclature[P]{$\mathcal{P}(k)$}{Collection of partitions of the set $\{0,\dots,k-1\}$}
 We regard each $\pi\in\mathcal{P}(k)$ as an
	equivalence relation. For $i\in\{0,\ldots,k-1\}$, write $[i]_\pi\coloneqq \{j\in\{0,\ldots,k-1\}:j\sim_\pi i\}$ for the $\pi$-equivalence class of $i$.
 \mynomenclature[P]{$[i]_\pi$}{The $\pi$-equivalence class of an element $i\in \{0,\dots,k-1\}$, given $\pi\in \mathcal{P}(k)$}
 We also regard $\pi$ itself as the collection of its equivalence classes, so that notation such as $C\in\pi$ means that $C$ is one of these classes.

	Given a tuple $(F_0,\ldots,F_{k-1})$ of faces, its \emph{equality partition} is the partition $\pi\in \mathcal{P}(k)$ defined by $i\sim_\pi j $ if and only if $F_i=F_j$, for each $0\le i,j<k$. Thus, the equivalence classes of $\pi$ are in bijection with the distinct lifted faces appearing in the tuple. Each class $C\in\pi$ has a relator type ${\typemap}(C)\coloneqq \mathrm{type}(F_i)$ for $i\in C$.

	\begin{defn}\label{defn: omega ell pi j}
		Fix a partition $\pi\in\mathcal{P}(k)$ and a type map ${\typemap}:\pi\to\{1,\ldots,m\}$.
 \mynomenclature[P]{${\typemap}$}{A map ${\typemap}\colon \pi\to \{1,\dots,m\}$}
 Let us denote by $\Omega_\ell(\pi,{\typemap})$
 \mynomenclature[P]{$\Omega_\ell(\pi,{\typemap})$}{Tuples $\mathbf{F}=(F_C)_{C\in \pi}$ of elements from $\mathcal{F}_\ell$ such that (i) $F_C=F_D$ if and only if $C=D$, (ii) ${\typemap}(C)=\mathrm{type}(F_C)$, (iii) $F_{[0]_\pi}=F_{{\typemap}([0]_\pi)}^0$, and (iv) $h_\ell(F_{[i]_\pi},F_{[i+1]_\pi})\neq 0$}
 the set of families of lifted faces $\mathbf{F}=(F_C)_{C\in\pi}\in\left(\mathcal{F}_\ell\right)^{\pi}$ that satisfy the following three conditions.
		\begin{enumerate}
			\item $\mathrm{type}(F_C)={\typemap}(C)$ for every class $C\in\pi$.
			\item $F_C\neq F_D$ whenever $C,D\in\pi$ and $C\neq D$.
			\item $F_{[0]_\pi}=F^0_{{\typemap}([0]_\pi)}$, and
			$h_\ell(F_{[i]_\pi},F_{[i+1]_\pi})\neq0$ for every $i\in\{0,\ldots,k-1\}$.
		\end{enumerate}
	\end{defn}
	The first two conditions ensure that the equality pattern is exactly $\pi$ and that the prescribed relator types are respected. The third condition fixes the representative required by the von Neumann trace and omits families whose product is zero. Define
	\begin{equation}\label{eq: def of Z}
		Z_\ell(\pi,{\typemap}) \coloneqq \sum_{\mathbf{F}\in\Omega_\ell(\pi,{\typemap})} \prod_{i=0}^{k-1} h_\ell(F_{[i]_\pi},F_{[i+1]_\pi}).
	\end{equation}
 \mynomenclature[P]{$Z_\ell(\pi,{\typemap})$}{The sum over all $\mathbf{F}=(F_C)_{C\in \pi}$ of the (non-zero) product $\prod_{i=0}^{k-1}h_\ell(F_{[i]_\pi},F_{[i+1]_\pi})$}
	Note that the sum is finite thanks to Lemma \ref{lem: trace expansion}.

	\begin{lem}\label{lem: partition trace}
	Let $\ell\ge 1$ and let $\Delta_{2,\ell}$ be the degree-$2$ Laplacian of the cellular $\ell^2$-chain complex of the universal cover $\widetilde{X}_\ell$ of the presentation complex of $G_{\ell}$. Then for every $k\geq1$ we have
		\begin{equation*}
			\mathrm{Tr}_m \left(\Delta_{2,\ell}^{\,k}\right) =\sum_{\pi\in\mathcal{P}(k)} \ \sum_{{\typemap}:\pi\to\{1,\ldots,m\}} Z_\ell(\pi,{\typemap}),
		\end{equation*}
 where for each $\pi\in \mathcal{P}(k)$ and $\typemap:\pi\to \{1,\ldots,m\}$ the value $Z_\ell(\pi,\typemap)$ is defined as in Equation \eqref{eq: def of Z}.
	\end{lem}
	
	\begin{proof} Thanks to Lemma \ref{lem: trace expansion}, we have that \[	\mathrm{Tr}_m\left(\Delta_{2,\ell}^k\right)=\sum_{F_0\in\mathcal{F}_\ell^0}
		\ \sum_{F_1,\ldots,F_{k-1}\in\mathcal{F}_\ell}
		\prod_{i=0}^{k-1}h_\ell(F_i,F_{i+1}).\]
		Given a tuple $(F_0,\ldots,F_{k-1})$ contributing a non-zero summand, consider its equality partition $\pi\in \mathcal{P}(k)$ as described in the paragraph preceding the lemma, as well as the type map ${\typemap}:\pi\to \{1,\ldots, m\}$. For each class $C\in\pi$, let $F_C$ be the common value of $F_i$ for $i\in C$. Different classes give different lifted faces, $F_{[0]_\pi}=F_0$ is the prescribed orbit representative, and every overlap factor is non-zero. Hence $(F_C)_{C\in\pi}\in\Omega_\ell(\pi,{\typemap})$.
		
		Conversely, suppose that $\pi$, ${\typemap}$, and $(F_C)_{C\in\pi}\in\Omega_\ell(\pi,{\typemap})$ are given. Define $F_i\coloneqq F_{[i]_\pi}$ for each $0\leq i\leq k-1$. The pairwise-distinctness condition in the definition of $\Omega_\ell$ ensures that the equality partition of the resulting tuple coincides with $\pi$. The representative condition ensures that $F_0\in\mathcal{F}_\ell^0$, and the corresponding trace summand is exactly the product above.
		
		We conclude that the non-zero face tuples that appear in the preceding sum decompose into a disjoint union indexed by $(\pi,{\typemap})$. This proves the statement of the lemma.
	\end{proof}
	
	\subsection{The off-diagonal-connections graph}
	\begin{defn}\label{defn: graph off diagonal connections}
 	Let $k\ge 1$ and let $\pi\in \mathcal{P}(k)$ be a partition of $\{0,\ldots,k-1\}$. Recall that we regard the latter indices cyclically as elements of $\Z/k\Z$.
	 \begin{itemize}
	 \item We define the set $\mathcal{O}_\pi$ of \emph{off-diagonal indices} of $\pi$ as \mynomenclature[P]{$\mathcal{O}_\pi$}{The subset of off-diagonal indices in $\{0,\dots,k-1\}$}
 \[\mathcal{O}_\pi\coloneqq \{ i\in \{0,\ldots,k-1\}: [i]_{\pi}\neq[i+1]_{\pi}\}.\]
 \item We denote by $\mathrm{diag}(\pi)\coloneqq k-|\mathcal{O}_{\pi}|$ the amount of diagonal indices.
 \mynomenclature[P]{$\mathrm{diag}(\pi)$}{The number of diagonal indices in $\{0,\dots,k-1\}$}
\item We define the \emph{off-diagonal connections graph} $\Gamma_{\pi}=(V_\pi,E_\pi)$
 \mynomenclature[P]{$\Gamma_\pi$}{The off-diagonal connections graph, where $V(\Gamma_\pi)=\pi$, and $\{[i]_\pi,[i+1]_\pi\}$ belongs to $E(\Gamma_\pi)$ for every $i\in \mathcal{O}_\pi$}
 as the simple undirected graph with vertex set $V_\pi\coloneqq\pi$, i.e., the equivalence classes of $\pi$, and where we say that for each $i\in \mathcal{O}_{\pi}$ the edge $\{[i]_\pi,[i+1]_\pi\}$ belongs to $E_\pi$. 
 
 Let us also define for each $e\in E_{\pi}$ its \emph{multiplicity} $\mu_\pi(e)\coloneqq \left|\left\{i\in \mathcal{O}_\pi: \{[i]_{\pi},[i+1]_\pi \}=e\right\} \right|$,
 \mynomenclature[P]{$\mu_\pi(e)$}{The multiplicity $\left|\left\{i\in \mathcal{O}_\pi: \{[i]_{\pi},[i+1]_\pi \}=e\right\} \right|$ of an edge $e\in E(\Gamma_\pi)$}
 so that $|\mathcal{O}_{\pi}|=\sum_{e\in E_{\pi}} \mu_\pi(e)$.
	 \end{itemize}

	\end{defn}
	
	\begin{figure}[ht]
		\centering

\tikzset{every picture/.style={line width=0.75pt}} 

\begin{tikzpicture}[x=0.75pt,y=0.75pt,yscale=-1,xscale=1]

\draw [color={rgb, 255:red, 0; green, 0; blue, 0 } ,draw opacity=0.8 ] (419.52,189.92) -- (518.38,90.94) ;
\draw [color={rgb, 255:red, 0; green, 0; blue, 0 } ,draw opacity=0.8 ] (416.53,188.89) -- (416.71,90.9) ;
\draw [fill={rgb, 255:red, 0; green, 0; blue, 0 } ,fill opacity=1 ] (231,29.72) .. controls (231,28.53) and (231.97,27.56) .. (233.17,27.56) .. controls (234.36,27.56) and (235.33,28.53) .. (235.33,29.72) .. controls (235.33,30.92) and (234.36,31.89) .. (233.17,31.89) .. controls (231.97,31.89) and (231,30.92) .. (231,29.72) -- cycle ;
\draw [fill={rgb, 255:red, 0; green, 0; blue, 0 } ,fill opacity=1 ] (278,63.72) .. controls (278,62.53) and (278.97,61.56) .. (280.17,61.56) .. controls (281.36,61.56) and (282.33,62.53) .. (282.33,63.72) .. controls (282.33,64.92) and (281.36,65.89) .. (280.17,65.89) .. controls (278.97,65.89) and (278,64.92) .. (278,63.72) -- cycle ;
\draw [fill={rgb, 255:red, 0; green, 0; blue, 0 } ,fill opacity=1 ] (297,120.72) .. controls (297,119.53) and (297.97,118.56) .. (299.17,118.56) .. controls (300.36,118.56) and (301.33,119.53) .. (301.33,120.72) .. controls (301.33,121.92) and (300.36,122.89) .. (299.17,122.89) .. controls (297.97,122.89) and (297,121.92) .. (297,120.72) -- cycle ;
\draw [fill={rgb, 255:red, 0; green, 0; blue, 0 } ,fill opacity=1 ] (279,176.72) .. controls (279,175.53) and (279.97,174.56) .. (281.17,174.56) .. controls (282.36,174.56) and (283.33,175.53) .. (283.33,176.72) .. controls (283.33,177.92) and (282.36,178.89) .. (281.17,178.89) .. controls (279.97,178.89) and (279,177.92) .. (279,176.72) -- cycle ;
\draw [fill={rgb, 255:red, 0; green, 0; blue, 0 } ,fill opacity=1 ] (231,211.72) .. controls (231,210.53) and (231.97,209.56) .. (233.17,209.56) .. controls (234.36,209.56) and (235.33,210.53) .. (235.33,211.72) .. controls (235.33,212.92) and (234.36,213.89) .. (233.17,213.89) .. controls (231.97,213.89) and (231,212.92) .. (231,211.72) -- cycle ;
\draw [fill={rgb, 255:red, 0; green, 0; blue, 0 } ,fill opacity=1 ] (171,211.72) .. controls (171,210.53) and (171.97,209.56) .. (173.17,209.56) .. controls (174.36,209.56) and (175.33,210.53) .. (175.33,211.72) .. controls (175.33,212.92) and (174.36,213.89) .. (173.17,213.89) .. controls (171.97,213.89) and (171,212.92) .. (171,211.72) -- cycle ;
\draw [fill={rgb, 255:red, 0; green, 0; blue, 0 } ,fill opacity=1 ] (124,177.72) .. controls (124,176.53) and (124.97,175.56) .. (126.17,175.56) .. controls (127.36,175.56) and (128.33,176.53) .. (128.33,177.72) .. controls (128.33,178.92) and (127.36,179.89) .. (126.17,179.89) .. controls (124.97,179.89) and (124,178.92) .. (124,177.72) -- cycle ;
\draw [fill={rgb, 255:red, 0; green, 0; blue, 0 } ,fill opacity=1 ] (105,120.72) .. controls (105,119.53) and (105.97,118.56) .. (107.17,118.56) .. controls (108.36,118.56) and (109.33,119.53) .. (109.33,120.72) .. controls (109.33,121.92) and (108.36,122.89) .. (107.17,122.89) .. controls (105.97,122.89) and (105,121.92) .. (105,120.72) -- cycle ;
\draw [fill={rgb, 255:red, 0; green, 0; blue, 0 } ,fill opacity=1 ] (124,64.72) .. controls (124,63.53) and (124.97,62.56) .. (126.17,62.56) .. controls (127.36,62.56) and (128.33,63.53) .. (128.33,64.72) .. controls (128.33,65.92) and (127.36,66.89) .. (126.17,66.89) .. controls (124.97,66.89) and (124,65.92) .. (124,64.72) -- cycle ;
\draw [fill={rgb, 255:red, 0; green, 0; blue, 0 } ,fill opacity=1 ] (171,29.72) .. controls (171,28.53) and (171.97,27.56) .. (173.17,27.56) .. controls (174.36,27.56) and (175.33,28.53) .. (175.33,29.72) .. controls (175.33,30.92) and (174.36,31.89) .. (173.17,31.89) .. controls (171.97,31.89) and (171,30.92) .. (171,29.72) -- cycle ;
\draw [color={rgb, 255:red, 255; green, 0; blue, 0 } ,draw opacity=1 ] (118.33,49.33) .. controls (134.33,33.33) and (143.33,12.33) .. (186.33,21.33) .. controls (229.33,30.33) and (166.33,58.33) .. (146.33,86.33) .. controls (126.33,114.33) and (137.33,158.67) .. (106.33,139.33) .. controls (75.33,120) and (102.33,65.33) .. (118.33,49.33) -- cycle ;
\draw [color={rgb, 255:red, 140; green, 100; blue, 40 } ,draw opacity=1 ] (276.33,133.33) .. controls (286.33,92.33) and (304.33,96.33) .. (313.33,111.33) .. controls (322.33,126.33) and (309.33,130.33) .. (299.33,159.33) .. controls (289.33,188.33) and (284.33,204.33) .. (267.33,190.33) .. controls (250.33,176.33) and (266.33,174.33) .. (276.33,133.33) -- cycle ;
\draw [color={rgb, 255:red, 0; green, 0; blue, 255 } ,draw opacity=1 ] (175.33,97) .. controls (207.33,6.33) and (238.67,3) .. (252.33,26) .. controls (266,49) and (227.33,73) .. (216.33,128) .. controls (205.33,183) and (188.33,231.33) .. (168.33,225) .. controls (148.33,218.67) and (143.33,187.67) .. (175.33,97) -- cycle ;
\draw [color={rgb, 255:red, 0; green, 200; blue, 0 } ,draw opacity=1 ] (200.33,87.33) .. controls (262.33,33.33) and (335.33,30.33) .. (273.33,104.33) .. controls (211.33,178.33) and (273.33,249.33) .. (235.33,230.33) .. controls (197.33,211.33) and (192,172.72) .. (127.33,193.33) .. controls (62.67,213.94) and (138.33,141.33) .. (200.33,87.33) -- cycle ;
\draw [color={rgb, 255:red, 255; green, 0; blue, 0 } ,draw opacity=1 ][fill={rgb, 255:red, 255; green, 0; blue, 0 } ,fill opacity=1 ] (416.47,195.06) .. controls (415.27,195.04) and (414.31,194.06) .. (414.32,192.87) .. controls (414.33,191.67) and (415.31,190.71) .. (416.51,190.72) .. controls (417.71,190.73) and (418.67,191.71) .. (418.65,192.91) .. controls (418.64,194.11) and (417.66,195.07) .. (416.47,195.06) -- cycle ;
\draw [color={rgb, 255:red, 0; green, 0; blue, 255 } ,draw opacity=1 ][fill={rgb, 255:red, 0; green, 0; blue, 255 } ,fill opacity=1 ] (416.56,89.06) .. controls (415.36,89.05) and (414.4,88.07) .. (414.41,86.87) .. controls (414.43,85.68) and (415.41,84.72) .. (416.6,84.73) .. controls (417.8,84.74) and (418.76,85.72) .. (418.75,86.92) .. controls (418.73,88.11) and (417.75,89.07) .. (416.56,89.06) -- cycle ;
\draw [color={rgb, 255:red, 0; green, 200; blue, 0 } ,draw opacity=1 ][fill={rgb, 255:red, 126; green, 211; blue, 33 } ,fill opacity=1 ] (521.56,89.14) .. controls (520.37,89.13) and (519.41,88.15) .. (519.42,86.95) .. controls (519.43,85.76) and (520.41,84.8) .. (521.61,84.81) .. controls (522.8,84.82) and (523.76,85.8) .. (523.75,87) .. controls (523.74,88.19) and (522.76,89.15) .. (521.56,89.14) -- cycle ;
\draw [color={rgb, 255:red, 140; green, 100; blue, 40 } ,draw opacity=1 ][fill={rgb, 255:red, 139; green, 87; blue, 42 } ,fill opacity=1 ] (521.46,196.14) .. controls (520.26,196.12) and (519.3,195.14) .. (519.32,193.95) .. controls (519.33,192.75) and (520.31,191.79) .. (521.51,191.8) .. controls (522.7,191.82) and (523.66,192.8) .. (523.65,193.99) .. controls (523.64,195.19) and (522.66,196.15) .. (521.46,196.14) -- cycle ;
\draw [color={rgb, 255:red, 0; green, 0; blue, 0 } ,draw opacity=0.8 ] (420.58,86.94) -- (517.58,86.93) ;
\draw [color={rgb, 255:red, 0; green, 0; blue, 0 } ,draw opacity=0.8 ] (521.52,189.97) -- (521.7,91.98) ;

\draw (360.84,129.41) node [anchor=north west][inner sep=0.75pt] [align=left] {$\displaystyle \mu _{\pi } =1$};
\draw (456.48,149.19) node [anchor=north west][inner sep=0.75pt] [align=left] {$\displaystyle \mu _{\pi } =1$};
\draw (526.84,129.77) node [anchor=north west][inner sep=0.75pt] [align=left] {$\displaystyle \mu _{\pi } =2$};
\draw (445.7,60.08) node [anchor=north west][inner sep=0.75pt] [align=left] {$\displaystyle \mu _{\pi } =3$};
\draw (112,115) node [anchor=north west][inner sep=0.75pt] [align=left] {$\displaystyle 0$};
\draw (130,63) node [anchor=north west][inner sep=0.75pt] [align=left] {$\displaystyle 1$};
\draw (171,34) node [anchor=north west][inner sep=0.75pt] [align=left] {$\displaystyle 2$};
\draw (225,34) node [anchor=north west][inner sep=0.75pt] [align=left] {$\displaystyle 3$};
\draw (267,64) node [anchor=north west][inner sep=0.75pt] [align=left] {$\displaystyle 4$};
\draw (285,115) node [anchor=north west][inner sep=0.75pt] [align=left] {$\displaystyle 5$};
\draw (267,164) node [anchor=north west][inner sep=0.75pt] [align=left] {$\displaystyle 6$};
\draw (225,195) node [anchor=north west][inner sep=0.75pt] [align=left] {$\displaystyle 7$};
\draw (171,195) node [anchor=north west][inner sep=0.75pt] [align=left] {$\displaystyle 8$};
\draw (129,164) node [anchor=north west][inner sep=0.75pt] [align=left] {$\displaystyle 9$};

\end{tikzpicture}

		\caption{The partition $\pi=\{\{0,1,2\},\{3,8\},\{4,7,9\},\{5,6\}\}\in \mathcal{P}(10)$ is represented on the left. Its corresponding off-diagonal connections graph $\Gamma_\pi$ is represented on the right. Note that $\pi$ is crossing.}
		\label{fig:placeholder}
	\end{figure}
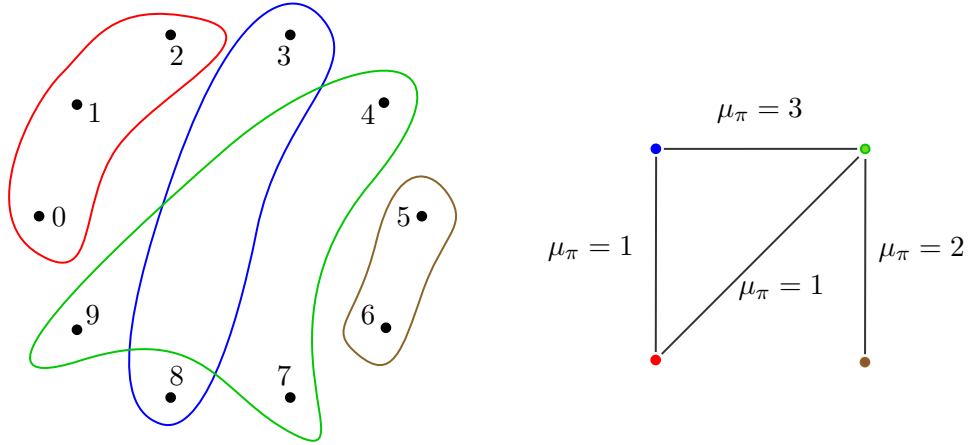

	Next, we recall some terminology to decompose a graph into its maximal $2$-connected subgraphs. We follow the notation of \cite[Section 3.1]{Diestel2017}. 
	\begin{defn}\label{def: blocks}
		Let $\Gamma$ be a finite connected simple graph.
		\begin{itemize}
			\item An edge $e$ is a \emph{bridge} if $\Gamma-e$ is disconnected.
			\item A vertex $v$ is a \emph{cut vertex} if $\Gamma-v$ is disconnected.
			\item A graph with at least three vertices is \emph{2-connected} if it is connected and has no cut vertex.
			\item A \emph{graph block} is either a bridge together with its two endpoints, or a maximal 2-connected subgraph.
		\end{itemize}
		We denote the family of graph blocks of $\Gamma$ by $\mathcal{B}(\Gamma)$.
 \mynomenclature[P]{$\mathcal{B}(\Gamma_\pi)$}{The family of $2$-connected components (also called graph blocks) of $\Gamma_\pi$}
	\end{defn} 
	Note that distinct graph blocks have disjoint edge sets and meet in at most one vertex, which is then a cut vertex. 	
	
	Let $E^+(\widetilde{X}_\ell)$ denote the set of oriented lifted $1$-cells corresponding to the basis of $C_1^{(2)}(\widetilde{X}_\ell)$
 \mynomenclature[C]{$E^+(\widetilde{X}_\ell)$}{The set of oriented lifted $1$-cells of $X_\ell$}
 chosen in Subsection \ref{subsection: preliminaries homology}, so that we have $C_1^{(2)}(\widetilde{X}_\ell) = \ell^2\left(E^+(\widetilde{X}_\ell)\right)$. For each $e\in E^+(\widetilde{X}_\ell)$, denote by $u_e$
 \mynomenclature[C]{$u_e$}{The unit vector in $C_1^{(2)}(\widetilde{X}_\ell)=\ell^2(E^+(\widetilde{X}_\ell))$ associated to $e\in E^+(\widetilde{X}_\ell)$}
 the corresponding cellular basis vector. Recall that, for every lifted face $F$, the vector $b_F=A_\ell u_F$ is the signed cellular boundary of $F$. For each $e\in E^+(\widetilde{X}_\ell)$ let us write $c_F(e)\coloneqq \langle b_F,u_e\rangle$.
 \mynomenclature[C]{$c_F(e)$}{The coefficient $\langle b_F,u_e\rangle$ of an edge $e\in E^+(\widetilde{X}_\ell)$ with respect to a lifted face $F\in \mathcal{F}_\ell$}
 On the event $\mathcal{G}_\ell$, Item \eqref{item: prob 2} of Proposition \ref{prop: lifted face geometry} implies that the
	boundary of $F$ traverses each geometric edge at most once. Consequently, we have
	\[
	c_F(e)=\begin{cases}
		1,
		&\text{if }\partial F \text{ traverses }e \text{ in its chosen orientation},\\
		-1,
		&\text{if } \partial F\text{ traverses }e \text{ in opposite orientation, and }\\
		0,
		&\text{if }e\text{ is not contained in }\partial F.
	\end{cases}
	\]
	In particular, $c_F(e)\in\{-1,0,1\}$ and $b_F= \sum_{e\in E^+(\widetilde{X}_\ell)}c_F(e)u_e$, where exactly $\ell$ terms are non-zero. Since the vectors $(u_e)_{e\in E^+(\widetilde{X}_\ell)}$ form an orthonormal basis of $\ell^2(E^+(\widetilde{X}_\ell))$, for any two lifted faces $F$ and $H$ we have
	\[
	h_\ell(F,H)=\langle b_F,b_H\rangle= \sum_{e\in E^+(\widetilde{X}_\ell)}c_F(e)c_H(e).
	\]
	Only edges contained in $\partial F \cap \partial H$ contribute to this sum.
	
	For each $\mathbf{F}=(F_C)_{C\in\pi}\in \Omega_{\ell}(\pi,{\typemap})$ and $i \in \mathcal{O}_\pi$, let us define
	\[
	\mathcal{E}_{i}(\mathbf{F})\coloneqq \left\{ e\in E^+(\widetilde{X}_\ell):e\subseteq \partial F_{[i]_\pi}\cap \partial F_{[i+1]_\pi} \right\},
	\]
 \mynomenclature[P]{$\mathcal{E}_{i}(\mathbf{F})$}{Edges in $E^+(\widetilde{X}_\ell)$ that belong to the intersection between $\partial F_{[i]_\pi}$ and $\partial F_{[i+1]_\pi}$}
	and denote $\mathcal{E}(\mathbf{F}) \coloneqq \prod_{i\in \mathcal{O}_\pi}\mathcal{E}_i(\mathbf{F})$.
 \mynomenclature[P]{$\mathcal{E}(\mathbf{F})$}{Tuples $(e_i)_{i\in \mathcal{O}_\pi}$, with $e_i\in \mathcal{E}_{i}(\mathbf{F})$ for each $i\in \mathcal{O}_\pi$}
 In other words, an element $\mathbf{e}\in \mathcal{E}(\mathbf{F})$ corresponds to a choice, for each $i\in \mathcal{O}_\pi$, of an edge $e_i\in \partial F_{[i]_{\pi}}\cap \partial F_{[i+1]_{\pi}}$.

	\begin{lem}\label{lemma: partition off diagonal}
		Let $\pi\in \mathcal{P}(k)$ and let ${\typemap}:\pi\to \{1,\ldots,m\}$. Let $Z_{\ell}(\pi,{\typemap})$ be defined as in Equation \eqref{eq: def of Z}. Then, on the event $\mathcal{G}_\ell$ of Proposition \ref{prop: lifted face geometry}, we have
		\[
		Z_{\ell}(\pi,{\typemap})=\ell^{\mathrm{diag}(\pi)}\sum_{\mathbf{F}\in \Omega_{\ell}(\pi,{\typemap})}\sum_{\mathbf{e}\in \mathcal{E}(\mathbf{F})} \prod_{i\in \mathcal{O}_\pi}c_{F_{[i]_\pi}}(e_i)c_{F_{[i+1]_\pi}}(e_i).
		\]
	\end{lem}
	\begin{proof}
		By Item \eqref{item: prob 2} of Proposition \ref{prop: lifted face geometry}, every lifted face has an embedded boundary consisting of $\ell$ edges. Consequently, $\|b_F\|_2^2=\ell$ for every lifted face $F$.
		
		Let us fix $\mathbf{F}\in \Omega_{\ell}(\pi,{\typemap})$. Note that for each $i\notin \mathcal{O}_\pi$ we have
		\[
		h_{\ell}(F_{[i]_\pi},F_{[i+1]_\pi} )=h_{\ell}(F_{[i]_\pi},F_{[i]_\pi} )=\|b_{F_{[i]_\pi}}\|_2^2=\ell.
		\]
		Since there are exactly $\mathrm{diag}(\pi)$ such indices, we have that
		\[
		\prod_{i=0}^{k-1}	h_{\ell}(F_{[i]_\pi},F_{[i+1]_\pi} )=\ell^{\mathrm{diag}(\pi)}\prod_{i\in \mathcal{O}_\pi}	h_{\ell}(F_{[i]_\pi},F_{[i+1]_\pi}).
		\]
		On the other hand, for each $i\in \mathcal{O}_{\pi}$ we have
		\[
		h_{\ell}(F_{[i]_\pi},F_{[i+1]_\pi}) = \sum_{e\in \mathcal{E}_i(\mathbf{F})}c_{F_{[i]_\pi}}(e)c_{F_{[i+1]_\pi}}(e).
		\]
		
		As a consequence, we get
		\begin{equation*}
			\prod_{i\in \mathcal{O}_\pi}	h_{\ell}(F_{[i]_\pi},F_{[i+1]_\pi})=\prod_{i\in \mathcal{O}_\pi}\left( \sum_{e\in \mathcal{E}_i(\mathbf{F})}c_{F_{[i]_\pi}}(e)c_{F_{[i+1]_\pi}}(e)\right)=\sum_{\mathbf{e}\in \mathcal{E}(\mathbf{F})}\prod_{i\in \mathcal{O}_\pi}c_{F_{[i]_\pi}}(e_i)c_{F_{[i+1]_\pi}}(e_i).
		\end{equation*}
		
		This concludes the proof.
	\end{proof}
	
	\subsection{Degrees of freedom and non-crossing partitions}
	
	Consider the event $\mathcal{G}_\ell$, let $F$ be a lifted face and let
	$e\subseteq\partial F$ be an edge. Denote by $\mathrm{pos}_F(e)\in\Z/\ell\Z$ \mynomenclature[P]{$\mathrm{pos}_F(e)$}{The starting position of an edge $e\in \partial F$}
 the unique boundary position at which the oriented contour of $F$ traverses $e$, in either direction. This position is unique by Item \eqref{item: prob 2} of Proposition \ref{prop: lifted face geometry}.
	
	\begin{lem}\label{lem: graph block reference position}
	Fix $k\geq1$ and consider a presentation $G_{\ell}$ in the event $\mathcal{G}_\ell$ of Proposition \ref{prop: lifted face geometry}. Let $\pi\in\mathcal{P}(k)$ and let ${\typemap}:\pi\to\{1,\ldots,m\}$. Let $\mathbf{F}=(F_C)_{C\in\pi}\in\Omega_\ell(\pi,{\typemap})$. For each $i\in \mathcal{O}_\pi$, choose an edge $e_i\subseteq F_{[i]_\pi}\cap F_{[i+1]_\pi}$. For $B\in\mathcal{B}(\Gamma_\pi)$ and $C\in V(B)$, set
		\begin{equation*}
			\mathcal{P}_{C,B}\coloneqq
			\left\{ \mathrm{pos}_{F_C}(e_i): i\in \mathcal{O}_\pi, \{[i]_\pi,[i+1]_\pi\}\in E(B),\text{ and } C\in\{[i]_\pi,[i+1]_\pi\}
			\right\}.
		\end{equation*}
		Then $\mathcal{P}_{C,B}$ is contained in a cyclic interval of length at most $2k\rho_\ell$.
	\end{lem}
	\begin{proof}
		Fix an integer $p\ge \max\{7,k+1\}$. We use Item \eqref{item: prob 2} of
		Proposition \ref{prop: lifted face geometry} for intersections
		of two faces, and Lemma \ref{lem: close overlap positions}.
		
		If $B$ is a bridge with endpoints $C$ and $D$, then every position in
		$\mathcal{P}_{C,B}$ comes from an edge of the single intersection path
		$F_C\cap F_D$. This path has length at most $\rho_\ell$.
		
		Now suppose that $B$ is maximal 2-connected. Choose a neighbour $D_0$ of $C$ in $B$ and one selected edge in $F_C\cap F_{D_0}$, with starting position $u_0$ on $\partial F_C$. Let $u\in\mathcal{P}_{C,B}$ arise from a neighbor $D$ of $C$. If $D=D_0$, then $d_{\ell}(u,u_0)\leq \rho_\ell$. If $D\neq D_0$, the graph $B-C$ is connected, so it contains a simple path connecting $D_0$ to $D$. This path uses at most $k-1$ vertices. Applying
		Lemma \ref{lem: close overlap positions} to $F_C$ together with the two faces $F_{D_0},F_D$ and the faces along this path, we obtain that $d_{\ell}(u,u_0)\leq k\rho_\ell$. Thus every point of $\mathcal{P}_{C,B}$ lies within cyclic distance $k\rho_\ell$ of $u_0$. Then $\mathcal{P}_{C,B}$ is contained in a cyclic interval
		of length at most $2k\rho_\ell$.
	\end{proof}
 \begin{defn}\label{defn: freesum}
 For a partition $\pi\in \mathcal{P}(k)$ let us write $\freesum{\pi}=\sum_{B\in \mathcal{B}(\Gamma_\pi)}|V(B)|$. \mynomenclature[P]{$\freesum{\pi}$}{The sum $\sum_{B\in \mathcal{B}(\Gamma_\pi)}|V(B)|$}
 \end{defn}
 One should think of this number as the number of degrees of freedom given by the positions in the off-diagonal terms of $\pi$. This is made concrete in the next lemma.
	\begin{lem} \label{lem:bound order l}
	Fix $k\geq1$, $\pi\in\mathcal{P}(k)$, and $\typemap:\pi\to\{1,\ldots,m\}$. Consider the event $\mathcal{G}_{\ell}$ of Proposition \ref{prop: lifted face geometry}. Then there is a constant $C>0$, independent of $\ell$ and of the presentation of $G_{\ell}$, such that for all sufficiently large $\ell$ we have
			\[
			|Z_\ell(\pi,\typemap)| 	\leq 	C\ell^{\mathrm{diag}(\pi)+\freesum{\pi}} (\log\ell)^{2k}.
			\]
		\end{lem}
	
	\begin{proof}
		From Lemma \ref{lemma: partition off diagonal} and the fact that the coefficients take values in $\{-1,0,1\}$, we have
		\begin{equation}\label{eq: non-zero summands}
			|Z_\ell(\pi,{\typemap})| \leq \ell^{\mathrm{diag}(\pi)} \left|\left\{
			(\mathbf{F},\mathbf{e}): \mathbf{F}\in\Omega_\ell(\pi,{\typemap}) \text{ and }\mathbf{e}\in\mathcal{E}(\mathbf{F}) \right\}\right|.
		\end{equation}
		We therefore have to bound the number of possible face families together
		with their selected common edges.
		
		Let $\mathbf{F}=(F_C)_{C\in \pi}\in\Omega_\ell(\pi,{\typemap})$ and $\mathbf{e}=(e_i)_{i\in \mathcal{O}_\pi}\in\mathcal{E}(\mathbf{F})$. For every $i\in \mathcal{O}_\pi$, put $C_i\coloneqq [i]_\pi$ and $D_i\coloneqq [i+1]_\pi$. The selected edge $e_i$ from $\mathbf{e}$ gives two cyclic boundary positions: $u_{i,C_i} \coloneqq \mathrm{pos}_{F_{C_i}}(e_i)$ and $u_{i,D_i}\coloneqq \mathrm{pos}_{F_{D_i}}(e_i)$. Thus there are exactly $2|\mathcal{O}_\pi|$ boundary-position variables.
		
		Note that the cardinality of $\{(C,B):B\in \mathcal{B}(\Gamma_\pi), C\in V(B)\}$ is $\freesum{\pi}$.
		
		Every graph edge $\{C_i,D_i\}$ belongs to a unique block $B_i$ of
		$\Gamma_\pi$. Assign $u_{i,C_i}$ to the incidence $(C_i,B_i)$ and $u_{i,D_i}$ to the incidence $(D_i,B_i)$. In this way the $2|\mathcal{O}_\pi|$ position variables are divided among the $\freesum{\pi}$ possible pairs $(C,B)$.
		
		Each pair $(C,B)$ has at least one position variable. Indeed, if $C\in V(B)$, then $C$ is incident to some graph edge of $B$, and that graph edge is represented by at least one index $i\in \mathcal{O}_\pi$. For each incidence $(C,B)$, choose and fix one such index $i(C,B)\in \mathcal{O}_\pi$, depending only on $\pi$, and call $u_{C,B}\coloneqq u_{i(C,B),C}$ the reference position of the incidence $(C,B)$.
		
		Every reference position is a place in $\Z/\ell\Z$, so that
		\begin{equation*}\label{eq: reference positions}
			\text{nr. of collections of reference positions}\le	\ell^{\freesum{\pi}}.
		\end{equation*}
		Once these reference positions have been fixed, Lemma \ref{lem: graph block reference position} implies that all the positions assigned to a fixed incidence $(C,B)$ lie in a cyclic interval of length at most $2k\rho_\ell$. Since this interval contains the chosen reference position $u_{C,B}$, every other position assigned to $(C,B)$ has at most $4k\rho_\ell+1\leq D_k\rho_\ell$ possible values, where $D_k$ depends only on $k$.

		There are $2|\mathcal{O}_\pi|-\freesum{\pi}$ non-reference position variables. It follows that the total number of possible collections of boundary positions is at most
		\begin{equation*}\label{eq: boundary positions}
			\ell^{\freesum{\pi}} (D_k\rho_\ell)^{2|\mathcal{O}_\pi|-\freesum{\pi}}.
		\end{equation*}
		
		For each $i\in \mathcal{O}_\pi$, we must also record whether the boundary
		orientations of $F_{C_i}$ and $F_{D_i}$ agree or disagree along the
		selected common edge $e_i$. Hence there are at most $2^{|\mathcal{O}_\pi|}$ possible collections of relative orientations.
		
		We next show that the boundary positions and relative orientations
		determine at most one face family $\mathbf{F}\in\Omega_\ell(\pi,{\typemap})$ and at most one collection of selected edges $\mathbf{e}=(e_i)_{i\in \mathcal{O}_\pi}$.
		
		The graph $\Gamma_\pi$ is connected. Choose a spanning tree $T\subseteq\Gamma_\pi$, root it at $C_0\coloneqq [0]_\pi$ and for every edge $\{C,D\}\in E(T)$, choose one index $i(C,D)\in \mathcal{O}_\pi$ satisfying $\{[i(C,D)]_\pi,[i(C,D)+1]_\pi\}=\{C,D\}$.
		
		By the definition of $\Omega_\ell(\pi,{\typemap})$, the root face is fixed: $F_{C_0} = F^0_{{\typemap}(C_0),\ell}$.
		
		Suppose inductively that $F_C$ has already been determined and that $D$ is an offspring of $C$ in the rooted spanning tree. Recall from Section \ref{section: probability small cancellation} that we write the relators as $r_{j,\ell}=s_{j,0}s_{j,1}\cdots s_{j,\ell-1}$, where $s_{j,p}\in \mathcal{A}$ for every $0\leq p\leq \ell-1$. We also write $w_{j,p}=s_{j,0}s_{j,1}\cdots s_{j,p-1}$
 \mynomenclature[G]{$w_{j,p}$}{The subword $s_{j,0}s_{j,1}\cdots s_{j,p-1}$ of $r_{j,\ell}$ consisting of the first $p$ positions}
 for each $0\leq p\leq \ell$ and $w_{j,0}=e_{G_\ell}$. Set $a\coloneqq {\typemap}(C)$ , $b\coloneqq {\typemap}(D)$, so that $F_C=gF^0_{a,\ell}$ and $\smash{F_D=hF^0_{b,\ell}}$ for some $g,h\in G_\ell$. We also write $p\coloneqq u_{i(C,D),C}$ and $q\coloneqq u_{i(C,D),D}$. The boundary oriented edge of $F_C$ at position $p$ is $(gw_{a,p}, gw_{a,p+1})$ while the boundary oriented edge of $F_D$ at position $q$ is $(hw_{b,q},hw_{b,q+1})$.	If the two boundary orientations agree along the common edge, then their initial vertices must coincide and hence $gw_{a,p}=hw_{b,q}$. This implies $h=gw_{a,p}w_{b,q}^{-1}$. If their orientations are opposite, the initial vertex of the first edge equals the terminal vertex of the second, and hence $gw_{a,p}=hw_{b,q+1}$. This implies $h=gw_{a,p}w_{b,q+1}^{-1}$. In either case, the already determined face $F_C$, the two boundary positions, the label $b$, and the relative orientation determine at most one possible face $F_D$.
		
		The inductive argument starting from the root along the spanning tree determines every face $F_C$. Once a face $F_C$ and a boundary position $u_{i,C}$ are known, the corresponding geometric edge is the edge of $\partial F_C$ starting at that position. Hence, every $e_i$ from $\mathbf{e}$ is also determined. 
		
		Combining the previous estimates, the cardinality on the right-hand side of Equation \eqref{eq: non-zero summands} is therefore at most
		\[
		2^{|\mathcal{O}_\pi|} \ell^{\freesum{\pi}} (D_k\rho_\ell)^{2|\mathcal{O}_\pi|-\freesum{\pi}}.\]
		Consequently,
		\[
		|Z_\ell(\pi,{\typemap})| \leq \ell^{\mathrm{diag}(\pi)}	2^{|\mathcal{O}_\pi|} \ell^{\freesum{\pi}} (D_k \rho_\ell)^{2|\mathcal{O}_\pi|-\freesum{\pi}}.
		\]
	Since $|\mathcal{O}_\pi|\leq k$, $\freesum{\pi}\leq 2|\mathcal{O}_\pi|$, and $\rho_\ell\geq1$, we obtain
	\[
	|Z_\ell(\pi,\typemap)| \leq C\ell^{\mathrm{diag}(\pi)+\freesum{\pi}}\rho_\ell^{2k}. \]
	Since $\rho_\ell=O(\log\ell)$, this implies
	\[
	|Z_\ell(\pi,\typemap)|	\leq	C'\ell^{\mathrm{diag}(\pi)+\freesum{\pi}}
	(\log(\ell))^{2k}
	\]
	for all sufficiently large $\ell$, with $C'$ independent of $\ell$. This proves the lemma.
	\end{proof}
	
	It will be useful to work with the structure of blocks in a connected graph.
	\begin{defn}\label{def: block cut graph}
Let $\pi \in\mathcal{P}(k)$ be a partition and let	$\mathrm{Cut}(\Gamma_\pi)$ be the set of cut vertices of $\Gamma_\pi$. The \emph{block-cut graph} of $\Gamma_\pi$ is the bipartite graph with vertex set equal to $\mathcal{B}(\Gamma_\pi)\cup\mathrm{Cut}(\Gamma_\pi)$,
		in which a block $B\in\mathcal{B}(\Gamma_\pi)$ is adjacent to a cut
		vertex $v\in\mathrm{Cut}(\Gamma_\pi)$ if and only if $v\in V(B)$. Since $\Gamma_\pi$ is connected, when $|\pi|\ge 2$ its block-cut graph is a tree; see \cite[Lemma 3.1.4]{Diestel2017}. 
 \mynomenclature[P]{$\mathrm{Cut}(\Gamma_\pi)$}{The set of cut vertices of $\Gamma_\pi$}
	\end{defn}
	We shall also use the bipartite incidence tree $\mathcal{I}_\pi$
 \mynomenclature[P]{$\mathcal{I}_\pi$}{ The bipartite incidence tree, with vertex set $\pi\sqcup \mathcal{B}(\Gamma_\pi)$, and an edge $\{C,B\}$ for each $C\in \pi$ and $B\in \mathcal{B}(\Gamma_\pi)$ such that $C\in V(B)$}
 with vertex set $\pi\sqcup\mathcal{B}(\Gamma_\pi)$ in which $C\in\pi$ is adjacent to $B\in\mathcal{B}(\Gamma_\pi)$ exactly when $C\in V(B)$. For $|\pi|\geq 2$, this is obtained from the block-cut tree by attaching each non-cut vertex to its unique graph block as a leaf. Thus $\mathcal{I}_\pi$ is a tree. When $|\pi|=1$, we take $\mathcal{I}_\pi$ to be its single vertex.
	
	\begin{lem}\label{lem: noncrossing_partitions}
		Let $k\ge 1$ and let $\pi\in \mathcal{P}(k)$ be a partition. Let $\freesum{\pi}$ be as in Definition \ref{defn: freesum}. Then 
		\begin{equation*}\label{eq: degrees inequality k}
			\freesum{\pi}+\mathrm{diag}(\pi)\le k,
		\end{equation*}
		with equality if and only if the partition $\pi$ is non-crossing. Moreover, if $\pi$ is non-crossing, every graph block $B\in\mathcal{B}(\Gamma_\pi)$ is either:
		\begin{enumerate}
			\item a simple cycle, with $\mu_\pi(e)=1$ for every $e\in E(B)$, or
			\item a bridge, occurring once in each direction in the cyclic
			sequence of off-diagonal transitions.
		\end{enumerate}
	\end{lem}
	\begin{proof}
	Let us first consider the case $|\pi|=1$. Then there are no off-diagonal
		transitions and no graph blocks. This implies the equality from the statement.
		
	Let us now assume that $|\pi|\geq 2$.
		
		If $B$ is a $2$-connected graph block, then
		\[
		\sum_{e\in E(B)}\mu_\pi(e) \geq |E(B)| =\frac12\sum_{C\in V(B)}\deg_B(C)
		\geq |V(B)|.
		\]
		Equality holds precisely when every edge has multiplicity one
		and every vertex has degree two, that is, when $B$ is a simple
		cycle whose edges occur once.
		
	If $B$ consists of a bridge $e$ and its endpoints, deleting $e$ separates $\Gamma_\pi$ into two components. The cyclic sequence of classes crosses from one component to the other as many times as it crosses back. Therefore $\mu_\pi(e)\geq 2=|V(B)|$, with equality precisely when $e$ occurs once in each direction.
		
	Since the graph blocks partition the edge set, we have
		\[
		k-\mathrm{diag}(\pi) =|\mathcal{O}_\pi| =\sum_{B\in\mathcal{B}(\Gamma_\pi)} \sum_{e\in E(B)}\mu_\pi(e) \geq	\sum_{B\in\mathcal{B}(\Gamma_\pi)}|V(B)| =\freesum{\pi}.
		\]
		Moreover, equality holds if and only if every graph block belongs to one of the two classes in the statement of the lemma.
		
	We next show that if $\pi$ is a non-crossing partition, then its blocks fall into one of these two classes. We proceed by induction on the number of classes.

Non-crossing partitions always contain a class consisting of consecutive positions; see \cite[Remark 9.2(2)]{NicaSpeicher2006}. Thus we may choose a class of $\pi$ of the form $C=\{i,i+1,\ldots,j\}$. Moreover, after removing the positions belonging to $C$, the classes that remain still form a non-crossing partition of the remaining positions. Denote this partition by $\pi'$, and remark that $|\pi'|=|\pi|-1$.

Since the positions of $C$ occur consecutively and $|\pi|\ge 2$, there is a well-defined class $L$ containing the position immediately before $C$, and a
well-defined class $H$ containing the position immediately after
$C$, where the positions are read in the original cyclic order.
		
If $L=H$, reinserting $C$ adds a bridge traversed once in each direction. If $L\neq H$, it replaces one occurrence of $L\to H$ by $L$, $C$, $H$. By induction, the edge $\{L,H\}$ in $\Gamma_{\pi'}$ belongs either to a simple cycle and occurs once, or is a bridge and occurs twice. In the first case, reinsertion subdivides that cycle edge. In the second case, one occurrence of the bridge remains and the other is replaced by the two-edge path through $C$; the resulting graph block is a triangle with each edge occurring once. All other graph blocks are unchanged.

Hence, we conclude that every block $B$ of a non-crossing partition $\pi$ is either a simple cycle with $\mu_\pi(e)=1$ for each $e\in E(B)$, or a bridge that occurs once in each direction.
		
Conversely, suppose that every block of a partition $\pi$ has one of the two forms in the statement of the lemma. We proceed once more by induction on the value $|\pi|$.

Choose a vertex $C$ belonging to exactly one graph block. Such a vertex exists since the graph of blocks is a finite tree, and hence has a leaf. Its incident edges have total multiplicity two. Hence the cyclic sequence enters $C$ once and leaves it once, so all positions belonging to $C$ form one cyclic interval.
		
Delete the positions of $C$. If $C$ belongs to a bridge, this removes a leaf and its incident bridge. If it belongs to a cycle of length at least four, this suppresses $C$ and shortens the cycle by one edge. If it belongs to a triangle, the remaining edge becomes a bridge occurring once in each direction.
Hence, we see that the new partition still has the stated graph-block structure. By the induction hypothesis, it must be non-crossing.
		
Since reinserting a class consisting of consecutive cyclic positions does not create a crossing, we conclude that $\pi$ is non-crossing.

This finishes the inductive step and hence the proof of the lemma.
	\end{proof}
	
	\begin{lem}\label{lem: counting size vertices blocks}
		Let $\pi\in \mathcal{P}(k)$ be a partition and let $\Gamma_{\pi}$ be the associated graph of off-diagonal connections (Definition \ref{defn: graph off diagonal connections}). Denote by $\mathcal{B}(\Gamma_{\pi})$ the blocks of $\Gamma_{\pi}$ (Definition \ref{def: blocks}). Then
		\[
		\sum_{B\in \mathcal{B}(\Gamma_\pi)}\left( |V(B)|-1 \right)= |\pi|-1.
		\]
	\end{lem}
	\begin{proof}
		We argue by induction on the number of blocks of $\Gamma_\pi$.
		
		If $\Gamma_\pi$ has only one block $B$, then $V(B)=V(\Gamma_\pi)$. Since the vertices of $\Gamma_\pi$ are the equivalence classes in $\pi$, we have $|V(\Gamma_\pi)|=|\pi|$, and hence $|V(B)|-1=|\pi|-1$.
		
		Suppose now that $\Gamma_\pi$ has at least two blocks. Note that every cut vertex of $\Gamma_{\pi}$ produces a vertex in the block-cut graph of degree at least $2$. Since the block-cut graph is a tree, it must have a vertex with degree $1$, which by the above argument must correspond to a block $B_0\in \mathcal{B}(\Gamma_\pi)$. Let $v$ be the unique cut vertex adjacent to $B_0$ in the block-cut graph.
		
		Let $\Gamma_\pi'$ be the graph obtained from $\Gamma_\pi$ by deleting
		the vertices of $B_0$ that are distinct from $v$. Then $\Gamma_\pi'$ is connected, and $\mathcal{B}(\Gamma_\pi')=\mathcal{B}(\Gamma_\pi)\backslash\{B_0\}$. Moreover, $|V(\Gamma_\pi)|=|V(\Gamma_\pi')|+|V(B_0)|-1$. Applying the induction hypothesis to $\Gamma_\pi'$, we obtain
		\begin{align*}
			\sum_{B\in\mathcal{B}(\Gamma_\pi)}\left(|V(B)|-1\right)&=
			|V(B_0)|-1+\sum_{B\in\mathcal{B}(\Gamma_\pi')}\left(|V(B)|-1\right)\\
			&=|V(B_0)|-1+|V(\Gamma_\pi')|-1\\
			&=|V(\Gamma_\pi)|-1=|\pi|-1.
		\end{align*}
		This completes the induction.
	\end{proof}

	\section{Convergence of the spectral measure to the free-Poisson law}\label{section:convergence of the spectral measure}
 
	The objective of this section is to prove Theorem \ref{thm:main spectral law}. We start by establishing the following proposition.
	\begin{prop}\label{prop:main_technical}
		Let $k\ge 1$ and recall that we denote by $\mathrm{NC}(k)$ the set of non-crossing partitions of $\{0,\ldots,k-1\}$ (Definition \ref{defn: noncrossing partitions}). Let $\pi\in \mathrm{NC}(k)$ and let ${\typemap}:\pi \to \{1,\ldots, m\}$. Let $Z_{\ell}(\pi,{\typemap})$ be defined as in Equation \eqref{eq: def of Z} for each $\ell\ge 1$. Then, the sequence of random variables $(\ell^{-k}Z_\ell(\pi,{\typemap}))_{\ell\geq 1}$ converges in probability to $(n-1)^{-(|\pi|-1)}$. In other words, for every $\varepsilon>0$ we have
		\begin{equation*}\label{eq: convergence noncrossing}
			\P_\ell\left(\left|\frac{1}{\ell^k}Z_{\ell}(\pi,{\typemap})-\left(\frac{1}{n-1}\right)^{|\pi|-1}\right|\geq \varepsilon\right)\xrightarrow[\ell\to \infty]{}0.
		\end{equation*}
	\end{prop}
	To prove this result, we proceed as follows. First, we establish a general weak ergodic theorem in Proposition \ref{prop: weak ergodic theorem}. Next, we introduce the function to which we will apply it in Definition \ref{defn: observables} and compare it with $\ell^{-k}Z_{\ell}(\pi,{\typemap})$ in Lemma \ref{lem: comparison between Z and average of observables}. Finally, in Lemma \ref{lem: expectation of the observable} we compute the value of the limit provided by the weak ergodic theorem for this specific function. The proof of Proposition \ref{prop:main_technical} is then presented in Subsection \ref{subsection: proof of the technical proposition}. We end the section by giving the proof of Theorem \ref{thm:main spectral law} in \S\ref{subsection:proof of spectral convergence}.
 
	\subsection{A weak ergodic theorem}\label{subsection:weak ergodic theorem}
	Fix $n\geq2$ and denote $\mathcal{A}\coloneqq\{x_1,x_1^{-1},\ldots, x_n,x_n^{-1}\}$. For every integer $R\geq 0$, we define $\mathrm{NB}_R$ to be the set of non-backtracking words of length $2R+1$ on the alphabet $\mathcal{A}$. That is, we set
	\begin{equation}\label{eq: non backtracking sequences}
		\mathrm{NB}_R\coloneqq\left\{\mathbf{y}=(y_{-R},\ldots,y_R)\in\mathcal{A}^{2R+1}: y_{i+1}\neq y_i^{-1}\text{ for every }-R\leq i<R \right\}.
	\end{equation}
 \mynomenclature[O]{$\mathrm{NB}_R$}{The set of non-backtracking words $(y_{-R},\dots,y_R)$ of length $2R+1$ on the alphabet $\mathcal{A}$}
	Note that there are $2n$ choices for $y_{-R}$ and $2n-1$ choices for each subsequent letter. Hence, it holds that $|\mathrm{NB}_R|=2n(2n-1)^{2R}$.

	\begin{defn}\label{defn: stationary non-backtracking law on NBR}
		The stationary non-backtracking law on $\mathrm{NB}_R$ is the probability measure $\eta_R\in \mathrm{Prob}(\mathrm{NB}_R)$ defined by
		\[
		\eta_R(\{\mathbf{y}\}) \coloneqq \frac{1}{2n(2n-1)^{2R}}, \text{ for each } \mathbf{y}\in\mathrm{NB}_R.
		\]
 \mynomenclature[O]{$\eta_R$}{The uniform probability measure on $\mathrm{NB}_R$}
	\end{defn}
	
	Recall that the random relators $r_{j,\ell}=s_{j,0}s_{j,1}\cdots s_{j,\ell-1}$, $ 1\leq j\leq m$ of $G_{\ell}$ are independent and uniformly distributed among the cyclically reduced words of length $\ell$ on $\mathcal{A}$. Also, recall that we work with the convention that the second index is always read cyclically as an element of $\Z/\ell\Z$.
	
	\begin{defn}\label{defn: windows}
		Let $p\in\Z/\ell\Z$, $j\in \{1,\ldots,m\}$, and $\ell \ge 2R+1$. We define the \emph{window of radius $R$ rooted at $p$ in the $j$-th relator} by
		\[W_{j,p}^R	\coloneqq \left(s_{j,p-R},\ldots,s_{j,p+R}\right)\in\mathrm{NB}_R.	\]
 \mynomenclature[O]{$W_{j,p}^R$}{The window of radius $R$ rooted at $0\leq p<\ell$ in the relator $r_{j,\ell}$. It is an element of $\mathrm{NB}_R$}
	\end{defn}
	\begin{defn}\label{defn: tau admissible roots}
		Let $f\geq1$ and let $\tau=(t_1,\ldots,t_f)\in\{1,\ldots,m\}^f$. We define the \emph{set of $\tau$-admissible roots} $\mathcal{R}_\ell(\tau)$ by
		\[
		\mathcal{R}_\ell(\tau)\coloneqq\left\{\mathbf{p}=(p_1,\ldots,p_f)\in(\Z/\ell\Z)^f:	p_i\neq p_j\text{ whenever }i\neq j\text{ and }t_i=t_j	\right\}.
		\]
 \mynomenclature[O]{$\mathcal{R}_\ell(\tau)$}{The set of $\tau$-admissible roots---where $\tau=(t_1,\dots,t_f)\in \{1,\dots,m\}^f$ and $f\geq 1$---consisting of tuples $\mathbf{p}=(p_1,\dots,p_f)\in (\Z/\ell\Z)^{f}$ such that $p_i\neq p_j$ whenever $i\neq j$ and $t_i=t_j$}
	\end{defn}
	Note that we have
	\begin{equation}\label{eq: number admissible roots}
		|\mathcal{R}_\ell(\tau)| =\ell^f+O(\ell^{f-1}).
	\end{equation}
	
	\begin{defn}\label{defn: average over tau admissible roots}
		Let $f\ge 1$ and let $\tau\in\{1,\ldots, m\}^f$. Let $\Phi:\mathrm{NB}_R^f\to\R$ be a function. Assume $\ell\geq\max\{f,2R+1\}$. We define the average $S_{R,\tau}(\Phi)$ of $\Phi$ over the set of $\tau$-admissible roots as
		\begin{equation*}\label{eq: average over window}
			S_{R,\tau}(\Phi)
			\coloneqq 	\frac{1}{|\mathcal{R}_\ell(\tau)|} \sum_{\mathbf{p}\in\mathcal{R}_\ell(\tau)}\Phi\left(W_{t_1,p_1}^R,\ldots,W_{t_f,p_f}^R\right).
		\end{equation*}
 \mynomenclature[O]{$S_{R,\tau}(\Phi)$}{The average value of an observable $\Phi\colon \mathrm{NB}_R^f\to \R$ over the set $\mathcal{R}_\ell(\tau)$}
	\end{defn}
	
	For $p,q\in\Z/\ell\Z$, let $d_\ell(p,q)$ denote their cyclic distance. For probability measures $\mu$ and $\nu$ on a finite set $E$, we use the convention
	\[
	\|\mu-\nu\|_{\mathrm{TV}}\coloneqq\frac{1}{2}\sum_{z\in E}|\mu(\{z\})-\nu(\{z\})|
	\]
	for their total variation distance.
	
	The following lemma gives a quantitative statement of approximate independence for windows that are sufficiently far apart from each other. More precisely, if roots belonging to the same relator are separated by a buffer of length at least $D$, then the joint law of the corresponding windows is exponentially close, in total variation, to $\eta_R^{\otimes f}$. Thus, up to an error that decays exponentially in the size of the buffer $D$, these windows behave as $f$ independent samples from the stationary non-backtracking law. Note that buffers are needed only for windows that belong to the same relator, since the law of distinct relators is already independent.
	
	\begin{lem} \label{lem: separated windows ergodic}
		Let $f\geq1$ and let $\tau\in \{1,\ldots,m\}^f$. Then there is a constant $C>0$ with the following property. Let $D\geq1$, and suppose that $\mathbf{p}=(p_1,\ldots,p_f)\in\mathcal{R}_\ell(\tau)$ is a $\tau$-admissible root vector that satisfies 
		\begin{equation}\label{eq: D separated roots}
			d_\ell(p_i,p_j)>2R+D \ \ \text{whenever }i\neq j\text{ and }t_i=t_j.
		\end{equation}
		Suppose that $\ell>f(2R+D)$. Then
		\begin{equation}\label{eq: separated window comparison}
			\left\|	\mathcal{L}\left(W_{t_1,p_1}^R,\ldots,W_{t_f,p_f}^R\right)
			-\eta_R^{\otimes f} \right\|_{\mathrm{TV}} \leq C(2n-1)^{-D}.
		\end{equation}
		Here $\mathcal{L}(\cdot)$ denotes the joint law of the variables inside with respect to the measure $\P_\ell$ of the random relators of $G_{\ell}$, and $\eta_R$ is the stationary non-backtracking law given in Definition \ref{defn: stationary non-backtracking law on NBR}.
	\end{lem}
	
	\begin{proof}
		Since each of the $m$ relators has independent distributions, it suffices to study windows lying on one given specific relator.
		
		Suppose that the roots of $u\leq f$ windows in the same relator are listed in the cyclic order of $\Z/\ell\Z$. Let us fix prescribed values $\mathbf{y}_1,\ldots,\mathbf{y}_u\in\mathrm{NB}_R$ for the corresponding windows, and let $\alpha_i$ and $\beta_i$ be, respectively, the first and last letters of $\mathbf{y}_i$ for each $1\le i \le u$.
		
		For every $1\le i \le u$, denote by $h_i$ the number of transitions from the last letter of the $i$-th window to the first letter of the next window, with the index $i$ read cyclically. Condition \eqref{eq: D separated roots} gives $h_i\geq D$, and since each window contains $2R+1$ letters, we have
		\begin{equation}\label{eq: sum gap lengths}
			\sum_{i=1}^u h_i=\ell-2Ru.
		\end{equation}
		Once the values of the windows have been prescribed, the number of ways to fill the gaps between them is exactly
		\begin{equation}\label{eq: count window filled}
			\prod_{i=1}^u N_{h_i}(\beta_i,\alpha_{i+1}),
		\end{equation}
		where $N_h$ is defined as in Lemma \ref{lem: number of cyclically reduced words} and $\alpha_{u+1}=\alpha_1$. By Equation \eqref{eq: path count}, we have
		\[
		N_{h_i}(\beta_i,\alpha_{i+1}) =\frac{(2n-1)^{h_i}}{2n} \left(1+\theta_i\right)
		\]
		with $|\theta_i|\leq2n(2n-1)^{-D}$ for each $1\le i \le u$. Since $u\leq f$ is fixed, Equations \eqref{eq: sum gap lengths} and \eqref{eq: count window filled} give
		\begin{equation*}
			\prod_{i=1}^u N_{h_i}(\beta_i,\alpha_{i+1}) =\frac{(2n-1)^{\ell-2Ru}}{(2n)^u}	\left(1+O((2n-1)^{-D})\right).
		\end{equation*}
		Dividing by $|\mathcal{C}_{\ell}|$ and using Equation \eqref{eq: number of cyclically reduced words}, we obtain
		\begin{align*}
			\P\left( W_{t,p_1}^R=\mathbf{y}_1,\ldots, W_{t,p_u}^R=\mathbf{y}_u\right)&=
			\frac{1}{(2n)^u(2n-1)^{2Ru}}\left(1+O((2n-1)^{-D})\right)\\
			& =\eta_R^{\otimes u}(\{(\mathbf{y}_1,\ldots,\mathbf{y}_u)\}) \left(1+O((2n-1)^{-D})\right).
		\end{align*}
		Summing these errors proves the total-variation	bound for one relator. Taking the product over the different relators and using independence proves that Equation \eqref{eq: separated window comparison} holds.
	\end{proof}
	We are now ready to prove the weak ergodic theorem we need.
	\begin{prop}\label{prop: weak ergodic theorem}
		Let $f\geq1$ and let $\tau\in\{1,\ldots,m\}^f$. Let $(R_\ell)_{\ell\geq1}$ be a sequence of non-negative integers such that $R_\ell=O(\log(\ell))$. For every $R\geq0$, let $\Phi_R:\mathrm{NB}_R^f\to\R$ be a function and denote by $S_{R,\tau}(\Phi_{R})$ the associated average from Definition \ref{defn: average over tau admissible roots}. Suppose that there are constants $C,A<\infty$, independent of $R$, such that
		\begin{equation}\label{eq: polynomial bound}
			\|\Phi_R\|_\infty\leq C(1+R)^A \text{ for all }R\ge 1.
		\end{equation}
		Then
		\[
		S_{R_\ell,\tau}(\Phi_{R_\ell}) -\int_{\mathrm{NB}_{R_\ell}^f} \Phi_{R_\ell}\,d\eta_{R_\ell}^{\otimes f} \xrightarrow[\ell\to\infty]{\P_{\ell}}0.
		\] 
	\end{prop}
	\begin{proof}
		In order to simplify the notation, let us denote $R=R_\ell$, $D=\lfloor\sqrt\ell\rfloor$, $M_\ell=\|\Phi_R\|_\infty$ and $I_\ell=\int\Phi_R\,d\eta_R^{\otimes f}$. Note that, since we assume that $R=O(\log\ell)$ and $f$ is fixed, we have $2f(2R+D)<\ell$ for all sufficiently large $\ell$.

		Let us say that a $\tau$-admissible root vector $\mathbf{p}\in\mathcal{R}_\ell(\tau)$ is \emph{separated} if it satisfies the condition from Equation \eqref{eq: D separated roots}. Note that for each fixed pair $i\neq j$ with $t_i=t_j$, there are at most $\ell\bigl(2(2R+D)+1\bigr)\ell^{f-2}$ root vectors that are $\tau$-admissible for which $d_\ell(p_i,p_j)\leq2R+D$. Summing over the at most $f(f-1)/2$ relevant pairs gives
		\begin{equation*}\label{eq: bound for nonseparated roots}
			\#\{\mathbf{p}\in\mathcal{R}_\ell(\tau):
			\mathbf{p}\text{ is not separated}\} = O\left((R+D)\ell^{f-1}\right).
		\end{equation*}
		Thanks to \eqref{eq: number admissible roots}, we have
		\begin{equation}\label{eq: proportion nonseparated roots}
			\frac{	\#\{\mathbf{p}\in\mathcal{R}_\ell(\tau):
				\mathbf{p}\text{ is not separated}\}}{|\mathcal{R}_\ell(\tau)|}=O\left(\frac{R+D}{\ell}\right).
		\end{equation}
		
		We can apply the same counting argument to pairs of root vectors $\mathbf{p},\mathbf{q}\in\mathcal{R}_\ell(\tau)$. Set
		\[ \mathbf{u}=(p_1,\ldots,p_f,q_1,\ldots,q_f) \ \ \text{ and }\ \ \ 
		\tau^{(2)}=(t_1,\ldots,t_f,t_1,\ldots,t_f).
		\]
		The same argument, now applied to the $2f$ centers in $\mathbf{u}$, gives
		\begin{equation}\label{eq: proportion nonseparated root pairs}
			\frac{
				\#\left\{
				(\mathbf{p},\mathbf{q})\in\mathcal{R}_\ell(\tau)^2:
				\begin{array}{c}
					d_\ell(u_i,u_j)\leq 2R+D \text{ for some } i\neq j\\
					\text{such that }t_i^{(2)}=t_j^{(2)}
				\end{array} \right\}
			}{|\mathcal{R}_\ell(\tau)|^2}=O\left(\frac{R+D}{\ell}\right).
		\end{equation}
		
		For each $\mathbf{p}=(p_1,\ldots,p_f)\in\mathcal{R}_\ell(\tau)$, let us define $X_{\mathbf{p}}\coloneqq \Phi_R\left(W_{t_1,p_1}^R,\ldots,W_{t_f,p_f}^R\right)$. If $\mathbf{p}$ is separated, then we can use Lemma \ref{lem: separated windows ergodic} to conclude that
		\[
		|\E\left(X_{\mathbf{p}}\right)-I_\ell| = O\left(M_\ell(2n-1)^{-D}\right).
		\]
		For a non-separated $\mathbf{p}$, we can bound $|\E \left(X_{\mathbf{p}}\right)-I_\ell|\leq2M_\ell$. Averaging these two bounds and using Equation \eqref{eq: proportion nonseparated roots} to bound the number of non-separated root vectors, we obtain that
		\begin{equation}\label{eq:ergodic bound}
			\left|
			\E\left( S_{R,\tau}(\Phi_R)\right)-I_\ell	\right|	=O\left(M_\ell 	\left( 	(2n-1)^{-D}+\frac{R+D}{\ell}\right)\right).
		\end{equation}
		We now estimate the variance of $S_{R,\tau}(\Phi_R)$. Suppose that the combined $2f$ centers belonging to the root vectors $\mathbf{p}$ and $\mathbf{q}$ are separated on each relator. Applying Equation \eqref{eq: separated window comparison} to these $2f$ windows shows that their joint law is within a constant factor of $(2n-1)^{-D}$ in total variation of the probability measure $\eta_R^{\otimes2f}$. In the latter law, the first $f$ windows and	the last $f$ windows are independent. It follows that
		\[
		|\mathrm{Cov}(X_{\mathbf{p}},X_{\mathbf{q}})|= O\left(M_\ell^2(2n-1)^{-D}\right).
		\]
		For the rest of the pairs, we can use the bound 
		\[
		\left|\mathrm{Cov}(X_{\mathbf{p}},X_{\mathbf{q}})\right|\leq\|X_{\mathbf{p}}-\E\left( X_{\mathbf{p}}\right)\|_\infty \|X_{\mathbf{q} }-\E\left( X_{\mathbf{q}}\right)\|_\infty \leq 4M_\ell^2.
		\]
		Using Equation \eqref{eq: proportion nonseparated root pairs} to bound the number of non-separated root pairs, we see that
		\begin{align}\label{eq: variance bound}
			\mathrm{Var}\left(S_{R,\tau}(\Phi_R)\right)&=\frac{1}{|\mathcal{R}_\ell(\tau)|^2}\sum_{\mathbf{p},\mathbf{q}\in\mathcal{R}_\ell(\tau)}\notag
			\mathrm{Cov}(X_{\mathbf{p}},X_{\mathbf{q}})\\
			&=O\left(M_\ell^2 \left((2n-1)^{-D}+\frac{R+D}{\ell}\right)\right).
		\end{align}

	By \eqref{eq: polynomial bound} we have $M_\ell=O((1+\log\ell)^A)$. Since $D=\lfloor\sqrt\ell\rfloor$, we have
	\[
	M_\ell^2\left((2n-1)^{-D}+\frac{R+D}{\ell}\right) \xrightarrow[\ell\to \infty]{}0.
	\]	
	Equations \eqref{eq:ergodic bound} and \eqref{eq: variance bound} then imply that
		\[
		\left|\E \left(S_{R_\ell,\tau}(\Phi_{R_\ell})\right)-I_\ell \right|\xrightarrow[\ell\to \infty]{}0 \text{ and }\mathrm{Var}\left(
		S_{R_\ell,\tau}(\Phi_{R_\ell})\right)\xrightarrow[\ell\to\infty]{}0.
		\]
		Fix $\varepsilon>0$. Then for all sufficiently large $\ell$ we have
		\[
		\left|\E \left(S_{R_\ell,\tau}(\Phi_{R_\ell})\right)-I_\ell\right|<\frac{\varepsilon}{2}.
		\]
		Therefore it follows from Chebyshev's inequality that
		\begin{align*}
			\P\left(\left|S_{R_\ell,\tau}(\Phi_{R_\ell})-I_\ell\right|	>\varepsilon	\right)	&\leq	\P\left(	\left|	S_{R_\ell,\tau}(\Phi_{R_\ell})	-\E \left(S_{R_\ell,\tau}(\Phi_{R_\ell})\right)	\right|
			>\frac{\varepsilon}{2}	\right)\\
			&\leq	\frac{4}{\varepsilon^2}	\mathrm{Var}\left(S_{R_\ell,\tau}(\Phi_{R_\ell})	\right)	\xrightarrow[\ell\to \infty]{}0.
		\end{align*}
		This is what we wanted to prove.
	\end{proof}
	\subsection{Definition of the observable function}\label{subsection:definition of observable}
	Let $k\geq 1$. Let us fix a non-crossing partition $\pi\in\mathrm{NC}(k)$ and a map ${\typemap}:\pi\to\{1,\ldots,m\}$. Set
	\begin{equation}\label{eq: choice of window radius}
		R_\ell\coloneqq 2k\rho_\ell,
	\end{equation}
 where $\rho_\ell$ is defined as in Equation \eqref{eq: rho ell}. The dependence on $k$ is not made explicit in the notation. We use $R_\ell$ as the radius of the relator windows below. On the event $\mathcal{G}_\ell$ of Proposition \ref{prop: lifted face geometry}, this choice ensures that a window centered at any selected boundary position contains every other selected position belonging to the same face--graph-block incidence, by 	Lemma \ref{lem: graph block reference position}. In particular, it satisfies
	$R_\ell\xrightarrow[\ell\to \infty]{}\infty$ and $R_\ell=O(\log(\ell))$. Hence, for all sufficiently large $\ell$, we have $2R_\ell+1<\ell$ and $k(2R_\ell+1)<\ell/2$.
	
	 Recall that we already have the expansion
	\[
	Z_{\ell}(\pi,{\typemap}) = \ell^{\mathrm{diag}(\pi)}
	\sum_{\mathbf{F}\in \Omega_{\ell}(\pi,{\typemap})} \sum_{\mathbf{e}\in \mathcal{E}(\mathbf{F})} \prod_{i\in \mathcal{O}_\pi} c_{F_{[i]_\pi}}(e_i)c_{F_{[i+1]_\pi}}(e_i)
	\]
	by Lemma \ref{lemma: partition off diagonal} on the event $\mathcal{G}_\ell$. Recall also that we defined
	\[
	\freesum{\pi} \coloneqq \sum_{B\in \mathcal{B}(\Gamma_\pi)}|V(B)|.
	\]
	The fact that $\pi$ is non-crossing ensures by Lemma \ref{lem: noncrossing_partitions} that $\freesum{\pi}+\mathrm{diag}(\pi)=k$.
	
	For every block $B\in\mathcal{B}(\Gamma_\pi)$, put $r_B\coloneqq |V(B)|$.
 \mynomenclature[O]{$r_B$}{The cardinality of $V(B)$, for $B\in \mathcal{B}(\Gamma_\pi)$}
 By Lemma \ref{lem: noncrossing_partitions}, every graph block is either a bridge traversed once in each direction or a simple cycle whose edges occur once. Enumerate the vertices of $B$ in the associated cyclic order $C_0^B,C_1^B,\ldots,C_{r_B-1}^B$.
 \mynomenclature[O]{$C_i^B$}{The $i$-th vertex in $V(B)$ for $B\in\mathcal{B}(\Gamma_\pi)$, where $0\leq i<r_B$ and the ordering has been chosen so that $(C_i^B,C_{i+1}^B)\in E(B)$ for every $0\leq i<r_B$. The indices are meant to be read cyclically}
 As in the preceding pages of the paper, indices are meant to be read cyclically modulo $r_B$. Choose the cyclic numbering so that the transitions belonging to $B$ are
	\[
	C_0^B\to C_1^B\to\cdots \to C_{r_B-1}^B\to C_0^B.
	\]
	For each $i$, let $\nu_B(i)\in\mathcal{O}_\pi$
 \mynomenclature[O]{$\nu_B(i)$}{The index in $\mathcal{O}_\pi\subseteq \{0,\dots,k-1\}$ corresponding to the class $C_i^B$, so that $C_i^B=[\nu_B(i)]_\pi$ for all $B\in \mathcal{B}(\Gamma_\pi)$ and $0\leq i<r_B$}
 be the corresponding index in the original trace expansion. Thus $[\nu_B(i)]_\pi=C_i^B$ and $[\nu_B(i)+1]_\pi=C_{i+1}^B$. For each $\mathbf{e}\in \mathcal{E}(\mathbf{F})$, we write $e_{B,i}\coloneqq e_{\nu_B(i)}$.
	
	 Let us define
	\[
	t_{B,i}\coloneqq{\typemap}(C_i^B), \text{ for each } 0\leq i<r_B,
	\]
	so that $t_{B,i}$ tells us which type of relator the face $F_{C_i^B}$ has as contour. Let us choose once and for all an ordering of the set
	\[
	\{(B,i):B\in\mathcal{B}(\Gamma_\pi),\ 0\leq i<r_B\},
	\]
 \mynomenclature[O]{$(B,i)$}{An arbitrary element of the set $\{(B,i):B\in \mathcal{B}(\Gamma_\pi),0\leq i<r_B\}$, which is ordered through an arbitrary choice of bijection with the set $\{1,\dots,\freesum{\pi}\}$}
	and let us denote
	\begin{equation}\label{eq: tau}
		\tau\coloneqq(t_{B,i})_{B,i} \in\{1,\ldots,m\}^{\freesum{\pi}}.
	\end{equation}
 \mynomenclature[O]{$(t_{B,i})_{B,i}$}{The tuple in $\{1,\dots,m\}^{\freesum{\pi}}$ given by $t_{B,i}=\typemap(C_i^B)$ for every $(B,i)$, where $\mathbf{F}\in\Omega_\ell(\pi,\typemap)$ is fixed}
	We use this same ordering for the coordinates of a $\tau$-admissible root vector $\mathbf{p}=(p_{B,i})_{B,i}$ and for the coordinates of an element in $\mathrm{NB}_R^{\freesum{\pi}}$.

	\begin{defn}\label{defn: currents}
		Let $T=\mathrm{Cay}(F_n,\mathcal{A})$ be the Cayley graph of the free
		group $F_n$ with respect to its symmetric generating set $\mathcal{A}$. Choose one orientation of every edge of $T$, and denote the resulting set of oriented representatives by $E^+(T)$.
 \mynomenclature[O]{$E^+(T)$}{The edge set of the Cayley graph $T$ of $F_n$, with an arbitrary choice of orientations}
		
		If $e$ is a directed edge of $T$, let $|e|\in E^+(T)$ be its underlying edge equipped with the chosen orientation. Define the vector $u_e\in\ell^2(E^+(T))$ by
		\[
		u_e\coloneqq
		\begin{cases}
			\delta_{|e|},
			&\text{if }e\text{ has the chosen orientation, and}\\
			-\delta_{|e|},
			&\text{if }e\text{ has the opposite orientation}.
		\end{cases}
		\]
 \mynomenclature[O]{$u_e$}{The unit vector in $\ell^2(E^+(T))$ associated to an arbitrary directed edge in $T$, corresponding to $\delta_e$ if $e\in E^+(T)$ and to $-\delta_{\lvert e\rvert}$ if $e\not\in E^+(T)$, where $\lvert e\rvert\in E^+(T)$ is the underlying edge equipped with the chosen orientation}
		For $g\in F_n$, define $gu_e\coloneqq u_{ge}$, for each directed edge $e$ of $T$, and extend this action linearly to $\ell^2(E^+(T))$.
	\end{defn}
	We now introduce a definition that associates to a non-backtracking tuple a geodesic path in the tree $T$. Recall the definition of the set $\mathrm{NB}_R$ of non-backtracking sequences in Equation \eqref{eq: non backtracking sequences}.
	\begin{defn}\label{def: signed current}
		Let $R\geq 0$ and let $\mathbf{y}=(y_{-R},\ldots,y_R)\in\mathrm{NB}_R$.
		\begin{itemize}
			\item Define for each $-R\leq i\leq R+1$ the vertex $v_{\mathbf{y},i}\in T$ given by
			\[
			v_{\mathbf{y},i}\coloneqq
			\begin{cases}
				y_{-1}^{-1}y_{-2}^{-1}\cdots y_i^{-1},
				&\text{if }-R\leq i<0,\\
				e_{F_n},
				&\text{if }i=0,\\
				y_0y_1\cdots y_{i-1},
				&\text{if }0<i\leq R+1.
			\end{cases}
			\]
 \mynomenclature[O]{$v_{\mathbf{y},i}$}{The element in $F_n$ corresponding to the $i$-th position determined by a non-backtracking path $\mathbf{y}\in\mathrm{NB}_R$. Thus, $v_{\mathbf{y},i}=y_0\cdots y_{i-1}$ if $0<i\leq R+1$, $v_{\mathbf{y},i}=y_{-1}^{-1}\cdots y_i^{-1}$ if $-R\leq i<0$ and $v_{\mathbf{y},0}=e_{F_n}$}
			Note that $v_{\mathbf{y},i+1}=v_{\mathbf{y},i}y_i$ for each $-R\leq i\leq R$.
			\item The directed path on $T$ with consecutive vertices $(v_{\mathbf{y},i})_{i=-R}^{R+1}$ is called the \emph{canonical window path} associated with $\mathbf{y}$. Since $\mathbf{y}$ is non-backtracking, its canonical window path is
			geodesic on $T$. Its \emph{central edge} is $\left(v_{\mathbf{y},0},v_{\mathbf{y},1}\right)=(e_{F_n},y_0)$.
			\item For $-R\leq i\leq R$, define its signed edge vector by $\varepsilon_{\mathbf{y}}(i) \coloneqq	u_{(v_{\mathbf{y},i},v_{\mathbf{y},i+1})}$.
 \mynomenclature[O]{$\varepsilon_{\mathbf{y}}(i)$}{The signed edge vector $u_{(v_{\mathbf{y},i},v_{\mathbf{y},i+1})}$ associated to $\mathbf{y}\in\mathrm{NB}_R$}
 In particular, we have $\varepsilon_{\mathbf{y}}(0)=u_{(e_{F_n},y_0)}$. Define its \emph{signed current} by
			\[
			b_{\mathbf{y}}\coloneqq \sum_{i=-R}^{R}\varepsilon_{\mathbf{y}}(i)\in \ell^2(E^+(T)).
			\]
 \mynomenclature[O]{$b_{\mathbf{y}}$}{The signed current of $\mathbf{y}\in\mathrm{NB}_R$, defined as the sum $\sum_{i=-R}^R\varepsilon_{\mathbf{y}}(i)\in \ell^2(E^+(T))$}
		\end{itemize}
	\end{defn}
	
	Suppose that $r$ windows form a cyclic chain and that their canonical paths are translated by $x_0,\ldots,x_{r-1}\in F_n$, where indices are read modulo $r$ and $x_r=x_0$. Their common absolute position in $T$ is irrelevant. Indeed, since the action of $F_n$ on $\ell^2(E^+(T))$ is unitary, setting $g_j\coloneqq x_0^{-1}x_j$, $0\leq j\leq r$, gives $g_0=g_r=e_{F_n}$ and
	\[
	\left\langle x_j\varepsilon_{\mathbf{y}_j}(0),x_{j+1}b_{\mathbf{y}_{j+1}}\right\rangle=\left\langle g_j\varepsilon_{\mathbf{y}_j}(0),g_{j+1}b_{\mathbf{y}_{j+1}}\right\rangle
	\]
	for every $0\leq j<r$. Thus we may normalize the first translated window to begin at $e_{F_n}$ without changing the product of the inner products. This justifies the next definition.
	\begin{defn}\label{defn: observables}
		Let $r\geq2$, let $R\geq0$, and let $(\mathbf{y}_0,\ldots,\mathbf{y}_{r-1})\in\mathrm{NB}_R^r$. Set $\mathbf{y}_r\coloneqq\mathbf{y}_0$, and $g_0=g_r\coloneqq e_{F_n}$. Define
		\[
		\Psi_{r,R}(\mathbf{y}_0,\ldots,\mathbf{y}_{r-1})	\coloneqq	\sum_{g_1,\ldots,g_{r-1}\in F_n}
		\prod_{j=0}^{r-1}	\left\langle	g_j\varepsilon_{\mathbf{y}_j}(0),	g_{j+1}b_{\mathbf{y}_{j+1}}
		\right\rangle.
		\]
 \mynomenclature[O]{$\Psi_{r,R}$}{The block map $\mathrm{NB}_R^{r}\to \R$}
		Define also	$\Phi_{\pi,R}:\mathrm{NB}_R^{\freesum{\pi}}\to\R$ by
		\begin{equation}\label{eq: definition of main observable}
			\Phi_{\pi,R}((\mathbf{y}_{B,j})_{B,j})
			\coloneqq
			\prod_{B\in\mathcal{B}(\Gamma_\pi)}
			\Psi_{r_B,R}
			(\mathbf{y}_{B,0},\ldots,\mathbf{y}_{B,r_B-1}).
		\end{equation}
 \mynomenclature[O]{$\Phi_{\pi,R}$}{The observable map $\mathrm{NB}_R^{\freesum{\pi}}\to \R$ obtained as the product $\prod_{B\in\mathcal{B}(\Gamma_\pi)}\Psi_{r_B,R}$ of the block maps}
	\end{defn}
	We see that the sum defining $\Psi_{r,R}$ is finite. Indeed, for every $0\leq j<r$, we have
	\begin{align*}
		\left\langle g_j\varepsilon_{\mathbf{y}_j}(0), g_{j+1}b_{\mathbf{y}_{j+1}}
		\right\rangle
		&=
		\sum_{q=-R}^{R} \left\langle g_j\varepsilon_{\mathbf{y}_j}(0),g_{j+1}\varepsilon_{\mathbf{y}_{j+1}}(q)
		\right\rangle.
	\end{align*}
	Because the path associated with $\mathbf{y}_{j+1}$ is geodesic, at most
	one term in this sum is non-zero, and it is either $1$ or $-1$. Once $g_j$ is fixed, each $q\in\{-R,\ldots,R\}$ determines at most one $g_{j+1}\in F_n$ for which the corresponding inner product is non-zero. Starting with $g_0=e_{F_n}$, there are consequently at most
	$(2R+1)^{r-1}$ tuples $(g_1,\ldots,g_{r-1})$ that contribute to $\Psi_{r,R}$. Since every contributing product has an absolute value of $1$, we obtain
	\[
	\|\Psi_{r,R}\|_\infty\leq(2R+1)^{r-1}.
	\]
	It follows from Equation \eqref{eq: definition of main observable} that
	\[
	\|\Phi_{\pi,R}\|_\infty \leq\prod_{B\in\mathcal{B}(\Gamma_\pi)}(2R+1)^{r_B-1}=(2R+1)^{|\pi|-1},
	\]
	where we used Lemma \ref{lem: counting size vertices blocks} in the last equality.
	The tuple $(\mathbf{y}_0,\ldots,\mathbf{y}_{r-1})\in\mathrm{NB}_R^r$ should be regarded as a collection of windows of radius $R$, one for each incidence in a graph block with $r$ vertices. The canonical path associated with $\mathbf{y}_j$ is a geodesic segment in the free Cayley tree $T$ that models these windows in the tree. For each $0\leq j<r$, the factor
	\[
	\left\langle g_j\varepsilon_{\mathbf{y}_j}(0), g_{j+1}b_{\mathbf{y}_{j+1}}\right\rangle
	\]
	tests whether the central edge of the $j$-th translated window occurs among the edges of the $(j+1)$-th translated window. 
	It follows that
	\[
	\prod_{j=0}^{r-1}\left\langle g_j\varepsilon_{\mathbf{y}_j}(0),g_{j+1}b_{\mathbf{y}_{j+1}} \right\rangle
	\]
	is non-zero precisely when the chosen translations place the $r$ windows in a cyclic chain: the central edge of every window must occur in the next window. In that case, the product is the product of the orientation signs of all the corresponding common edges. Thus, the value $\Psi_{r,R}(\mathbf{y}_0,\ldots,\mathbf{y}_{r-1})$ is the signed number of relative placements of these windows in $T$ which satisfy all the edge-identification conditions of the graph block.
	
	Finally,
	\[
	\Phi_{\pi,R}((\mathbf{y}_{B,j})_{B,j})= \prod_{B\in\mathcal{B}(\Gamma_\pi)}
	\Psi_{r_B,R} (\mathbf{y}_{B,0},\ldots,\mathbf{y}_{B,r_B-1})
	\]
	assembles the above over all graph blocks of $\Gamma_\pi$. The coordinate $\mathbf{y}_{B,j}$ is the window associated with the incidence of the face class $C_j^B$ with the graph block $B$. Since $\freesum{\pi}=\sum_{B\in\mathcal{B}(\Gamma_\pi)}r_B$, there is exactly one window coordinate for every such incidence. Consequently, $\Phi_{\pi,R}$ is the signed count of all local edge-identification patterns prescribed by $\pi$, expressed entirely in terms of the relator windows of radius $R$.
	
\begin{lem}
	\label{lem: comparison between Z and average of observables}
	Fix $k\ge 1$. Let $\pi\in\mathrm{NC}(k)$ with $|\pi|\geq2$, and let $\typemap:\pi\to\{1,\ldots,m\}$. Let $\tau$ and $\Phi_{\pi,R}$ be defined as in Equations \eqref{eq: tau} and \eqref{eq: definition of main observable}. Consider the event $\mathcal{G}_{\ell}$ of Proposition \ref{prop: lifted face geometry}.
	
	Then there exists a constant $C$, independent of $\ell$ and of the presentation of $G_{\ell}$, such that, for all sufficiently large $\ell$ we have 
	\[
\left|\frac{1}{\ell^k}Z_\ell(\pi,\typemap)	-	S_{R_\ell,\tau}(\Phi_{\pi,R_\ell})	\right|	\leq	C\frac{(2R_\ell+1)^{|\pi|-1}}{\ell}.	\]
	In particular, we have
	\[
	\frac{1}{\ell^k}Z_\ell(\pi,\typemap)- S_{R_\ell,\tau}(\Phi_{\pi,R_\ell})
	\xrightarrow[\ell\to\infty]{\P_{\ell}}0.
	\]
\end{lem}

	\begin{proof}
		Fix a presentation in the event $\mathcal{G}_\ell$, with $\ell$ sufficiently large. To alleviate the notation, let us write $R=R_\ell$, $f=\freesum{\pi}$ and $d=\mathrm{diag}(\pi)$. Then, we have $d+f=k$ and
		\[
		\sum_{B\in\mathcal{B}(\Gamma_\pi)}(r_B-1)=|\pi|-1.
		\]
		
	Recall that for each $1\leq j\leq m$ we write the $j$-th defining relator of $G_{\ell}$ as $r_{j,\ell}=s_{j,0}s_{j,1}\cdots s_{j,\ell-1}$, with $s_{j,l}\in\mathcal{A}$ for each $0\leq l<\ell$. Recall also that we read the indices of the letters cyclically. Define the corresponding prefix elements in $G_\ell$ by $w_{j,0}=e_{G_\ell}$ and $w_{j,l}=s_{j,0}\cdots s_{j,l-1}$ for each $1\leq l\leq\ell$. Since $w_{j,\ell}=e_{G_\ell}$, we extend this notation periodically to all $l\in\Z$. We choose each representative $F^0_{j,\ell}$ with boundary starting at $e_{G_\ell}$.
		
		For each $\mathbf{p}\in(\Z/\ell\Z)^f$, set $\mathbf{y}_{B,i}\coloneqq W^R_{t_{B,i},p_{B,i}}$ and $\mathbf{y}_{B,r_B}\coloneqq\mathbf{y}_{B,0}$. Let us define $\mathcal{H}_{R}(\mathbf{p})$ to be the set of all tuples $\mathbf{g}=(g_{B,i})_{B\in\mathcal{B}(\Gamma_\pi),\,0\leq i<r_B}$ such that
 \begin{itemize}
 \item $g_{B,i}\in F_n$ for all $B,i$ and $g_{B,0}=g_{B,r_B}=e_{F_n}$ for all $B$, and
 \item $\prod_{B\in\mathcal{B}(\Gamma_\pi)}
			\prod_{i=0}^{r_B-1}
			\left\langle
			g_{B,i}\varepsilon_{\mathbf{y}_{B,i}}(0),
			g_{B,i+1}b_{\mathbf{y}_{B,i+1}}
			\right\rangle\neq0$,
 \end{itemize}
		where $g_{B,r_B}=g_{B,0}$. Define also the sets
		\[
		\mathcal{D}_{\ell}\coloneqq
		\left\{(\mathbf{p},\mathbf{g}):\mathbf{p}\in(\Z/\ell\Z)^f,\ \mathbf{g}\in\mathcal{H}_R(\mathbf{p}) \right\},\ \ 
		\mathcal{D}_{\ell}^{\mathrm{admissible}}\coloneqq \left\{ (\mathbf{p},\mathbf{g})\in\mathcal{D}_{\ell}: \mathbf{p}\in\mathcal{R}_{\ell}(\tau) \right\},
		\]
		and
		\[
		\mathcal{K}_{\ell}\coloneqq \left\{ (\mathbf{F},\mathbf{e}): \mathbf{F}\in\Omega_{\ell}(\pi,\typemap),\ \mathbf{e}\in\mathcal{E}(\mathbf{F})
		\right\}.
		\]
		The following claim is our main tool to compare the expansion of $Z_\ell$ with the average of the observable.
		\begin{claim}\label{claim: bijection construction}
			There is a map
			\[\Xi:\mathcal{D}_{\ell}\longrightarrow (\mathcal{F}_{\ell})^\pi \times \left(E^+(\widetilde X_\ell)\right)^{|\mathcal{O}_\pi|}
			\]
			with the following properties. Let us write $\Xi(\mathbf{p},\mathbf{g}) =((F_C)_{C\in\pi},(e_i)_{i\in\mathcal{O}_\pi})$. Then:
			\begin{enumerate}
				\item\label{item: claim1} For every $B\in\mathcal{B}(\Gamma_\pi)$ and all $0\leq i<r_B$, we have
				\[
				\left\langle g_{B,i}\varepsilon_{\mathbf{y}_{B,i}}(0), g_{B,i+1}b_{\mathbf{y}_{B,i+1}} \right\rangle = c_{F_{C_i^B}}(e_{B,i}) c_{F_{C_{i+1}^B}}(e_{B,i}).
				\]
				\item\label{item: claim2} Writing
				\[
				\mathcal{D}_\ell^{\mathrm{distinct}}\coloneqq \left\{(\mathbf{p},\mathbf{g})\in\mathcal{D}_\ell: F_C\neq F_D 	\text{ for every distinct }C,D\in\pi\right\},
				\]
				the restriction of $\Xi$ to	$\mathcal{D}_\ell^{\mathrm{distinct}}$ is a bijection between $\mathcal{D}_\ell^{\mathrm{distinct}}$ and $\mathcal{K}_\ell$.
			\end{enumerate}
		\end{claim}
		
		\begin{proof}[Proof of Claim \ref{claim: bijection construction}]
	Consider $(\mathbf{p},\mathbf{g})\in\mathcal{D}_\ell$. We proceed to explain a construction that gives a collection of faces $(F_C)_{C\in\pi}$ and a collection of edges $(e_i)_{i\in\mathcal{O}_\pi}$, and then set $\Xi(\mathbf{p},\mathbf{g}) \coloneqq ((F_C)_{C\in\pi},(e_i)_{i\in\mathcal{O}_\pi})$.
			
	Recall the definition of the incidence tree $\mathcal{I}_\pi$ given in the paragraph following Definition \ref{def: block cut graph}. Let us root $\mathcal{I}_\pi$ at the vertex $C_*=[0]_\pi$, and define $F_{C_*}\coloneqq F^0_{\typemap(C_*),\ell}$.
			
		We process the graph blocks in an order compatible with this rooted tree. Each block $B$ is reached through its parent class $C_a^B$, whose face has already been determined. Write $F_{C_a^B}=h_{B,a}F^0_{t_{B,a},\ell}$ for some $h_{B,a}\in G_\ell$. Define
			\[
			z_B\coloneqq h_{B,a}w_{t_{B,a},p_{B,a}}q_\ell(g_{B,a})^{-1}\in G_\ell,
			\]
			and, for every $0\leq i<r_B$, set $h_{B,i}\coloneqq z_Bq_\ell(g_{B,i})w_{t_{B,i},p_{B,i}}^{-1}$, and $F_{C_i^B}\coloneqq h_{B,i}F^0_{t_{B,i},\ell}$.

		For $i=a$, this agrees with the face already prescribed. Since $\mathcal{I}_\pi$ is a tree, this procedure determines every face $F_C$ uniquely.
			
Given $j\in\mathcal{O}_\pi$, let $(B,i)$ be the unique pair such that $j=\nu_B(i)$. Define $e_j=e_{B,i}$ to be the edge of $\partial F_{C_i^B}$ at position $p_{B,i}$, equipped with its chosen orientation in $E^+(\widetilde X_\ell)$. Set $\mathbf{F}\coloneqq(F_C)_{C\in\pi}$, $\mathbf{e}\coloneqq(e_j)_{j\in\mathcal{O}_\pi}$, and define $\Xi(\mathbf{p},\mathbf{g})\coloneqq(\mathbf{F},\mathbf{e})$.
			
	Let us first prove that the equality in Item \eqref{item: claim1} holds. For each $B\in\mathcal{B}(\Gamma_\pi)$ and every $0\leq i<r_B$, the definition of $h_{B,i}$ ensures that, for all $-R\leq u\leq R+1$, we have
			\[
			h_{B,i}w_{t_{B,i},p_{B,i}+u}=z_Bq_\ell\left(g_{B,i}v_{\mathbf{y}_{B,i},u}\right).
			\]
			Therefore, the map $x\mapsto z_Bq_\ell(x)$ sends the canonical window path of $\mathbf{y}_{B,i}$ translated by $g_{B,i}$ to the boundary window of $F_{C_i^B}$ with root $p_{B,i}$. The central edge of the translated window maps to $e_{B,i}$, traversed in the boundary direction of $F_{C_i^B}$. Since $(\mathbf{p},\mathbf{g})\in\mathcal{D}_\ell$, we have
			\[
			\left\langle g_{B,i}\varepsilon_{\mathbf{y}_{B,i}}(0), g_{B,i+1}b_{\mathbf{y}_{B,i+1}} \right\rangle\neq 0
			\]
			for every $i$. Thus, the central edge of each translated window belongs to the next translated window. The union of these window paths is therefore connected and has at most $r_B(2R+1)\leq k(2R+1)<\ell/2$ edges. Item \eqref{item: prob 5} of Proposition \ref{prop: lifted face geometry} implies that $x\mapsto z_Bq_\ell(x)$ is injective on this union. Hence, there is an identification between the union of the translated window paths and the union of the corresponding boundary windows.
			
			For each $0\leq i<r_B$, we have
			\[
			\left\langle
			g_{B,i}\varepsilon_{\mathbf{y}_{B,i}}(0), g_{B,i+1}b_{\mathbf{y}_{B,i+1}}\right\rangle = \sum_{u=-R}^{R} \left\langle g_{B,i}\varepsilon_{\mathbf{y}_{B,i}}(0), g_{B,i+1}\varepsilon_{\mathbf{y}_{B,i+1}}(u)
			\right\rangle.
			\]
			Exactly one summand is non-zero, and the corresponding edge maps to $e_{B,i}$. The two traversal directions become the boundary directions of $F_{C_i^B}$ and $F_{C_{i+1}^B}$ along that edge. Their relative sign is therefore $c_{F_{C_i^B}}(e_{B,i})c_{F_{C_{i+1}^B}}(e_{B,i})$. This product is independent of the chosen orientation of $e_{B,i}$, since reversing that orientation changes both coefficients by a factor of $-1$. We conclude that
			\[
			\left\langle g_{B,i}\varepsilon_{\mathbf{y}_{B,i}}(0), g_{B,i+1}b_{\mathbf{y}_{B,i+1}} \right\rangle = c_{F_{C_i^B}}(e_{B,i})c_{F_{C_{i+1}^B}}(e_{B,i}).
			\]
			This is the desired equality.
			
			Next let us prove that Item \eqref{item: claim2} holds. The identification established in the preceding paragraphs shows that $e_{B,i}\in \partial F_{C_i^B}\cap\partial F_{C_{i+1}^B}$ for every $B$ and $i$. Indeed, the central edge of the $i$-th translated window belongs to the next translated window, and both occurrences map to $e_{B,i}$.
			
		If $(\mathbf{p},\mathbf{g})\in\mathcal{D}_\ell^{\mathrm{distinct}}$, then the associated tuple $\mathbf{F}$ belongs to $\Omega_\ell(\pi,\typemap)$. To see this, note that consecutive distinct faces share an edge. By Item \eqref{item: prob 2} of Proposition \ref{prop: lifted face geometry}, their intersection is a single path. Their relative boundary orientation is constant along that path, so we have
			\[
			|h_\ell(F_{C_i^B},F_{C_{i+1}^B})| = \#\{\text{edges in } \partial F_{C_i^B}\cap\partial F_{C_{i+1}^B}\} >0.
			\]
			Also, by construction, we have $\mathrm{type}(F_{C_i^B}) =t_{B,i} = \typemap(C_i^B)$ and $F_{C_*}=F^0_{\typemap(C_*),\ell}$. The diagonal factors are non-zero because every face has boundary of length $\ell$. Thus $\mathbf{F}\in\Omega_\ell(\pi,\typemap)$ and $\mathbf{e}\in\mathcal{E}(\mathbf{F})$. We conclude that the restriction $\Xi|_{\mathcal{D}_\ell^{\mathrm{distinct}}}: \mathcal{D}_\ell^{\mathrm{distinct}}\longrightarrow\mathcal{K}_\ell$ is well-defined.
			
	To show that this restriction is a bijection, consider $(\mathbf{F},\mathbf{e})\in\mathcal{K}_\ell$ and define $p_{B,i}\coloneqq \mathrm{pos}_{F_{C_i^B}}(e_{B,i})$ and $\mathbf{p}[\mathbf{e}]\coloneqq(p_{B,i})_{B,i}$. Since $\pi\in\mathrm{NC}(k)$, we have
			$|\mathcal{O}_\pi|=\freesum{\pi}$, so $\mathbf{p}[\mathbf{e}]\in(\Z/\ell\Z)^f$.
			
		Every lifted face of type $j$ can be written uniquely as $F=hF^0_{j,\ell}$, with $h\in G_\ell$. Write $\beta_F(p)\coloneqq hw_{j,p}$ for $p\in\Z/\ell\Z$. Thus the directed boundary edge at position $p$ is $(\beta_F(p),\beta_F(p+1))$. By Item \eqref{item: prob 2} of Proposition \ref{prop: lifted face geometry}, the map $\beta_F:\Z/\ell\Z\to\widetilde{X}_\ell^{(0)}$ is injective.
			
	For each graph block $B$ and $0\leq i<r_B$, let $\mathbf{y}_{B,i}\coloneqq W^R_{t_{B,i},p_{B,i}}$. Lemma \ref{lem: graph block reference position} ensures that this window contains both edges $e_{B,i-1}$ and $e_{B,i}$. Fix a block $B$, and let $D_B$ be the union of these $r_B$ boundary windows, viewed as a subgraph of $\widetilde{X}_\ell^{(1)}$. Two successive windows share the selected edge $e_{B,i}$, so $D_B$ is connected. Moreover, we have
			\[
			|E(D_B)|\leq r_B(2R+1)\leq k(2R+1)<\ell/2.
			\]
			If $D_B$ contained a cycle, a simple cycle in $D_B$ would have a non-empty freely reduced label representing the identity in $G_\ell$, of length less than $\ell/2$. This contradicts Item \eqref{item: prob 5} of Proposition \ref{prop: lifted face geometry}. Therefore $D_B$ is a tree. Hence there is a unique label-preserving embedding $\iota_B:D_B\to T=\mathrm{Cay}(F_n,\mathcal{A})$ that maps $\beta_{F_{C_0^B}}(p_{B,0})$ to $e_{F_n}$.
			
		Define $g_{B,i}\coloneqq\iota_B\left(\beta_{F_{C_i^B}}(p_{B,i})\right)$ for each $B,i$. Then $g_{B,0}=e_{F_n}$, and for each $-R\leq u\leq R+1$, we have
			\[
			\iota_B\left(\beta_{F_{C_i^B}}(p_{B,i}+u)\right)= g_{B,i}v_{\mathbf{y}_{B,i},u}.
			\]
			Since the selected edge $e_{B,i}$ belongs to the next boundary window, the corresponding inner product is non-zero. Thus $\mathbf{g}=(g_{B,i})_{B,i}\in\mathcal{H}_R(\mathbf{p}[\mathbf{e}])$.
			
			Set $\Upsilon(\mathbf{F},\mathbf{e}) \coloneqq(\mathbf{p}[\mathbf{e}],\mathbf{g})$. The construction of $\Xi$ recovers the prescribed root face and the relative placements of all windows, and hence recovers $(\mathbf{F},\mathbf{e})$. In particular, $\Upsilon$ takes values in $\mathcal{D}_\ell^{\mathrm{distinct}}$.
			Conversely, starting with $(\mathbf{p},\mathbf{g})\in\mathcal{D}_\ell^{\mathrm{distinct}}$, the roots are recovered as the positions of the selected edges. For each block, the injective projection of its translated windows has inverse equal to the normalized embedding $\iota_B$, so the translations $\mathbf{g}$ are recovered as well.
		
		Therefore $\Upsilon$ and $\Xi\vert_{\mathcal{D}_\ell^{\mathrm{distinct}}}$ are inverses of each other. This finishes the proof.
		\end{proof}

	From the definition of $\mathcal{D}_\ell^{\mathrm{admissible}}$ and distributivity of multiplication over addition, we have
		\begin{equation}\label{eq:average sum rewritten}
			|\mathcal{R}_\ell(\tau)|S_{R,\tau}(\Phi_{\pi,R}) = \sum_{(\mathbf{p},\mathbf{g})\in\mathcal{D}_\ell^{\mathrm{admissible}}}
			\prod_{B\in\mathcal{B}(\Gamma_\pi)}
			\prod_{i=0}^{r_B-1}\left\langle g_{B,i}\varepsilon_{\mathbf{y}_{B,i}}(0),g_{B,i+1}b_{\mathbf{y}_{B,i+1}} \right\rangle.
		\end{equation}
		On the other hand, using Claim \ref{claim: bijection construction} and the expression of $Z_\ell(\pi,\typemap)$ from Lemma \ref{lemma: partition off diagonal}, we get
		\begin{equation}\label{eq:yet another rewriting of Z}
			\ell^{-d}Z_\ell(\pi,\typemap)= \sum_{(\mathbf{p},\mathbf{g})\in\mathcal{D}_\ell^{\mathrm{distinct}}}\prod_{B\in\mathcal{B}(\Gamma_\pi)} \prod_{i=0}^{r_B-1}\left\langle g_{B,i}\varepsilon_{\mathbf{y}_{B,i}}(0), g_{B,i+1}b_{\mathbf{y}_{B,i+1}}\right\rangle.
		\end{equation}
		Hence, because all terms in these sums have absolute value $1$, it suffices to bound the cardinality of the symmetric difference of
		$\mathcal{D}_\ell^{\mathrm{distinct}}$ and $\mathcal{D}_\ell^{\mathrm{admissible}}$. We do so by bounding the cardinalities of $\mathcal{D}_\ell\backslash\mathcal{D}_\ell^{\mathrm{distinct}}$ and $\mathcal{D}_\ell\backslash\mathcal{D}_\ell^{\mathrm{admissible}}$ in the next two claims.
		
		\begin{claim}\label{claim: D distinct}
			We have
			\[
			\left| \mathcal{D}_\ell\backslash\mathcal{D}_\ell^{\mathrm{distinct}}\right| \leq \binom{|\pi|}{2}\ell^{f-1}(2R+1)^{|\pi|-1}.
			\]
		\end{claim}
		
		\begin{proof}[Proof of Claim \ref{claim: D distinct}]
			We first fix an order in which to introduce the face classes. Root the incidence tree $\mathcal{I}_\pi$ at $C_*=[0]_\pi$ as above. For each block $B$, let $C_{a(B)}^B$ be its parent class. Starting from this class, follow the cyclic order $C_{a(B)}^B,C_{a(B)+1}^B,\ldots,C_{a(B)+r_B-1}^B$, where the indices are read modulo $r_B$, and omit the final transition back to $C_{a(B)}^B$.
			
			Process the blocks in a fixed order consistent with the rooted incidence tree, introducing all their remaining classes in the displayed order.
			Since $\mathcal{I}_\pi$ is a tree, every class $C\neq C_*$ is introduced exactly once. For each such class, denote its introduction incidence by $(B(C),i(C))$, so that $C=C_{i(C)}^{B(C)}$, and write $P(C)\coloneqq C_{i(C)-1}^{B(C)}$ for the preceding class. By construction, $P(C)$ is introduced before $C$.
			
			Let $(\mathbf{p},\mathbf{g})\in\mathcal{D}_\ell$. For each $C\neq C_*$, put $B=B(C)$ and $i=i(C)$, and record the unique value $u_C\in\{-R,\ldots,R\}$ satisfying
			\[
			\left\langle g_{B,i-1}\varepsilon_{\mathbf{y}_{B,i-1}}(0), g_{B,i}\varepsilon_{\mathbf{y}_{B,i}}(u_C)\right\rangle\neq0.
			\]
			Such an index exists since $\mathbf{g}\in\mathcal{H}_R(\mathbf{p})$, and it is unique because the canonical window path of $\mathbf{y}_{B,i}$ is geodesic.
		
			There are $|\pi|-1$ recorded such values, so there are at most $(2R+1)^{|\pi|-1}$ possible records, independently of the root vector.
			
			Let us first explain why the full root vector $\mathbf{p}$ and the values $(u_C)_{C\neq C_*}$ determine at most one element of $\mathcal{D}_\ell$. First, we have that the root face $F_{C_*}$ is prescribed. Suppose now that $C=C_i^B$ is the next class to be introduced and that $F_{P(C)}=F_{C_{i-1}^B}$ has already been determined. Its selected outgoing edge is the boundary edge at position $p_{B,i-1}$. The recorded offset requires this edge to coincide with the boundary edge of $F_C$ at position $p_{B,i}+u_C\text{ mod }\ell.$ The type of $F_C$, this boundary position, and the already determined edge specify at most one lifted face $F_C$. Indeed, the letters at the two positions determine whether the boundary directions agree or disagree, and the corresponding boundary vertex then determines the lift uniquely. Proceeding in the chosen order determines all the faces.
		
		Note that the same data also determine the normalized translations uniquely. In a block $B$, consider one of its recorded transitions from $C_{i-1}^B$ to $C_i^B$. Let $a=(\mathbf{y}_{B,i-1})_0$ and $b=(\mathbf{y}_{B,i})_{u_C}$. If the corresponding edges can be identified, either $b=a$ and their directions agree, or $b=a^{-1}$ and their directions disagree. In these two cases, respectively, we have
			\[
			g_{B,i}=g_{B,i-1}v_{\mathbf{y}_{B,i},u_C}^{-1} \ \ \text{ or }\ \ \ 
			g_{B,i}=g_{B,i-1}v_{\mathbf{y}_{B,i},u_C+1}^{-1}.
			\]
			These alternatives are determined by the letters. The recorded transitions form a path through all vertices of $B$, so these relations determine all its translations up to a common left translation. The normalization $g_{B,0}=e_{F_n}$ determines them uniquely. Thus, the full root vector and the offsets determine at most one element of $\mathcal{D}_\ell$.
			
		Now, let us fix distinct classes $D,C\in\pi$, with $D$ preceding $C$ in the chosen order, and consider configurations satisfying $F_C=F_D$. Let us write $B=B(C)$ and $i=i(C)$. We claim that the offsets and all root coordinates except $p_{B,i}$ determine at most one such configuration.
			
		Indeed, the preceding reconstruction determines every face introduced before $C$ without using $p_{B,i}$. In particular, it determines $F_D$, $F_{P(C)}$, and the selected outgoing edge $e$ of $F_{P(C)}$ at position $p_{B,i-1}$.
		
		If $F_C=F_D$, then $e$ belongs to $\partial F_D$. By Item \eqref{item: prob 2} of Proposition \ref{prop: lifted face geometry}, the boundary contour of $F_D$ is embedded, so there is a unique position $q=\mathrm{pos}_{F_D}(e)\in\Z/\ell\Z$. By the definition of $u_C$ and the construction of $\Xi$, this same edge occurs on $F_C$ at position $p_{B,i}+u_C$. Consequently,	$p_{B,i}=q-u_C\text{ mod }\ell$. This recovers the missing root coordinate. The full root vector and the offsets then determine at most one configuration, as shown above.
			
	There are at most $\ell^{f-1}$ choices for the remaining root coordinates and $(2R+1)^{|\pi|-1}$ choices for the offsets. Thus, for the fixed pair $D,C$, there are at most $\ell^{f-1}(2R+1)^{|\pi|-1}$ configurations with $F_C=F_D$. Every element of $\mathcal{D}_\ell\backslash\mathcal{D}_\ell^{\mathrm{distinct}}$ has at least one such pair. Summing over the at most $\binom{|\pi|}{2}$ pairs proves the claim.
		\end{proof}
		
		\begin{claim}\label{claim: D admiss}
			We have
			\[
			\left|\mathcal{D}_\ell\backslash\mathcal{D}_\ell^{\mathrm{admissible}} \right| \leq \binom{f}{2}\ell^{f-1}(2R+1)^{|\pi|-1}.
			\]
		\end{claim}
		
		\begin{proof}[Proof of Claim \ref{claim: D admiss}]
			There are at most $\binom{f}{2}\ell^{f-1}$ non-admissible root vectors, since non-admissibility requires equality of two root coordinates with the same relator type. For each fixed root vector, the offset encoding from the preceding proof gives at most $(2R+1)^{|\pi|-1}$ possible normalized translation tuples in $\mathcal{H}_R(\mathbf{p})$. The multiplication of these bounds gives the bound that appears in the claim.
		\end{proof}
		
	Using Equations \eqref{eq:average sum rewritten} and \eqref{eq:yet another rewriting of Z}, Claims \ref{claim: D distinct} and \ref{claim: D admiss}, and the inequality
		\[
		\left|\mathcal{D}_\ell^{\mathrm{distinct}} \triangle\mathcal{D}_\ell^{\mathrm{admissible}}\right|\leq \left| \mathcal{D}_\ell\backslash\mathcal{D}_\ell^{\mathrm{distinct}}\right|+\left|\mathcal{D}_\ell\backslash\mathcal{D}_\ell^{\mathrm{admissible}}	\right|,
		\]
		we obtain
		\[
		\left|\ell^{-d}Z_\ell(\pi,\typemap)- |\mathcal{R}_\ell(\tau)|S_{R,\tau}(\Phi_{\pi,R})\right|=O\left(\ell^{f-1}(2R+1)^{|\pi|-1}\right).
		\]
		Finally, using that $|\mathcal{R}_\ell(\tau)|=\ell^f+O(\ell^{f-1})$ and $\|\Phi_{\pi,R}\|_\infty\leq(2R+1)^{|\pi|-1}$, we get
		\begin{align*}
			\left|\frac{Z_\ell(\pi,\typemap)}{\ell^k}-S_{R,\tau}(\Phi_{\pi,R})\right|&\leq\frac1{\ell^f}\left|\frac{Z_\ell(\pi,\typemap)}{\ell^d} -|\mathcal{R}_\ell(\tau)|S_{R,\tau}(\Phi_{\pi,R}) \right|\\
			&\ \ +\left|1-\frac{|\mathcal{R}_\ell(\tau)|}{\ell^f}\right|\left|S_{R,\tau}(\Phi_{\pi,R})\right|\\
			&=O\left(\frac{(2R+1)^{|\pi|-1}}{\ell}\right).
		\end{align*}
		This proves the claimed inequality on $\mathcal{G}_\ell$. Since $R_\ell=O(\log\ell)$, the last bound tends to zero as $\ell\to\infty$. Together with the fact that $\P_\ell(\mathcal{G}_\ell)\xrightarrow[\ell\to \infty]{}1$, this also proves the convergence in probability.
	\end{proof}

 \subsection{Computing the expectation of the observable}\label{subsection:expected value of observable}
	Next, let us compute the value of
	\[
	\int_{\mathrm{NB}_R^{\freesum{\pi}}} \Phi_{\pi,R}\ d\eta_R^{\otimes\freesum{\pi}}.
	\]
	The observable $\Phi_{\pi,R}$ is a product of observables $\Psi_{r_B,R}$, one for each graph block $B\in\mathcal{B}(\Gamma_\pi)$. Since these factors depend on disjoint window coordinates and $\eta_R^{\otimes\freesum{\pi}}$ is a product measure, their expectations factor as a product. It is therefore enough to compute the value of the quantity
	\[
	\int_{\mathrm{NB}_R^r}\Psi_{r,R}\ d\eta_R^{\otimes r}
	\]
	for one graph block with $r$ vertices, where the map $\Psi_{r,R}$ is given in Definition \ref{defn: observables}. 

 Recall that we are working with the Cayley graph $T$ of the free group $F_n$. For $e,f\in E^+(T)$, define $d_T(e,f)$ to be the number of
	steps from $e$ to $f$ along the unique edge path joining them. Thus $d_T(e,e)=0$, two distinct edges sharing a vertex have distance $1$, and in general $d_T(e,f)=q$ if the unique edge path from $e$ to $f$ consists of $q+1$ edges. If $e\neq f$, orient the unique shortest edge path from $e$ to $f$. Let $\sigma(e,f)\in\{-1,1\}$ be the product of the signs with which this path traverses $e$ and $f$, relative to their chosen orientations in $E^+(T)$. Set $\sigma(e,e)=1$, and define
	\begin{equation}\label{eq: observable kernel}
		K(e,f)\coloneqq\frac{\sigma(e,f)}{n(2n-1)^{d_T(e,f)}}, \text{ and }
 K_R(e,f)\coloneqq K(e,f)\mathds{1}_{\{d_{T}(e,f)\le R\}}
	\end{equation}
 for each $R\ge 0$.
 
 The computation of the expectation of $\Psi_{r,R}$ will have as an underlying component the next lemma. 

\begin{lem}\label{lem: averaging one window}
		Recall the setting and notation from Definitions \ref{defn: stationary non-backtracking law on NBR}, \ref{defn: currents} and \ref{def: signed current}. Then for every $e,f\in E^+(T)$ and every $R\geq0$ we have
		\begin{equation}\label{eq: averaging one window}
			\int_{\mathrm{NB}_R}\sum_{g\in F_n}\left\langle u_e,gb_{\mathbf{y}}\right\rangle
			\left\langle g\varepsilon_{\mathbf{y}}(0),u_f\right\rangle
			\ d\eta_R(\mathbf{y})=K_R(e,f).
		\end{equation}
	\end{lem}
	\begin{proof}
		Fix $e,f\in E^+(T)$. Given $\mathbf{y}=(y_{-R},\dots,y_R)\in \mathrm{NB}_R$, the inner product $\left\langle g\varepsilon_{\mathbf{y}}(0),u_f\right\rangle$ is non-zero only when $(g,gy_0)$ and $f$ have the same underlying undirected edge. This uniquely determines the translating element $g$, so the sum
 \[\sum_{g\in F_n}\left\langle u_e,gb_{\mathbf{y}}\right\rangle\left\langle g\varepsilon_{\mathbf{y}}(0),u_f\right\rangle\]
 consists of only one term, which we write $\left\langle u_e,g_0b_{\mathbf{y}}\right\rangle\left\langle g_0\varepsilon_{\mathbf{y}}(0),u_f\right\rangle$. If the inner product $\left\langle u_e,g_0b_{\mathbf{y}}\right\rangle$ is non-zero then necessarily $h\coloneqq d_T(e,f)\leq R$. There are two possible directions in which the geodesic segment joining $e$ and $f$ can be traversed. For each direction, the required segment has probability $\frac{1}{2n(2n-1)^h}$ under the stationary non-backtracking law. In both directions, the product of the two signs in the term $\left\langle u_e,gb_{\mathbf{y}}\right\rangle\left\langle g\varepsilon_{\mathbf{y}}(0),u_f\right\rangle$ is $\sigma(e,f)$. Therefore, the left-hand side of \eqref{eq: averaging one window} is
		\[
		2\frac{\sigma(e,f)}{2n(2n-1)^h}=K(e,f).
		\]
		If $d_T(e,f)>R$, both sides vanish. This proves the identity.
	\end{proof}
	
\begin{lem}	\label{lem: expectation of the observable}
 For every $r\geq2$ we have
		\begin{equation}\label{eq: expectation of one graph block}
			\lim_{R\to \infty}\int_{\mathrm{NB}_R^r} \Psi_{r,R}(\mathbf{y}_0,\ldots,\mathbf{y}_{r-1})\ d\eta_R^{\otimes r}= \frac{1}{(n-1)^{r-1}}.
		\end{equation}
 	As a consequence, for every $\pi\in \mathrm{NC}(k)$ with $|\pi|\ge 2$ we have
		\begin{equation}\label{eq: expectation observable limit}
			\lim_{R\to\infty}\int_{\mathrm{NB}_R^{\freesum{\pi}}}
			\Phi_{\pi,R}\ d\eta_R^{\otimes\freesum{\pi}}
			=\left(\frac{1}{n-1}\right)^{|\pi|-1}.
		\end{equation}
	\end{lem}
 \begin{proof}
 For all $r\geq 2$, $R\geq 0$ and $\mathbf{y}_0\in\mathrm{NB}_R$, write
 \[F_{r,R}(\mathbf{y}_0)\coloneqq \int_{\mathrm{NB}_R^{r-1}}\Psi_{r,R}(\mathbf{y}_0,\dots,\mathbf{y}_{r-1})d\eta_R^{\otimes r-1}(\mathbf{y}_1,\dots,\mathbf{y}_{r-1}).\]
 Let $\Omega\subseteq \partial T\times \partial T$ be the subspace of all bi-infinite geodesics in $T$, and let $\eta$ be the probability measure on $\Omega$ satisfying 
 \[\eta(\{\gamma\in \Omega:\gamma\rvert_{[-R,R]}=\mathbf{y}\})=\eta_R(\{\mathbf{y}\})\]
 for all $\mathbf{y}\in\mathrm{NB}_R$ and $R\geq 0$. It is then immediate that
 \[\int_{\mathrm{NB}_R}F_{r,R}(\mathbf{y}_0)d\eta_R(\mathbf{y_0})=\int_{\Omega}F_{r,R}(\gamma\rvert_{[-R,R]})d\eta(\gamma).\]
 In what follows, we will prove that
 \begin{enumerate}
 \item $F_{r,R}(\gamma\rvert_{[-R,R]})$ converges to $(n-1)^{-(r-1)}$ as $R\to \infty$ for every $\gamma\in \Omega$, and
 \item the sequence $(F_{r,R})_{R\geq 0}$ is uniformly bounded by a constant $C>0$.
 \end{enumerate}
Let $\gamma\in \Omega$ and write $\mathbf{y}_0\coloneqq \gamma\rvert_{[-R,R]}\in\mathrm{NB}_R$. To make the computation less cumbersome, write $\mathbf{Y}_i\coloneqq (\mathbf{y}_i,\dots,\mathbf{y}_{r-1})$ and $\mathbf{g}_i=(g_i,\dots,g_{r-1})$ for every $1\leq i\leq r-1$. We have
\begin{align*}
 F_{r,R}(\mathbf{y}_0)&=\int\sum_{\mathbf{g}_1\in F_n^{r-1}}\langle \varepsilon_{\mathbf{y}_0}(0),g_1b_{\mathbf{y}_1}\rangle \langle g_1\varepsilon_{\mathbf{y}_1}(0),g_2b_{\mathbf{y}_2}\rangle \cdots \langle g_{r-1}\varepsilon_{\mathbf{y}_{r-1}}(0),b_{\mathbf{y}_0}\rangle d\eta_R^{\otimes r-1}(\mathbf{Y}_1).
\end{align*}
By writing 
\[\langle g_1\varepsilon_{\mathbf{y}_1}(0),g_2b_{\mathbf{y}_2}\rangle=\sum_{e_1\in E^+(T)}\langle g_1\varepsilon_{\mathbf{y}_1}(0),u_{e_1}\rangle\langle u_{e_1},g_2b_{\mathbf{y}_2}\rangle,\]
we obtain
\begin{align*}
 F_{r,R}(\mathbf{y}_0)=\int\sum_{e_1}\sum_{\mathbf{g}_2}&\left(\int\sum_{g_1}\langle \varepsilon_{\mathbf{y}_0}(0),g_1b_{\mathbf{y}_1}\rangle \langle g_1\varepsilon_{\mathbf{y}_1}(0),u_{e_1}\rangle d\eta_R(\mathbf{y}_1)\right)\cdot\\
 &\langle u_{e_1},g_2b_{\mathbf{y}_2}\rangle \cdots \langle g_{r-1}\varepsilon_{\mathbf{y}_{r-1}}(0),b_{\mathbf{y}_0}\rangle d\eta_R^{\otimes r-2}(\mathbf{Y}_2).
\end{align*}
The oriented bi-infinite geodesic $\gamma$ in $T$ determines, for each $e\in E^+(T)$, a signed coefficient
	\[ c_\gamma(e)\coloneqq
	\begin{cases}
		1,&\gamma\text{ traverses }e\text{ in its chosen orientation},\\
		-1,&\gamma\text{ traverses }e\text{ in the opposite orientation, and }\\
		0,&e\text{ does not belong to }\gamma.
	\end{cases}
	\]
In particular, we can write $ \varepsilon_{\mathbf{y}_0}(0)=c_\gamma(e_0)u_{e_0}$ for some $e_0\in E^+(T)$. Using Lemma \ref{lem: averaging one window}, we get
\[F_{r,R}(\mathbf{y}_0)=\sum_{e_1\in B(e_0,R)}c_\gamma(e_0)K(e_0,e_1)\int \sum_{\mathbf{g}_2}\langle u_{e_1},g_2b_{\mathbf{y}_2}\rangle \cdots \langle g_{r-1}\varepsilon_{\mathbf{y}_{r-1}}(0),b_{\mathbf{y}_0}\rangle d\eta_R^{\otimes r-1}(\mathbf{Y}_2).\]
Iterating the process, we find that
\begin{align*}
 F_{r,R}(\mathbf{y}_0)=\sum_{e_1,\dots,e_{r-2}}c_\gamma(e_0)\prod_{i=1}^{r-2}K(e_{i-1},e_{i})\int \sum_{g_{r-1}}\langle u_{e_{r-2}},g_{r-1}b_{\mathbf{y}_{r-1}}\rangle\langle g_{r-1}\varepsilon_{\mathbf{y}_{r-1}}(0),b_{\mathbf{y}_0}\rangle d\eta_R(\mathbf{y}_{r-1}),
\end{align*}
where the sum is over tuples satisfying $e_{i}\in B(e_{i-1},R)$ for each $1\leq i\leq r-2$. Expanding the current $b_{\mathbf{y}_0}$ and applying Lemma \ref{lem: averaging one window} once again, we arrive to
\[F_{r,R}(\mathbf{y}_0)=\sum_{e_1,\dots,e_{r-2}}c_\gamma(e_0)\prod_{i=1}^{r-2}K(e_{i-1},e_{i})\sum_{q=-R}^{R}\mathds{1}_{\{d_T(e_{r-2},e(q))\leq R\}}K(e_{r-2},e(q)),\]
where $e(q)$ is the underlying edge of $\varepsilon_{\mathbf{y}_0}(q)$ for each $-R\leq q\leq R$. The edge $e(q)$ is contained in $\gamma$ at distance at most $R$ from $e_0$. Hence, taking into account the direction in which this edge is traversed by $\gamma$ we obtain the following expression.
\begin{equation*}\label{eq: nice expression}
 F_{r,R}(\mathbf{y}_0)=\sum_{e_1,\dots,e_{r-1}}c_\gamma(e_0)c_\gamma(e_{r-1})\prod_{i=1}^{r-1}K(e_{i-1},e_{i}),
\end{equation*}
where the sum is over tuples satisfying $e_{i}\in B(e_{i-1},R)$ for all $1\leq i\leq r-1$ and $e_{r-1}\in B(e_0,R)$. In order to understand $F_{r,R}(\mathbf{y}_0)$ as $R\to \infty$, we need the following two claims.

\begin{claim}\label{claim: unrestricted sum}
 The unrestricted sum
 \[F_{r,\infty}(\mathbf{y}_0)\coloneqq \sum_{e_1,\dots,e_{r-1}}c_\gamma(e_0)c_\gamma(e_{r-1})\prod_{i=1}^{r-1}K(e_{i-1},e_{i}) \]
 is absolutely convergent, and $|F_{r,R}(\mathbf{y}_0)-F_{r,\infty}(\mathbf{y}_0)|\to 0$ as $R\to \infty$.
\end{claim}

\begin{proof}
 For $\mathbf{e}=(e_1,\ldots,e_{r-1})$, let us put 
		\[
		L(\mathbf{e})\coloneqq\sum_{j=1}^{r-1}d_T(e_{j-1},e_j).
		\]
		Only tuples for which $e_{r-1}$ belongs to $\gamma$ contribute to the sum. For each fixed $r$, the number of such tuples satisfying $L(\mathbf{e})=h$ is at most
		\begin{equation}\label{eq: edge tuple bound expectation}
			C_r(1+h)^{A_r}(2n-1)^{h/2}
		\end{equation}
		for constants $C_r,A_r<\infty$. Indeed, we can join successive selected edges
		by their unique shortest paths and close the resulting route by the segment of $\gamma$ from $e_{r-1}$ to $e_0$. Every edge outside $\gamma$ on this closed route is traversed at least twice. Hence, at most $h/2+r$ edges may be chosen away from $\gamma$; their labels give at	most a constant multiple of $(2n-1)^{h/2}$ possibilities, while the finitely many choices of the shape and the lengths of the connecting segments give the polynomial factor.
		
		By the definition of the kernel $K$, a summand with $L(\mathbf{e})=h$ has absolute value at most $n^{-(r-1)}(2n-1)^{-h}$. Together with Equation \eqref{eq: edge tuple bound expectation}, this gives the convergent upper bound
		\[
		C_r\sum_{h\geq0}(1+h)^{A_r}(2n-1)^{-h/2}<\infty.
		\]
		Moreover, we also obtain that
		\begin{equation*}\label{eq: block expectation error}
			|F_{r,R}(\mathbf{y}_0)-F_{r,\infty}(\mathbf{y}_0)|	\leq C_r(1+R)^{A_r}(2n-1)^{-R/2},
		\end{equation*}
		because every tuple that is not $R$-spaced satisfies $L(\mathbf{e})>R$.
\end{proof}

\begin{claim}\label{claim: contraction identity}
		For every $e\in E^+(T)$ we have that
		\begin{equation}\label{eq: contraction identity}
			\sum_{f\in E^+(T)}K(e,f)c_\gamma(f) =\frac{1}{n-1}c_\gamma(e).
		\end{equation}
		The series on the left is absolutely convergent.
	\end{claim}
	
	\begin{proof}
		Suppose first that $e$ belongs to $\gamma$. The line $\gamma$ contains one edge at distance $0$ from $e$ and two edges at every distance $h\geq1$. Moreover, for every $f$ on $\gamma$, we have
		$\sigma(e,f)c_\gamma(f)=c_\gamma(e)$. It follows from Equation \eqref{eq: observable kernel} that
		\begin{equation*}
			\sum_{f\in E^+(T)}K(e,f)c_\gamma(f)=\frac{c_\gamma(e)}{n}\left(1+2\sum_{h\geq1}(2n-1)^{-h}\right)=\frac{c_\gamma(e)}{n}\left(1+\frac{1}{n-1}\right)=\frac{1}{n-1}c_\gamma(e).
		\end{equation*}
		
		Suppose now that $e$ does not belong to $\gamma$. Let $v$ be the vertex of $\gamma$ closest to $e$. Pair the edges on the two rays of $\gamma$ starting at $v$ according to their distance from $v$. The two edges in each pair have the same distance from $e$ and opposite signed contributions, so they cancel. The sum is therefore zero, which agrees with the right-hand side of Equation \eqref{eq: contraction identity}. Absolute convergence follows in both cases from the geometric series $\sum_{h\geq0}(2n-1)^{-h}$.
	\end{proof}

 Using Claim \ref{claim: unrestricted sum} and iteratively applying Claim \ref{claim: contraction identity} yield
 \begin{align*}
 \lim_{R\to\infty}F_{r,R}(\mathbf{y}_0)=\sum_{e_1,\dots,e_{r-1}}c_\gamma(e_0)c_\gamma(e_{r-1})\prod_{i=1}^{r-1}K(e_{i-1},e_{i})=\frac{c_\gamma(e_0)^2}{(n-1)^{r-1}}=\frac{1}{(n-1)^{r-1}}.
 \end{align*}
 Note that $F_{r,\infty}(\mathbf{y}_0)$ does not depend on $\gamma$. Indeed, a tree automorphism takes any oriented geodesic line with a distinguished edge to any other such pair and preserves all distance restrictions. The auxiliary choices of edge orientations do not affect the summands, as changing one such orientation switches an even number of signs in each product.

 We conclude from the above and from the second assertion in Claim \ref{claim: unrestricted sum} that $F_{r,R}(\mathbf{y}_0)$ is bounded above by a constant $C>0$ not depending on $R$ or $\gamma$. The Dominated Convergence Theorem now gives
 \begin{equation*}
 \lim_{R\to\infty}\int_{\mathrm{NB}_R^{r}}\Psi_{r,R}\,d\eta_R^{\otimes r}=\int_{\Omega}\frac{1}{(n-1)^{r-1}}d\eta(\gamma)=\frac{1}{(n-1)^{r-1}}.
 \end{equation*}

Finally, let us deduce that Equation \eqref{eq: expectation observable limit} holds. We have from Equation \eqref{eq: definition of main observable} that $\Phi_{\pi,R}=\prod_{B\in\mathcal{B}(\Gamma_\pi)}\Psi_{r_B,R}$. The factors on the right depend on disjoint groups of coordinates of $\mathrm{NB}_R^{\freesum{\pi}}$. Since $\eta_R^{\otimes\freesum{\pi}}$ is a product measure, we have
		\[
		\int\Phi_{\pi,R}\ d\eta_R^{\otimes\freesum{\pi}}=	\prod_{B\in\mathcal{B}(\Gamma_\pi)}	\int\Psi_{r_B,R}\ d\eta_R^{\otimes r_B}.
		\]
		
	Suppose that $|\pi|\geq 2$. Thanks to Equation \eqref{eq: expectation of one graph block}, we have that
		\[
		\lim_{R\to \infty}\prod_{B\in\mathcal{B}(\Gamma_\pi)}	\int\Psi_{r_B,R}\ d\eta_R^{\otimes r_B}=\prod_{B\in\mathcal{B}(\Gamma_\pi)}
		(n-1)^{-(r_B-1)}=(n-1)^{-\sum_B(r_B-1)}=\left(\frac{1}{n-1}\right)^{|\pi|-1},
		\]
		where in the last equality we used Lemma \ref{lem: counting size vertices blocks}. This proves the convergence in Equation \eqref{eq: expectation observable limit}.
 \end{proof}

\subsection{Proof of Proposition \ref{prop:main_technical}}\label{subsection: proof of the technical proposition}
Fix $\pi\in\mathrm{NC}(k)$ and $\typemap:\pi\to\{1,\ldots,m\}$.

If $|\pi|=1$, the only face is the prescribed root face. By Item \eqref{item: prob 2} of Proposition \ref{prop: lifted face geometry}, on the event $\mathcal{G}_\ell$ we have $Z_\ell(\pi,\typemap)=\ell^k$. Since $\P_\ell(\mathcal{G}_\ell)\xrightarrow[\ell\to \infty]{}1$, the proposition follows in this case.

Suppose now that $|\pi|\geq2$. Let us consider $\tau$ as defined in Equation \eqref{eq: tau} and $R_\ell=2k\rho_\ell$ as in Equation \eqref{eq: choice of window radius}. Put
\[
I_\ell =\int_{\mathrm{NB}_{R_\ell}^{\freesum{\pi}}} \Phi_{\pi,R_\ell}\ d\eta_{R_\ell}^{\otimes\freesum{\pi}}.
\]
The observable satisfies $\|\Phi_{\pi,R_\ell}\|_\infty
\leq(2R_\ell+1)^{\freesum{\pi}}$, with $R_\ell=O(\log(\ell))$. Therefore, the weak ergodic theorem from Proposition \ref{prop: weak ergodic theorem} gives
\[
S_{R_\ell,\tau}(\Phi_{\pi,R_\ell})-I_\ell \xrightarrow[\ell\to\infty]{\text{in probability}}0.
\]
By Lemma \ref{lem: expectation of the observable}, we also have $I_\ell\xrightarrow[\ell\to \infty]{}(n-1)^{-(|\pi|-1)}$.

Fix $\varepsilon>0$. For all sufficiently large $\ell$, Lemma \ref{lem: comparison between Z and average of observables}
gives on the event $\mathcal{G}_\ell$ that
\[
\left|\ell^{-k}Z_\ell(\pi,\typemap) -S_{R_\ell,\tau}(\Phi_{\pi,R_\ell})\right|\leq \varepsilon/3.
\]
On the other hand, the convergence of $I_\ell$ gives
\[
|I_\ell-(n-1)^{-(|\pi|-1)}|\leq \varepsilon/3.
\]
Therefore, we have
\[
\P_\ell\left(
	\left|\ell^{-k}Z_\ell(\pi,\typemap)
	-(n-1)^{-(|\pi|-1)}\right|\ge\varepsilon
	\right)\leq \P_\ell(\mathcal{G}_\ell^c)
	+\P_\ell\left(
	|S_{R_\ell,\tau}(\Phi_{\pi,R_\ell})-I_\ell|\ge\varepsilon/3
	\right)\xrightarrow[\ell\to \infty]{}0.
\]

Since $\varepsilon>0$ was arbitrary, this proves the desired convergence in probability. \qed

\subsection{The proof of Theorem \ref{thm:main spectral law}}\label{subsection:proof of spectral convergence}
We are finally ready to apply the results from Sections \ref{section: trace laplacian} and \ref{section:convergence of the spectral measure} to prove our second main theorem.
\begin{proof}[Proof of Theorem \ref{thm:main spectral law}]
Let us first establish convergence in probability of the $k$-th moments for every $k\geq 0$. 
	
	Recall that $\sigma_\infty=(n-1)\mu_\lambda+(m-n+1)\delta_0$. For $k=0$, both measures have total mass $m$, and hence we have
	\[
\left| \int_\R 1\ d\sigma_\ell(t)- \int_\R 1\ d\sigma_\infty(t)\right|=0
	\]
	for every $\ell\geq 1$.
	
	Let us now fix $k\geq 1$. By Lemma \ref{lem: partition trace}, we have
	\[
	\int_\R t^k\ d\sigma_\ell(t)=\mathrm{Tr}_m(\widehat{\Delta}_{2,\ell}^k)=	\frac{1}{\ell^k}\mathrm{Tr}_m(\Delta_{2,\ell}^k) =\frac{1}{\ell^k} \sum_{\pi\in\mathcal{P}(k)} \sum_{\typemap:\pi\to\{1,\dots,m\}}
	Z_\ell(\pi,\typemap).
	\]
	Thus, by Lemma \ref{lem:bound order l}, there are events $\mathcal{E}_\ell$ with $\P_\ell(\mathcal{E}_\ell)\to 1$ such that for all sufficiently large $\ell$ we have
	\begin{align*}
		 \left| \int_\R t^k d\sigma_\ell(t) - \frac{1}{\ell^k} \sum_{\pi\in\mathrm{NC}(k)} \sum_{\typemap\colon\pi\to\{1,\dots,m\}} Z_\ell(\pi,\typemap) \right|&\leq
		\frac{1}{\ell^k}\sum_{\pi\in\mathcal{P}(k)\backslash\mathrm{NC}(k)}	\sum_{\typemap\colon\pi\to\{1,\dots,m\}} |Z_\ell(\pi,\typemap)| \\
		&\leq \frac{1}{\ell^k} \sum_{\pi\in\mathcal{P}(k)\backslash\mathrm{NC}(k)}
		C_k\ell^{\mathrm{diag}(\pi)+\freesum{\pi}} (\log\ell)^{2k}m^{|\pi|}.
	\end{align*}
	Here $C_k$ can be chosen uniformly over all partitions and type
	maps, since there are only finitely many of them for each fixed $k$.
	
	Lemma \ref{lem: noncrossing_partitions} establishes that $\mathrm{diag}(\pi)+\freesum{\pi}<k$ whenever $\pi\in\mathcal{P}(k)\backslash\mathrm{NC}(k)$. Consequently, on the event $\mathcal{E}_\ell$, we have
	\begin{align*}
	\left| \int_\R t^k\,d\sigma_\ell(t)- \frac{1}{\ell^k} \sum_{\pi\in\mathrm{NC}(k)} \sum_{\typemap:\pi\to\{1,\dots,m\}} Z_\ell(\pi,\typemap) \right|& \leq C_k|\mathcal{P}(k)|m^k \frac{(\log\ell)^{2k}}{\ell}
		\xrightarrow[\ell\to \infty]{} 0.
	\end{align*}
	Therefore, given $\varepsilon>0$, the expression on the left is less than $\varepsilon/2$ on the event $\mathcal{E}_\ell$ for all sufficiently large $\ell$.
	
By Proposition \ref{prop:main_technical} together with a union bound, there are events $\mathcal{F}_\ell$ with $\mathbb{P}_\ell(\mathcal{F}_\ell)\to 1$ on which
	\[
	\left| \frac{Z_\ell(\pi,\typemap)}{\ell^k} - \left(\frac{1}{n-1}\right)^{|\pi|-1}\right|<	\frac{\varepsilon}{2|\mathrm{NC}(k)|m^k}
	\]
	simultaneously for every $\pi\in\mathrm{NC}(k)$ and every $\typemap\colon\pi\to\{1,\dots,m\}$.

	Thus, on the intersection $\mathcal{E}_\ell\cap\mathcal{F}_\ell$ and for all
	sufficiently large $\ell$, we have
	\begin{align*}
		\left|\int_\R t^k\,d\sigma_\ell(t)-(n-1)\sum_{\pi\in\mathrm{NC}(k)}\left(\frac{m}{n-1}\right)^{|\pi|}\right|&\leq\frac{\varepsilon}{2}+ \sum_{\pi\in\mathrm{NC}(k)} \sum_{\typemap\colon\pi\to\{1,\dots,m\}}\left|	\left(\frac{1}{n-1}\right)^{|\pi|-1}-\frac{Z_\ell(\pi,\typemap)}{\ell^k}\right|\\
		&<\frac{\varepsilon}{2}+\frac{\varepsilon}{2}=\varepsilon.
	\end{align*}

By Proposition \ref{prop:free poisson moments}, we have
\[
\int_\R t^k\ d\mu_\lambda(t)
= \sum_{\pi\in\mathrm{NC}(k)}\lambda^{|\pi|} = \sum_{\pi\in\mathrm{NC}(k)} \left(\frac{m}{n-1}\right)^{|\pi|},
\]
where $\lambda=m/(n-1)$. Since $k\geq 1$, the atom at zero does not contribute to the $k$-th moment of $\sigma_\infty$. It follows that, on the event $\mathcal{E}_\ell\cap\mathcal{F}_\ell$, we have
\[
\left|
\int_\R t^k\ d\sigma_\ell(t) - \int_\R t^k\ d\sigma_\infty(t) \right|
< \varepsilon\]
for all sufficiently large $\ell$. As $\P_\ell(\mathcal{E}_\ell\cap\mathcal{F}_\ell)\to 1$, this proves convergence in probability of every moment.

	We next choose a single sequence of events on which these moment convergences hold simultaneously in the limit. 
	
	Write $M_{\ell,k}\coloneqq \int_\R t^k\ d\sigma_\ell(t)$ and $M_{\infty,k}\coloneqq \int_\R t^k\ d\sigma_\infty(t)$. For every integer $r\geq 2$, convergence in probability of the first $r$ moments allows us to choose an integer $L_r$ such that
	\[
	\P_\ell\left( \max_{1\leq k\leq r}|M_{\ell,k}-M_{\infty,k}|>\frac{1}{r}
	\right) <\frac{1}{r} \ \ \text{for every }\ell\geq L_r.
	\]
	We may choose the integers $L_r$ to be strictly increasing and to satisfy $L_r\geq r$.
	
	For $\ell\geq L_2$, let us set $r(\ell)\coloneqq \max\{r\geq 2: L_r\leq\ell\}$, and define \[\mathcal{H}_\ell\coloneqq \left\{\max_{1\leq k\leq r(\ell)}|M_{\ell,k}-M_{\infty,k}|\leq\frac{1}{r(\ell)}\right\}.\]

	For the finitely many values $\ell\in \{1,\ldots,L_2\}$, let $\mathcal{H}_\ell$ be the whole probability space. Then we have $r(\ell)\to\infty$ and $\P_\ell(\mathcal{H}_\ell)\xrightarrow[\ell\to \infty]{} 1$. Moreover, for every fixed $k\geq 1$, the $k$-th moments converge to $M_{\infty,k}$ along any choice of presentations from the events $\mathcal{H}_\ell$.
	
Choose one presentation from each $\mathcal{H}_\ell$, and denote the corresponding tracial spectral measure by $\nu_\ell$. We have that
	\[
	\int_\R t^k\ d\nu_\ell(t) \xrightarrow[\ell\to \infty]{} \int_\R t^k\ d\sigma_\infty(t)\ \ \ \text{for every }k\geq 0.
	\]
	We will show that $\nu_\ell$ converges weakly to $\sigma_\infty$.
	
	Each measure $\nu_\ell$ is supported on $[0,\infty)$ and has total mass $m$. Convergence of the first moments gives
	\[
	\sup_{\ell\geq 1}\left\{\int_0^\infty t\ d\nu_\ell(t)\right\}<\infty.
	\]
	Consequently, we have
	\[
	\nu_\ell((K,\infty)) \leq \frac{1}{K}\int_0^\infty t\ d\nu_\ell(t)
	\]
	for every $K>0$, where the bound is uniform in $\ell$. It follows that $(\nu_\ell)_{\ell\geq 1}$ is a tight sequence of measures.
	
	Consider a weakly convergent subsequence $(\nu_{\ell_i})_{i\geq 1}$, and let $\nu$ be its weak limit. The measure $\nu$ is supported on $[0,\infty)$ and has total mass $m$. We will show that $\nu=\sigma_\infty$.
	
	Fix $k\geq 1$. Since the $(k+1)$-st moments converge, we have
	\[
	C_{k+1}\coloneqq \sup_{\ell\geq 1}	\int_0^\infty t^{k+1}\ d\nu_\ell(t)<\infty.
	\]
	By lower semicontinuity under weak convergence, we have
	\[
	\int_0^\infty t^{k+1}\ d\nu(t)\leq C_{k+1}.
	\]
	For $N>0$, let us define the bounded continuous function $q_{k,N}:\R\to \R$ as
	\[
	q_{k,N}(t)\coloneqq \min\{(\max\{t,0\})^k,N^k\}
	\]
	for each $t\in \R$. For either $\eta=\nu_{\ell_i}$ or $\eta=\nu$, we have
	\[
	0\leq\int_0^\infty t^k\ d\eta(t) - \int_\R q_{k,N}(t)\ d\eta(t) \leq
	\int_{(N,\infty)}t^k\ d\eta(t) \leq \frac{C_{k+1}}N.
	\]
	Since $\nu_{\ell_i}$ converges weakly to $\nu$, we have
	\[
	\int_\R q_{k,N}(t)\ d\nu_{\ell_i}(t) \xrightarrow[\ell \to \infty]{}\int_\R q_{k,N}(t)\ d\nu(t).
	\]
	Combining the last two estimates gives
	\[
	\limsup_{i\to\infty}
	\left|\int_0^\infty t^k\ d\nu_{\ell_i}(t)-\int_0^\infty t^k\ d\nu(t) \right|
	\leq \frac{2C_{k+1}}N.
	\]
	Letting $N\to\infty$, and using convergence of the moments of $\nu_\ell$, we obtain
 \[
	\int_0^\infty t^k\ d\nu(t) = \int_0^\infty t^k\,d\sigma_\infty(t).
	\]
	Thus $\nu$ and $\sigma_\infty$ have the same moments of every order.

We will now show, using the Stone-Weierstrass theorem, that $\nu=\sigma_{\infty}$. Since $\sigma_\infty$ has compact support, we can find $L>0$ such that $\supp{\sigma_\infty}\subseteq[0,L]$. For each $b>L$ and $k\ge 1$, the fact that $\nu$ and $\sigma_{\infty}$ have the same moments implies that
\[
b^k\nu([b,\infty))\le \int_0^{\infty} t^k \ d\nu(t)=\int_0^{\infty} t^k \ d\sigma_{\infty}(t)\le mL^k,
\]
so that 
\[
\nu([b,\infty))\le m\left(\frac{L}{b}\right)^k.
\]
Letting $k$ tend to infinity and using that $b>L$, we obtain $\nu([b,\infty))=0$.

Since the Stone-Weierstrass theorem guarantees that polynomials are dense with the uniform topology in $C([0,L])$, we conclude that $\nu=\sigma_{\infty}$.

We conclude from the preceding argument that every subsequential weak limit of $(\nu_\ell)_{\ell\geq 1}$ is equal to $\sigma_\infty$. By tightness of the sequence together with Prokhorov's theorem, applied to the probability
measures $m^{-1}\nu_\ell$, we conclude that the sequence $\nu_\ell$ converges weakly to $\sigma_\infty$.

We have proved that for every bounded continuous function $f:\R\to \R$ and every $\varepsilon>0$, there exists $L_0$ such that
	\[
	\left|\int_\R f(t)\ d\sigma_\ell(t)-\int_\R f(t)\ d\sigma_\infty(t) \right|< \varepsilon
	\]
	for every presentation in $\mathcal{H}_\ell$ and for every $\ell\geq L_0$. This finishes the proof of the theorem.
    
\end{proof}
\section{The proof of Theorem \ref{thm: main theorem} in the non-negative Euler characteristic case}\label{section:main theorem low deficiency}
	
In this section we present the proof of Theorem \ref{thm: main theorem} in the non-negative Euler characteristic case, that is, when $m\ge n-1$ where $n$ is the number of generators and $m$ is the number of relators of the random group $G_{\ell}$.
	
	\begin{customthm}{A.1}\label{thm:low_deficiency}
		Let $G_\ell=\langle x_1,\ldots,x_n\mid r_1,\ldots,r_m\rangle$, $\ell\ge 1$, be a random group with $n\geq 2$ generators and $m\geq n-1$ relators in the few-relator model. Then, asymptotically almost surely as $\ell\to\infty$, we have
		\begin{equation*}\label{eq: low deficiency all equalities}
			\beta_1^{(2)}(G_\ell)
			=
			\mathrm{cost}(G_\ell)-1
			=
			\rg(G_\ell)
			=
			0.
		\end{equation*}
	\end{customthm}

 \subsection{Vanishing of the first \texorpdfstring{$\ell^2$}{l2}-Betti number}\label{subsection:vanishing of Betti number}
	\begin{thm}\label{thm: vanishing first l2 betti number}
		Let $2\le n \le m+1$ and for each $\ell\ge 1$ consider the law $\P_{\ell}$ of the few-relator random group $G_{\ell}$ with $n$ generators and $m$ relators. Then
		\[
		\P_{\ell}\left(\beta_1^{(2)}(G_{\ell})=0\right)\xrightarrow[\ell\to \infty]{}1.
		\]
	\end{thm}
	\begin{proof}
 Denote by $\widehat{\Delta}_{2,\ell}$ the normalized degree-$2$ cellular $\ell^2$-Laplacian on the universal cover of the presentation complex of $G_\ell$, and denote by $\sigma_{\ell}$ the associated tracial spectral measure. Thanks to Theorem \ref{thm:main spectral law}, we know that for every $\varepsilon>0$ and every bounded continuous function $f:\R\to \R$ we have
 \begin{equation}\label{eq: conv from thm B}
\P_\ell\left(\left| \int f(t)\ d\sigma_\ell(t) - \int f(t) \ d\sigma_{\infty}(t)\right|>\varepsilon\right)\xrightarrow[\ell\to \infty]{}0,
 \end{equation}
 where $\sigma_\infty=(n-1)\mu_{\lambda}+(m-n+1)\delta_0$ and where $\mu_\lambda$ is the free-Poisson law of mean $\lambda=m/(n-1)$.

 Let us consider $\varepsilon=1/4$ and choose a bounded continuous function $f_0:\R\to \R_{\ge 0}$ such that $f_0(0)\ge 1$ and such that
 \[
 \int_{\R}f_0(t) \ d\sigma_{\infty}(t)<\sigma_\infty(\{0\})+\frac{1}{4}=m-n+1+\frac{1}{4}.
 \]

 Note that, since $f_0$ takes non-negative values and $f_0(0)\ge 1$, for each $\ell\ge 1$ we have
 \[
 \int f_0(t) \ d\sigma_{\ell} (t)-\int f_0 (t)\ d\sigma_\infty(t) \ge \sigma_{\ell}(\{0\})-\int f_0 (t)\ d\sigma_\infty(t)> \sigma_{\ell}(\{0\})-\sigma_{\infty}(\left\{0\right\})-\frac{1}{4}.
 \]
 Hence, we have the containment of events
 \[\left\{\sigma_{\ell}(\{0\})>\sigma_{\infty}(\{0\})+\frac{1}{2}\right\}\subseteq \left\{\left| \int f_0(t)\ d\sigma_\ell(t) - \int f_0(t) \ d\sigma_{\infty}(t)\right|>\frac{1}{4} \right\}.\] 
 From this, the convergence of Equation \eqref{eq: conv from thm B} guarantees that
 \[
 \P_{\ell}\left(\sigma_{\ell}(\{0\})-\sigma_\infty(\{0\})>\frac{1}{2}\right)\xrightarrow[\ell \to \infty]{}0.
 \]
 We conclude that, asymptotically almost surely as $\ell\to \infty$, we have
 \begin{equation}\label{eq: bound for sigma ell}
 \sigma_{\ell}(\{0\})\le \frac{1}{2} +\sigma_\infty(\{0\}) =\frac{1}{2}+m-n+1
 \end{equation}
 Consider a presentation $G_\ell$ in the intersection of the event $\mathcal{G}_\ell$ from Proposition \ref{prop: lifted face geometry} and the event where Equation \eqref{eq: bound for sigma ell} holds. Together with the Euler-Poincaré formula (Equation \eqref{eq: euler poincare}), we have
 \[\beta_1^{(2)}(G_\ell)=n-m-1+\sigma_\ell(\{0\})\leq n-m-1+\left(m-n+1+\frac{1}{2}\right) =\frac{1}{2}.\]

 Thanks to Items \eqref{item: prob 6} and \eqref{item: prob 7} of Proposition \ref{prop: lifted face geometry}, the group $G_\ell$ is torsion-free and satisfies the Atiyah conjecture, so that $\beta_1^{(2)}(G_\ell)$ is an integer. We thus conclude that it must be zero. This finishes the proof.
\end{proof}
\begin{rem}
The fact that the strong Atiyah conjecture is satisfied with asymptotic probability is an essential ingredient in the above proof that allows us to reach the conclusion that $\beta_1^{(2)}(G_\ell)=0$ with asymptotic probability one. Without it, the application of Theorem \ref{thm:main spectral law} would allow us to conclude only that $\beta_1^{(2)}(G_\ell)\xrightarrow[\ell\to \infty]{\P_\ell}0$.
\end{rem}
	\subsection{Vanishing rank gradient for virtually fibered groups}\label{subsection:rank gradient criterion}
	In this section, we prove Theorem \ref{thm:low_deficiency}. The asymptotic almost-sure vanishing of the first $\ell^2$-Betti number proved in Theorem \ref{thm: vanishing first l2 betti number} implies asymptotic almost-sure virtual fiberedness by \cite[Theorem 5.3]{Kielak2020}. This, together with the criterion for vanishing of the absolute rank gradient given in Proposition \ref{prop: vanishing_rank_gradient}, yields the result.
 
 \begin{lem}\label{lem:absolute_rank_gradient}
		Let $G$ be a finitely generated residually finite group. For every $i\geq 0$, let $G_i$ be the intersection of all subgroups $H\leqslant G$ of index $[G:H]\leq i$. Then 
		\[\rg(G)=\rg(G,(G_i)_{i\geq 0}).\]
	\end{lem}
	
	\begin{proof}
		This follows from the fact that the sequence $(G_i)_{i\geq 0}$ is coinitial in the set of finite-index subgroups ordered by inclusion. See also \cite[Corollary 11]{abert2012rank} and its proof.
	\end{proof}
	
	\begin{prop}\label{prop: vanishing_rank_gradient}
		Let $G$ be a finitely generated group that fits into a short exact sequence
		\[
		1\longrightarrow N\longrightarrow G\longrightarrow Q\longrightarrow 1,\]
		where $N$ is infinite and finitely generated. Let $(G_i)_{n\geq 0}$ be a decreasing chain of finite-index normal subgroups of $G$ with trivial intersection. Suppose that
		$[G:NG_i]\xrightarrow[i\to \infty]{}\infty$.
		Then
		\[\rg\left(G,(G_i)_{i\geq 0}\right)=0.\]
		In particular, if $Q$ has infinitely many finite-index subgroups, then $\rg(G)=0$.
	\end{prop}
	\begin{proof}	
		For each $i\ge 1$ let us denote $a_i \coloneqq[N:N\cap G_i]$ and $b_i\coloneqq [G:NG_i]$. Since $G_i\triangleleft G$ is a normal subgroup, the second isomorphism theorem tells us that
		\[NG_i/G_i \cong N/(N\cap G_i),\]
		
		so that we have $[G:G_i]=[G:NG_i][NG_i:G_i]=a_i b_i$.
		
		Let	$r\coloneqq d(N)$ and $s\coloneqq d(Q)$ be the minimum number of generators of $N$ and $Q$, respectively. Note the hypotheses of the proposition imply that both of these quantities are finite: $r$ is finite since $N$ is finitely generated, and $s$ is finite since $Q$ is finitely generated, it being a quotient of $G$.
		
		The restriction of the quotient map $G\to Q=G/N$ to $G_i$ yields a short exact sequence
		\[
		1\longrightarrow N\cap G_i \longrightarrow	G_i	\longrightarrow	NG_i/N	\longrightarrow 1,
		\]
		and thus $d(G_i) \leq d(N\cap G_i)+d(NG_i/N)$. Since $N\cap G_i$ has index $a_i$ in $N$, it follows from the Reidemeister-Schreier method that
		\[
		d(N\cap G_i)-1 \leq a_i(r-1).
		\]
		
		Moreover, $NG_i/N$ is a subgroup of $Q$ of index $[Q:NG_i/N] =[G:NG_i]=b_i$,
		so again by the Reidemeister-Schreier method we have
		\[d(NG_i/N)-1 \leq b_i(s-1).\]
		As a consequence, we get
		\[
		d(G_i)-1	\leq 1+a_i(r-1)+b_i(s-1).
		\]
		Dividing by $[G:G_i]=a_ib_i$, we obtain
		\begin{equation}\label{eq: N GN}
			\frac{d(G_i)-1}{[G:G_i]}\leq\frac{1}{a_ib_i}+\frac{r-1}{b_i}+\frac{s-1}{a_i}.
		\end{equation}
		
		We have as a hypothesis of the proposition that $b_i=[G:NG_i]\xrightarrow[i\to\infty]{}\infty$. We claim that it also holds that $a_i=[N:N\cap G_i]\xrightarrow[i\to\infty]{}\infty$. Indeed, note that we have
		\[
		\bigcap_{i\geq 0}(N\cap G_i) =N\cap\bigcap_{i\geq 0}G_i=\{e_G\}.
		\]
		Since the chain $(G_i)_{i\geq 0}$ is decreasing, the sequence $(a_i)_{i\geq 0}$ is non-decreasing. If this sequence were bounded, it would eventually be constant. This would then imply that the chain $(N\cap G_i)_{i\geq 0}$ eventually stabilizes, since the subgroups $N\cap G_i$ are nested. Its intersection would therefore be a finite-index subgroup of $N$. This is not possible, since the intersection is trivial and $N$ is infinite. Hence $a_i\xrightarrow[i\to \infty]{}\infty$.
		
		From this, it follows from Inequality \eqref{eq: N GN} that
		\[
		\rg\left(G,(G_i)_{i\geq 0}\right)=0.
		\]
		
		Finally, assume $Q$ has infinitely many finite-index subgroups. Then, since $Q$ is finitely generated, it admits a sequence $(Q_i)_{i\geq 0}$ of finite-index subgroups such that $([Q:Q_i])_{i\geq 0}$ is unbounded. For each $i\geq 0$, set $q_i\coloneqq[Q:Q_i]$ and let $G_i=\bigcap_{[G:H]\leq i+1}H\leqslant G$. Then $[G:\pi^{-1}(Q_i)]=[Q:Q_i]=q_i$, so $G_{q_i}\leqslant \pi^{-1}(Q_i)$. Therefore $\pi(G_{q_i})\leqslant Q_i$, which gives
		\[[G:NG_{q_i}]=[G:\pi^{-1}(\pi(G_{q_i}))]=[Q:\pi(G_{q_i})]\geq [Q:Q_i]=q_i.\]
		We conclude that the sequence $([G:NG_{q_i}])_{i\geq 0}$ is unbounded. It follows from the first part of the statement and Lemma \ref{lem:absolute_rank_gradient} that
		\[\rg(G)=\rg(G,(G_i)_{i\geq 0})=0.\]
		
	\end{proof}
	
	\begin{proof}[Proof of Theorem \ref{thm:low_deficiency}]
	Consider a presentation in the event $\mathcal{G}_\ell$ of Proposition \ref{prop: lifted face geometry}. Thanks to Theorem \ref{thm: vanishing first l2 betti number}, we have that the event where $\beta_1^{(2)}(G_\ell)=0$ has asymptotic probability $1$, and hence its intersection with $\mathcal{G}_{\ell}$ also does.
	
	 By Items \eqref{item: prob 6} and \eqref{item: prob 7} of Proposition \ref{prop: lifted face geometry}, the group $G_\ell$ is infinite, torsion-free, residually finite, and virtually RFRS. Hence, we can use Kielak's virtual fiberedness theorem \cite[Theorem 5.3]{Kielak2020} to conclude the existence of a finite-index subgroup $H\leqslant G_\ell$ and an exact sequence
	\[
	1\longrightarrow N\longrightarrow H \longrightarrow\Z \longrightarrow1
	\]
	with $N$ finitely generated. If $N$ is infinite, Proposition \ref{prop: vanishing_rank_gradient} gives $\rg(H)=0$. If $N$ is finite, torsion-freeness implies $N=\{e\}$, so $H\cong\Z $ and again $\rg(H)=0$.
	
	For every finite-index subgroup $K\leq H$, we have
	\[
	\frac{d(K)-1}{[G_\ell:K]} =\frac{1}{[G_\ell:H]} \frac{d(K)-1}{[H:K]}.
	\]
	Taking infimum, we obtain
	\[
	0\leq\rg(G_\ell) \leq\frac{\rg(H)}{[G_\ell:H]}=0.
	\]
	The general inequalities from Equation \eqref{eq:general-inequalities} therefore imply $\beta_1^{(2)}(G_\ell) =\mathrm{cost}(G_\ell)-1 =\rg(G_\ell)=0$. This finishes the proof.
\end{proof}

	\section{The proof of Theorem \ref{thm: main theorem} in the negative Euler characteristic case}\label{section:main theorem high deficiency}
	
	In this part of the paper, we complete the proof of Theorem \ref{thm: main theorem} by establishing Theorem \ref{thm:high_deficiency} below. We point out that the proof of Theorem \ref{thm:high_deficiency} is self-contained and independent from the previous sections.
	
	\begin{customthm}{A.2}\label{thm:high_deficiency}
		Let $G_\ell=\langle x_1,\ldots,x_n\mid r_1,\ldots,r_m\rangle$, $\ell\ge 1$, be a random group with $n\geq2$ generators and $0\leq m\leq n-2$ relators in the few-relator model. Then, asymptotically almost surely as $\ell\to\infty$, we have
		\begin{equation*}\label{eq: high deficiency all equalities}
			\beta_1^{(2)}(G_\ell)
			=
			\mathrm{cost}(G_\ell)-1
			=
			\rg(G_\ell)
			=
			n-m-1.
		\end{equation*}
	\end{customthm}
	
	\subsection{Graphs and labeled graphs}
	We will use throughout the notion of directed labeled graphs and paths in them. In this subsection we gather the convention we use.
	
	A \emph{directed graph} is a tuple $\Gamma=(V(\Gamma),E(\Gamma),o,t)$ where $V(\Gamma)$ and $E(\Gamma)$ are the sets of vertices and edges,
	respectively, and $o,t:E(\Gamma)\longrightarrow V(\Gamma)$ assign to each edge its initial and terminal vertex. We allow loops and multiple edges. An edge $e$ is thus regarded as having a distinguished orientation, from $o(e)$ to $t(e)$. For $e\in E(\Gamma)$ and $\varepsilon\in\{\pm1\}$, we use the
	notation $e^\varepsilon$ for the traversal of $e$ in its given direction
	if $\varepsilon=1$, and in the opposite direction if $\varepsilon=-1$. Thus $o(e^{-1})=t(e),\qquad t(e^{-1})=o(e)$.
	
	We allow paths, however, to traverse edges in either direction. More precisely, a path of length $k\ge 0$ is a sequence $P=e_1^{\varepsilon_1}\cdots e_k^{\varepsilon_k}$ such that
	\[
	t(e_j^{\varepsilon_j})=o(e_{j+1}^{\varepsilon_{j+1}})\ \text{ for all } 1\leq j<k.
	\]
	Its initial and terminal vertices are denoted by $o(P)$ and $t(P)$. A path is \emph{closed} if $o(P)=t(P)$. Connectedness, components, spanning subgraphs, and trees are always understood with respect to the underlying unoriented graph; in particular, the direction of an edge does not restrict the direction in which it may be used in a path.
	
	Let $S=\{x_1,\ldots,x_n\}$. An $S$-\emph{labeled directed graph} is a
	directed graph $\Gamma$ together with a map $\lab:E(\Gamma)\longrightarrow S$. Given $1\le i\le n$, an edge labeled by $x_i$ will be called an \(x_i\)-edge. If $P=e_1^{\varepsilon_1}\cdots e_k^{\varepsilon_k}$ is a path on $\Gamma$, its
	\emph{label} is the word
	\[
	\lab(P)\coloneqq\lab(e_1)^{\varepsilon_1}\cdots\lab(e_k)^{\varepsilon_k}\in F(S).
	\]
	Thus, traversing an edge against its distinguished orientation contributes
	the inverse of its label. When $S$ is a generating set of a group $G$,
	we also use $\lab(P)$ for the element of $G$ represented by this word,
	when there is no risk of confusion. We will call this element the \emph{product label} of $P$.
	
	We consider right Cayley graphs. In other words, if $G$ is a group generated by $S=\{x_1,\ldots,x_n\}$, the labeled Cayley graph $\Cay(G,S)$ has vertex set $G$, and, for every $g\in G$ and $1\leq i\leq n$, it has a directed edge $(g,gx_i)$ from $g$ to $gx_i$ labeled by $x_i$. Consequently, if $P$ is a path from $g$ to $h$, then $h=g\cdot \lab(P)$ as a group element in $G$. In particular, left multiplication by an element of $G$ acts on $\Cay(G,S)$ by label-preserving and orientation-preserving graph automorphisms. The word metric associated with $S$ is the graph metric on the underlying unoriented Cayley graph.
	
	More generally, if $H\leq G$, the \emph{Schreier graph} $\Sch(H\backslash G,S)$
	of the right action of $G$ on $H\backslash G$ has vertex set $H\backslash G$, with a directed $x_i$-edge $Hg\longrightarrow Hgx_i$ for every $Hg\in H\backslash G$ and $1\leq i\leq n$.
	
	Suppose now that $\Gamma$ is an $S$-labeled graph and that $S$ generates a group $G$. For a distinguished vertex $v\in V(\Gamma)$, we define the \emph{labeled fundamental group} of $\Gamma$ at $v$ by
	\[
	\pi_1^{\lab}(\Gamma,v)\coloneqq\left\{\lab(P)\in G :P\text{ is a closed path in }\Gamma\text{ based at }v \right\}.
	\]
	Equivalently, this is the image in $G$ of the ordinary fundamental group of the underlying graph under the homomorphism $F(S)\to G$ induced by the labeling. When $\Gamma$ is a subgraph of $\Sch(H\backslash G,S)$, we always use the canonical basepoint $H$, unless stated otherwise. In this case $\pi_1^{\lab}(\Gamma,H)\leq H$ and for the full Schreier graph one has $\pi_1^{\lab}\left(\Sch(H\backslash G,S),H\right)=H$. When the basepoint is clear, we simply write $\pi_1^{\lab}(\Gamma)$, or simply $\pi_1(\Gamma)$.

	If \(\Gamma\) is finite, we use
	\[
	\deg_\Gamma^+(v)
	\coloneqq \left|\{e\in E(\Gamma):o(e)=v\}\right|
	\]
	for its out-degree. Thus
	\[\frac{1}{|V(\Gamma)|} \sum_{v\in V(\Gamma)}\deg_\Gamma^+(v) = \frac{|E(\Gamma)|}{|V(\Gamma)|}.\]
	We refer to this quantity as the \emph{average degree} of the directed
	graph. Notice that this convention counts every directed labeled edge
	once.
	
	\subsection{Cost}\label{subsection:cost}
	Let $G$ be a countable group acting on the left on a standard Borel space $X$ by Borel automorphisms. Let $\mu$ be a $G$-invariant probability measure on $X$. Denote by $E$ the associated orbit equivalence relation on $X$. That is, for each $x,y\in X$ we have $xEy$ if and only if there is $g\in G$ with $y=gx$. The equivalence relation $E$ can be viewed as a Borel subset of $X\times X$, or equivalently as a graph on $X$. A \emph{Borel subgraph} of $E$ is a directed graph on $X$ such that its set of edges is a Borel subset of $E$. Let $S\subseteq X\times X$ be a graph on $X$. A \emph{path} from $x$ to $y$ in $S$ of length $k\ge 1$ is a sequence of elements $x_0,x_1,\ldots,x_k\in X$ such that
	\begin{enumerate}
		\item $x_0=x$ and $x_k=y$, and
		
		\item $(x_i,x_{i+1})\in S$ or $(x_{i+1},x_i)\in S$ for each $i=0,1,\ldots,k-1$.
	\end{enumerate}
	We define the graph $S^k$ by saying that $(x,y)\in S^k$ if and only if there is a path of length at most $k$ from $x$ to $y$ in $S$.
	
	We say that a subgraph $S$ of $E$ is \emph{spanning} if for any $(x,y)\in E$ with $x\neq y$ there exists a path from $x$ to $y$. Equivalently, $S$ is spanning if and only if $E=\bigcup_{k\ge 0}S^k$.
	
	The \emph{edge-measure} of a Borel subgraph $S$ of $E$ is defined as
	\[
	e(S)=\int_{X}\mathrm{deg}_S(x)d\mu(x),
	\]
	where \[
	\mathrm{deg}_S(x)\coloneqq\left|\left\{ y\in X: (x,y) \in S\right\}\right|
	\]
	denotes the number of edges in $S$ with initial vertex equal to $x$. 
	
	The \emph{cost} of $E$ is defined as 
	\[
	\mathrm{cost}(E)\coloneqq \inf\left\{e(S): S \text{ is a spanning Borel subgraph of }E\right\}.
	\]
	The cost of $G$ is defined as 
	\begin{equation*}
		\mathrm{cost}(G)\coloneqq \inf\left\{\mathrm{cost}(E):\ \begin{aligned}& E\text{ is the orbit equivalence relation associated with an ergodic}\\
			&\text{and essentially free action of }G\text{ on a standard Borel space}\end{aligned}\right\}.
	\end{equation*}
	
	We say that $G$ has \emph{fixed price} if the value $\mathrm{cost}(E)$ is the same for all orbit equivalence relations of ergodic and essentially free actions of $G$ on a standard Borel space.

    \subsection{Groupoid cost}\label{subsection: grupoid cost}
	We now recall the definition of groupoid cost, as defined in \cite[\S 2]{abert2012rank}, which extends the notion of cost for non-essentially free actions.
	
	Let $G$ be a countable group that acts on the left on a standard Borel space $X$ by Borel automorphisms and with a $G$-invariant probability measure $\mu$ on $X$. Consider the product space $G\times X$, where $G$ is endowed with the discrete topology. Denote the product measure by $ \mathrm{m}$. A \emph{graphing} is a Borel subset of $G\times X$. For a graphing $\Phi$ and $g\in G$ we denote the $g$-fibre of $\Phi$ by
	\[
	\Phi_g\coloneqq \{x\in X: (g,x)\in \Phi\}.
	\]
	
	For each $x\in X$ we denote the set of neighbors of $x$ in $\Phi$ by
	\[
	N(\Phi,x)\coloneqq \{g\in G: x\in \Phi_g\},
	\]
	and the degree of $x$ in $\Phi$ by $\mathrm{deg}_{\Phi}(x)\coloneqq |N(\Phi,x)|$. Then we have
	\begin{equation*}\label{eq: m}
		\mathrm{m}(\Phi)=\sum_{g\in G}\mu(\Phi_g)=\int_{X}\mathrm{deg}_\Phi(x)d\mu(x).
	\end{equation*}
	
	Let $I$ be the graphing whose fiber for each $g\in G$ is given by $I_g\coloneqq \begin{cases}
		X, &\text{ if }g=e_G, \text{ and}\\ \varnothing, &\text{ otherwise}.
	\end{cases}$
	
	For a graphing $\Phi$, define the graphing $\Phi^{T}$ as the one with fibers $(\Phi^T)_g\coloneqq g^{-1}\cdot\Phi_{g^{-1}}$ for each $g\in G$. Define also the graphing $\overline{\Phi}\coloneqq \Phi\cup \Phi^T\cup I$.
	
	Given graphings $\Phi,\Psi$ we define their product $\Phi\cdot \Psi$ as the graphing with fibers
	\[
	(\Phi\cdot \Psi)_g\coloneqq \bigcup_{h\in G} \left(\Phi_h\cap h^{-1}\cdot(\Psi_{gh^{-1}})\right),
	\]
	for each $g\in G$. Equivalently, $(g,x) \in \Phi\cdot \Psi$ if and only if there are $g_1,g_2\in G$ with $(g_1,x)\in \Phi$, $(g_2,g_1x)\in \Psi$ and $g_2g_1=g$.
	
	Define $\Phi^1=\overline{\Phi}$ and $\Phi^{k}=\Phi^{k-1}\cup (\Phi^{k-1}\cdot \overline{\Phi})$. A graphing $\Phi$ is called an \emph{L-graphing} if $G\times X=\bigcup_{k\ge 1}\Phi^k$. The \emph{groupoid cost} of $(G,X)$ is defined as
	\[
	\mathrm{gcost}(G,X)\coloneqq \inf\left\{ \mathrm{m}(\Phi): \Phi \text{ is an }\text{L-graphing} \right\}.
	\]
	
	A graphing $\Phi$ is called finitely supported if $\Phi_g$ is empty for all but finitely many $g\in G$.

    It is proved in \cite[Lemma 6]{abert2012rank} that $\mathrm{gcost}(G, X)\ge\mathrm{cost}(E_{G\curvearrowright X})$, with equality if the action $G\curvearrowright X$ is essentially free and $G$ is finitely generated.

    \subsection{The profinite action}
    We follow the notation from \cite[Section 3]{abert2012rank}. Let $G$ be a countable residually finite group, and $S\subseteq G$ a finite generating set. The profinite completion of $G$ can be identified as
    \[\widehat{G}\coloneqq \varprojlim_n G/G_i=\left\{(w_i)_{i\geq 0}\in \prod_{i\geq 0}G/G_i:w_{i+1}\subseteq w_i\text{ for all }i\ge 0\right\},\]
    where for each $i\geq 0$, the finite-index subgroup $G_i\leqslant G$ is defined as the intersection of all normal subgroups $H\lhd G$ such that $[G:H]\leq i$. Given $w\in G/G_j$, we define
    \[\mathrm{Sh}(w)\coloneqq \{(w_i)_{i\ge 0}\in \widehat{G}:w_j= w\}.\]

    The following lemma expresses the concepts of L-graphings and edge-measure in the particular case of the profinite action defined above. Its proof is elementary and follows from the arguments in \cite[Section 3]{abert2012rank}.
	
	\begin{lem}\label{lemma:L-graphings}
		Let $G$ be a countable group and $S$ be a finite generating set. Let $H\lhd G$ be a finite-index subgroup and let $\Gamma \leq \mathrm{Cay}(G/H,S)$ be a spanning connected subgraph. Consider the graphing $\Phi\subseteq G\times \widehat{G}$ defined by setting, for each $s\in S$, the fiber 
		\(
		\Phi_{s^{-1}}=\bigcup_{w}\mathrm{Sh}(w^{-1})
		\),
		where the union ranges over all vertices $w\in G/H$ such that $(w,ws)\in E(\Gamma)$. Then:
		\begin{enumerate}
			\item $\Phi$ is an L-graphing if and only if $\pi_1(\Gamma)=H$, and
			\item \(\mathbf{m}(\Phi)=|E(\Gamma)|/|V(\Gamma)|.\)
		\end{enumerate} 
	\end{lem}
	\subsection{The minimum condition}\label{subsection: minimum condition}
	
	\begin{defn}\label{def: minimum condition} Let $n>m\ge 1$. Let $F_n$ be the free group of rank $n$ with free basis $\{x_1,\ldots,x_n\}$ and let $(r_1,\ldots, r_m)$ be an $m$-tuple of cyclically reduced words on $F_n$. Let $\phi:F_n\to \Z$ be a homomorphism. Let $r=s_0\cdots s_{L-1}$ be a cyclic word with $\phi(r)=0$. Along the boundary cycle of $r$, assign to the vertex reached after reading the prefix $s_0\cdots s_{p-1}$ the height $\phi(s_0\cdots s_{p-1})$. Since $\phi(r)=0$, this is well defined on the cyclic boundary. Extend the height linearly over the boundary edges. We call the set of points of minimum height the \emph{lower section} of $r$ and denote it by $L_\phi(r)$.

The pair $((r_1,\ldots, r_m),\phi)$ satisfies the \emph{minimum condition} if, possibly after renaming the generators $\{x_1,\ldots, x_n\}$, for each $j=1,\ldots,m$ it holds that $\phi(r_j)=0$ and that the lower section $L_{\phi}(r_j)$ consists of 
		\begin{itemize}
			\item either exactly one vertex, whose two adjacent edges are labeled by $x_j^{\pm 1}$ and $x_n^{\pm 1}$, or
			\item exactly one edge, labeled by $x_j^{\pm 1}$, whose adjacent edges are labeled by $x_n^{\pm 1}$.
		\end{itemize}
		The pair $((r_1,\ldots, r_m),\phi)$ satisfies a \emph{standard minimum condition} if there exists $\phi:F_n\to \Z$ such that, in addition to the above, satisfies $\phi(x_i)\ge 0$ for each $1\le i< n$ and $\phi(x_n)<0$.
	\end{defn}
	
	\begin{rem}
		Modulo changing generators $x_i$ to $x_i^{-1}$, every pair $((r_1,\dots,r_m),\phi)$ satisfying the minimum condition is seen to satisfy the standard minimum condition. Further, we may assume that the map $\phi$ is surjective.
	\end{rem}
	
	Let $G$ be a finitely generated group and let $S=\{x_1,\ldots, x_n\}$ be a generating set. Let $h:G\to \Z$ be a homomorphism. For each $i=1,\ldots,n$ and every $x_i$-edge $e=(g,gx_i)$ in $\mathrm{Cay}(G,S)$, let us define its \emph{lowest $h$-value} by 
	\begin{equation*}\label{eq: lowest h value}
		\ell(e)\coloneqq \min\{h(g),h(gx_i)\}.
	\end{equation*}
	
	\begin{lem}\label{lem: path lemma}
		Let $G=\langle x_1,\dots,x_n\mid r_1,\dots,r_m\rangle$ and let $\phi:F_n\to \Z$ be a surjective homomorphism such that $((r_1,\ldots,r_m),\phi)$ satisfies the standard minimum condition. Let $h:G\to \Z$ be the homomorphism induced by $\phi$, and set $S=\{x_1,\ldots,x_n\}$.
		
		Then there exist constants $C,R\in \N$, depending only on $r_1,\ldots,r_m$ and $\phi$, such that for every $1\le j\le m$ and every $x_j$-edge $e=(g,gx_j)$ in $\mathrm{Cay}(G,S)$, there is a finite path $P(e)$ joining $g$ and $gx_j$ in $\mathrm{Cay}(G,S)$ with the following properties.
		\begin{enumerate}
			\item Every $x_i$-edge $e'$ occurring in $P(e)$ satisfies either $i>m$, or $i\le m$ and
			\[
			\ell(e)<\ell(e')\le \ell(e)+C.
			\]
			\item The path $P(e)$ is contained in the $R$-neighborhood of the edge $e$ in
			$\mathrm{Cay}(G,S)$.
			
			\item The product of the labels along the path $P(e)$ is equal to $x_j$ in $G$.
		\end{enumerate}
	\end{lem}
	\begin{proof}
		For each $j=1,\ldots,m$, choose the distinguished $x_j$-edge in the $r_j$-cycle given by the standard minimum condition. After cyclically conjugating, and if necessary, inverting, $r_j$, we may write $r_j=x_j u_j$ in such a way that the $x_j$-edge at the start of $r_j$ is the distinguished edge. In other words, in the corresponding $r_j$-cycle based at $e_G$, the distinguished edge is $e_j\coloneqq (e_G,x_j)$ with label $x_j$. Note that, since word metrics are left invariant, in order to establish the statement of the lemma it is sufficient to find $C,R\in \N$ that satisfy the desired properties for all $e_j$, $1\le j \le m$.
		
		Let $Q_j$ be the complementary path from $e_G$ to $x_j$ obtained by going around the rest of the $r_j$-cycle. Its label is $u_j^{-1}$. Since $x_j u_j=r_j=e_G$ in $G$, the word $u_j^{-1}$ and $x_j$ evaluate to the same element in $G$. Hence the label of $Q_j$ represents $x_j$. By the standard minimum condition, $e_j$ meets the lower section of the $r_j$-cycle, and every other edge meeting the lower section is labeled by $x_n$ or $x_n^{-1}$. Hence, every $x_i$-edge $e'$ of $Q_j$ with $i\leq m$ has both endpoints strictly above the minimum height $\ell(e_j)$, so we have $\ell(e')>\ell(e_j)$.
		
		Define 
		\[
		C\coloneqq \max\left\{ \ell(e')-\ell(e_j) : 1\le j \le m, e'\text{ an edge in the path }Q_j\right\}.
		\]
		Hence, for every $x_i$-edge $e'$ of some path $Q_j$ with $i\le m$, we have $0<\ell(e')-\ell(e_j)\le C$. Also, let us set $R=\max_{1\le j\le m}|r_j|$, where $|r_j|$ denotes the word length of $r_j$ as a reduced word in the free group $F_n$. Note that any edge in $Q_j$ is within distance $R$ in $\mathrm{Cay}(G,S)$ of $e_j$, so the last property of the lemma is satisfied.
	\end{proof}

For the random presentations considered here, the minimum condition was proved to hold with asymptotic probability $1$ when $m<n-1$ in \cite[Proposition 5.3]{KielakKrophollerWilkes2022}.
	
	\subsection{Upper bound on the rank gradient and proof of Theorem \ref{thm:high_deficiency}}\label{subsection:upper bound on rank gradient}
	The objective of this section is to establish the following inequality.
	
	\begin{prop}\label{prop: groupoid cost}
		Let $((r_1,\ldots, r_m),\phi)$ be a pair satisfying a minimum condition and let $G=\langle x_1,\dots,x_n\mid r_1,\dots,r_m\rangle$ with $m\le n-1$. Suppose that $G$ is residually finite and denote by $\widehat{G}$ the profinite completion of $G$. Then $\mathrm{gcost}(G,\widehat{G})\leq n-m$.
	\end{prop}
	
	\begin{lem}\label{lemma:spanning_graph}
		Let $((r_1,\ldots, r_m),\phi)$ be a pair satisfying a minimum condition and let $G=\langle x_1,\dots,x_n\mid r_1,\dots,r_m\rangle$ with $m\le n-1$. Let $C,R\in \N$ be the constants given by Lemma \ref{lem: path lemma}. Let $N\ge 1$, and let $H\lhd G$ be a finite-index normal subgroup such that $H\le \ker(G\xrightarrow{h}\Z\to \Z/N\Z)$, where $h:G\to \Z$ is the homomorphism induced by $\phi$. Set $S=\{x_1,\ldots,x_n\}$, and assume that the quotient map $G\to G/H$ is injective on $\bigcup_{j=1}^m C(r_j)\cup B_S(e_G,R)$.
		
		Then the labeled Cayley graph $\mathrm{Cay}(G/H,S)$ contains a connected spanning subgraph $\Gamma$ with average degree at most $n-m+m(C+1)/N$. Moreover, the labeled fundamental group of $\Gamma$ is $H$, that is, $\pi_1(\Gamma)=H$.
	\end{lem}
	\begin{proof}
		Let us first note that, since $H\le \ker(G\xrightarrow{h}\Z\to \Z/N\Z)$, the map
		$h$ induces a homomorphism $\bar{h}:G/H\to \Z/N\Z$.
		
		For each $i=1,\ldots, n$ and every $x_i$-edge $\bar{e}=(\bar{g},\bar{g} x_i)$ in $\mathrm{Cay}(G/H,S)$, choose a	lift $e=(g,gx_i)$ in $\mathrm{Cay}(G,S)$ and define
		$\bar{\ell}(\bar{e})$ to be the image of $\ell(e)$ in $\Z/N\Z$. Note that this is well-defined. Indeed, let $e_1$ and $e_2$ be both lifts of $\bar{e}$ in $\mathrm{Cay}(G,S)$. Since $e_1$ and $e_2$ project to the same edge of $G/H$, their initial vertices differ by an element of $H$, so that $g_2^{-1}g_1\in H$.
		Hence $h(g_1)-h(g_2)=h(g_2^{-1}g_1)\in N\Z$ and therefore $\ell(e_1)\equiv \ell(e_2)\ \mathrm{mod}\ N$.
		
		If $N\le C+1$, we take $\Gamma=\mathrm{Cay}(G/H,S)$. Then $\Gamma$ is connected,
		its labeled fundamental group is $H$, and its average degree is at most $n$, which is
		bounded above by $n-m+m\frac{C+1}{N}$. Thus we may assume from now on that $N>C+1$.
		
		Let $B\coloneqq\{0,1,\ldots,C\}\subset \Z/N\Z$. Define $\Gamma$ to be the subgraph of $\mathrm{Cay}(G/H,S)$ with vertex set $G/H$ and containing as edges:
		\begin{itemize}
			\item every $x_i$-edge with $i>m$, and
			\item every $x_i$-edge with $1\le i\le m$ and $\bar\ell(\bar e)\in B$.
		\end{itemize}
		Let us first estimate the average degree $\frac{|E(\Gamma)|}{|V(\Gamma)|}$.

		For each $i>m$ all $x_i$-edges are kept, so in total they contribute $[G:H](n-m)$ edges to $\Gamma$. Let us now count the retained edges with labels $x_1,\ldots,x_m$. Fix $1\le i\le m$. Since the minimum condition gives $h(x_i)\ge 0$, an $x_i$-edge $\bar{e}=(\bar{g},\bar{g} x_i)$ is retained if and only if $\bar{\ell}(\bar{e})=\bar{h}(\bar{g})\in B.$ Since $\phi$ is surjective, the map $h:G\to \Z$ is surjective, and hence so is the map $G\xrightarrow{h}\Z\to \Z/N\Z$. Since we are assuming that $H\le \mathrm{ker}(G\xrightarrow{h}\Z\to \Z/N\Z)$, the induced map $\bar{h}:G/H\to \Z/N\Z$ in the quotient $G/H$ is also surjective. Thus, every fiber of $\bar{h}$ has the same cardinality. In other words, we have
		\[
		\left|\bar{h}^{-1}(\{a\})\right|=\frac{[G:H]}{N}
		\]
		for every $a\in \Z/N\Z$. Since we have $|B|=C+1$, then
		
		\[
		|\bar h^{-1}(B)|=\sum_{a\in B}\left|\bar{h}^{-1}(\{a\})\right|=(C+1)\frac{[G:H]}{N}.
		\]
		It follows that, for this fixed $i\le m$, the number of retained $x_i$-edges is $(C+1)\frac{[G:H]}{N}.$ Summing over $i=1,\ldots,m$, the total number of retained edges with labels in $\{x_1,\ldots,x_m\}$ is $m(C+1)\frac{[G:H]}{N}$. We conclude that \[|E(\Gamma)|=(n-m)[G:H]+m(C+1)\frac{[G:H]}{N},\]
		and therefore the average degree in $\Gamma$ is 
		\[
		\frac{|E(\Gamma)|}{|V(\Gamma)|}=n-m+m\frac{C+1}{N}.
		\]
		This proves the desired average degree estimate.

		Let us now show that for every $j\leq m$, every
		$x_j$-edge that got removed from $\mathrm{Cay}(G/H,S)$ has its endpoints connected within $\Gamma$ through a path with product label $x_j$. From this, it will follow both that $\Gamma$ is connected and that $\pi_1(\Gamma)=H$.

		For each $a\in \Z /N\Z $ let us define
		\[
		d(a)=\min\{k\ge 0: a+k\in B \ \mathrm{mod}\ N\}.
		\]
		
		In other words, $d(a)$ is the distance, moving in the positive cyclic direction, from $a$ to the retainment window $B=\{0,1,\ldots,C\}$. In particular, $d(a)=0$ if and only if $a\in B$.
		
		Let $\bar{e}$ be an $x_j$-edge with $1\le j\le m$. Let us prove by induction on the value $d(\bar{\ell}(\bar{e}))$ that the endpoints of $\bar{e}$ are connected in $\Gamma$.
		
		If $d(\bar{\ell}(\bar{e}))=0$, then $\bar{\ell}(\bar{e})\in B$, so by definition of $\Gamma$, the edge $\bar{e}$ is retained. Thus, in particular, the two endpoints of $\bar{e}$ are connected in $\Gamma$.
		
		Now let us suppose that $d(\bar\ell(\bar e))>0$, and assume the claim has already been proved for all $x_j$-edges $\bar{f}$ with $j\le m$ and $d(\bar{\ell}(\bar{f}))<d(\bar{\ell}(\bar{e}))$.
		
		Choose a lift $e=(g,gx_j)$ of $\bar{e}$ to $\mathrm{Cay}(G,S)$. By Lemma \ref{lem: path lemma}, there is a path $P(e)$ from $g$ to $gx_j$ with product label $x_j$ and such that every $x_i$-edge $e'$ of $P(e)$
		satisfies either that $i>m$, or that $i\le m$ and
		\[
		\ell(e)<\ell(e')\le \ell(e)+C.
		\]
		Project $P(e)$ to a path $\bar{P}(\bar{e})$ in $\mathrm{Cay}(G/H,S)$ joining
		the endpoints of $\bar{e}$. It is clear that the product label of $\bar{P}(\bar{e})$ is precisely $x_j$. Every $x_i$-edge in $\bar{P}(\bar e)$ with $i>m$ lies in $\Gamma$. Now consider an $x_i$-edge $\bar{f}$ of $\bar P(\bar e)$ with $i\le m$. Let $f$ be the corresponding lift occurring in $P(e)$. Then $\ell(e)<\ell(f)\le \ell(e)+C$ and hence $\ell(f)=\ell(e)+t$ for some integer $t\in \{1,\ldots, C\}$. Taking residue class modulo $N$, we get $\bar{\ell}(\bar{f})=\bar\ell(\bar{e})+t$. Since $d(\bar{\ell}(\bar {e}))>0$, we have $\bar{\ell}(\bar{e})\notin B$. Choose the representative $A\in\{C+1,C+2,\ldots,N-1\}$ of $\bar{\ell}(\bar{e})$. Then $d(\bar{\ell}(\bar{e}))=N-A.$ Since $1\le t\le C$, there are two possibilities:
		\begin{itemize}
			\item If $A+t\ge N$, then
			\[
			\bar{\ell}(\bar{f})=A+t-N\in \{0,1,\ldots,C-1\}\subseteq B.
			\]
			As a consequence, $\bar{f}$ belongs to $\Gamma$.
			
			\item If $A+t<N$, then $\bar{\ell}(\bar{f})$ is represented by $A+t$, which still lies outside $B$, and
			\[
			d(\bar{\ell}(\bar{f})) =N-(A+t)<N-A=d(\bar\ell(\bar e)).\]
			By the induction hypothesis, $\bar{f}$ has its endpoints connected in $\Gamma$ through a path with product label $x_i$.
		\end{itemize}
		In any case, we can replace the edge $\bar{f}$ by a path in $\Gamma$ with the same product label. After doing so for every $x_i$-edge $\bar{f}$ of $\bar{P}(\bar{e})$ with $i\leq m$, we obtain a path in $\Gamma$ joining $\overline{g}$ and $\overline{gx_j}$ with product label $x_j$. This proves the induction step. 
		
		The connectedness of $\Gamma$ now follows from the above argument and the fact that $x_j$-edges with $j>m$ were all retained. It also immediately follows that the labeled fundamental group of $\Gamma$ is $H$. Indeed, first note that a cycle in $\Gamma$ is a cycle in $\mathrm{Cay}(G/H,S)$, and so the product of its labels has to be equal to $e_{G/H}$. Therefore $\pi_1(\Gamma)\leqslant H$. Let us see that $H\leqslant \pi_1(\Gamma)$. Every $g\in H$ defines a cycle in $\mathrm{Cay}(G/H,S)$. By replacing the edges of this cycle that got removed in the construction of $\Gamma$ with the path in $\Gamma$ granted by the above induction argument, we can build a cycle in $\Gamma$ with product label $g$. This shows that $H\leqslant \pi_1(\Gamma)$.
	\end{proof}
	
	\begin{proof}[Proof of Proposition \ref{prop: groupoid cost}]
		For each $N\geq 1$, define
		\[
		H_N:=\bigcap_{H\lhd G,[G:H]\leq N}H.
		\]
		It is readily checked that $\widehat{G}$ can be identified with $\varprojlim_N G/ H_N$. Since $\ker(G\xrightarrow{h}\Z\to \Z/N\Z)$ has index $N$ in $G$, we get that $H_N\leqslant \ker(G\xrightarrow{h}\Z\to \Z/N\Z)$. Since $G$ is residually finite and the subgroups $H_N$ form
		a decreasing chain with trivial intersection, the quotient map $G\to G/H_N$ is injective on the fixed finite set required by Lemma \ref{lemma:spanning_graph} for all sufficiently large $N$. Letting $C,R\in \N$ be the constants given by Lemma \ref{lem: path lemma}, by Lemma \ref{lemma:spanning_graph} there is a spanning subgraph $\Gamma\leq \mathrm{Cay}(G/H_N,S)$ with $\pi_1(\Gamma)=H_N$ and 
		\[
		\frac{|E(\Gamma)|}{|V(\Gamma)|}\leq n-m+\frac{m(C+1)}{N}.
		\]
		Define the graphing $\Phi_N\subseteq G\times \widehat{G}$ by setting 
		\[
		(\Phi_N)_{x_i^{-1}}=\bigcup_{w}\mathrm{Sh}(w^{-1}),
		\]
		where the union ranges over all vertices $w\in G/H_N$ such that $(w,wx_i)\in E(\Gamma)$. Then by Lemma \ref{lemma:L-graphings} $\Phi_N$ is an L-graphing with edge measure 
		\[
		\mathbf{m}(\Phi_N)\leq n-m+\frac{m(C+1)}{N}.
		\]
		We thus conclude that
		\[
		\mathrm{gcost}(G,\widehat{G})\leq \mathbf{m}(\Phi_N)\leq n-m+\frac{m(C+1)}{N}.
		\]
		Taking $N\to \infty$ we obtain the desired result.
	\end{proof}
	
\begin{proof}[Proof of Theorem~\ref{thm:high_deficiency}]
Suppose that $1\leq m\leq n-2$, and consider the event $\mathcal{G}_\ell$. By Items \eqref{item: prob 7} and \eqref{item: prob 8} of Proposition \ref{prop: lifted face geometry}, the group $G_\ell$ is residually finite, its presentation satisfies a minimum condition, and it holds that $\beta_1^{(2)}(G_\ell)=n-m-1$.
	
	Proposition \ref{prop: groupoid cost} gives
	\[
	\mathrm{gcost}(G_\ell,\widehat G_\ell)\leq n-m.
	\]
	Since the profinite action is free, its groupoid cost equals its cost. Lemma \ref{lem:absolute_rank_gradient} and \cite[Theorem 1]{abert2012rank} then give
	\[
	\rg(G_\ell)+1 =\mathrm{cost}(G_\ell\curvearrowright\widehat G_\ell) \leq n-m.
	\]
	Together with Inequalities \eqref{eq:general-inequalities}, this yields
	\[
	n-m-1	=\beta_1^{(2)}(G_\ell)	\leq\mathrm{cost}(G_\ell)-1	\leq\rg(G_\ell)	\leq n-m-1.
	\]
	All inequalities are therefore equalities. The conclusion holds with probability tending to one because we have $\P_\ell(\mathcal{G}_\ell)\xrightarrow[\ell\to \infty]{}1$.
\end{proof}
	
\makeatletter
\nom@tempdim=\nomlabelwidth\relax
\makeatother
\input{main.nls}
	
	\bibliographystyle{alpha}
	\bibliography{biblio}
\end{document}